\documentclass[11pt]{amsart}
\usepackage{amssymb,booktabs,caption,enumitem,mathdots,mathtools,nicematrix,tikz}
\usepackage[colorlinks,allcolors=blue]{hyperref}
\usepackage[margin=1.5in]{geometry}

\makeatletter
\g@addto@macro\bfseries{\boldmath}
\makeatother

\makeatletter
\newlength{\@thlabel@width}%
\newcommand{\thmenumhspace}{\settowidth{\@thlabel@width}{\itshape1.}\sbox{\@labels}{\unhbox\@labels\hspace{\dimexpr-\leftmargin+\labelindent/5+\labelsep+\@thlabel@width-\itemindent}}}
\makeatother

\newtheorem{theorem}{Theorem}
\newtheorem{corollary}[theorem]{Corollary}
\newtheorem{lemma}[theorem]{Lemma}
\newtheorem{proposition}[theorem]{Proposition}

\newcommand{\thistheoremname}{}
\newtheorem*{genericthm*}{\thistheoremname}
\newenvironment{namedthm*}[1]
 {\renewcommand{\thistheoremname}{#1}
 \begin{genericthm*}}
 {\end{genericthm*}}

\theoremstyle{definition}
\newtheorem{definition}[theorem]{Definition}

\theoremstyle{remark}
\newtheorem{example}[theorem]{Example}
\newtheorem{remark}[theorem]{Remark}

\numberwithin{theorem}{section}

\DeclareMathOperator{\Area}{Area}
\DeclareMathOperator{\charac}{char}
\DeclareMathOperator{\GW}{GW}
\DeclareMathOperator{\order}{ord}
\DeclareMathOperator{\rank}{rank}
\DeclareMathOperator{\Res}{Res}
\DeclareMathOperator{\supp}{supp}

\newcommand{\vareps}{\varepsilon}
\newcommand\A[1]{\mathbb A^{#1}}
\newcommand{\N}{\nu}
\newcommand{\Q}{\mathbb{Q}}
\newcommand{\C}{\mathbb{C}}
\newcommand{\R}{\mathbb{R}}
\newcommand{\B}{\mathcal{B}}
\newcommand{\Z}{\mathbb{Z}}
\newcommand{\m}{{\mathfrak{m}}}
\newcommand{\Ac}{{\mathcal{A}}}
\newcommand{\eps}{\varepsilon}
\newcommand{\Npoly}{\Gamma}
\newcommand{\nuf}{\nu(f)}
\newcommand{\nug}{\nu(g)}
\newcommand{\hyp}{\mathbb{H}}
\newcommand{\E}[1]{{E}_0(\nabla{#1})}

\begin{document}

\title[Characterization of $\mathbb{A}^1$-Milnor numbers of small rank]{Characterization of $\mathbb{A}^1$-Milnor numbers\\ of small rank}

\author{C-Y. Jean Chan}
\address[C-Y.~J.~Chan]{Department of Mathematics\\ Central Michigan University, Mt. Pleasant, MI 48859\\ USA 
}
\email{chan1cj@cmich.edu}

\author{Thomas Hagedorn}
\address[T.~Hagedorn]{
Department of Mathematics and Statistics\\ The College of New Jersey\\ 2000 Pennington Road\\ Ewing, NJ 08618\\ USA
}
\email{hagedorn@tcnj.edu}

\author{Joel Louwsma}
\address[J.~Louwsma]{Department of Mathematics\\ Niagara University\\ Niagara University, NY 14109\\ USA}
\email{jlouwsma@niagara.edu}

\author{Michael Wijaya}
\address[M.~Wijaya]{Trinity School\\ New York, NY 10024\\ USA}
\email{michael.wijaya@trinityschoolnyc.org}

\begin{abstract}
The $\mathbb{A}^1$-Milnor number of a polynomial map is a quadratic enrichment of the classical Milnor number, defined as the local $\mathbb{A}^1$-Brouwer degree of the gradient. We study which elements of the Grothendieck--Witt ring of nondegenerate symmetric bilinear forms arise as $\mathbb{A}^1$-Milnor numbers. Over fields of characteristic not $2$, $3$, or~$5$, we show that every element of the Grothendieck--Witt ring of rank at most~$7$ that is a local $\mathbb{A}^1$-Brouwer degree is also an $\mathbb{A}^1$-Milnor number. Unlike for local $\mathbb{A}^1$-Brouwer degrees, however, in general not all $\mathbb{A}^1$-Milnor numbers can be obtained from two-variable polynomials. Over an arbitrary field of characteristic not $2$, $3$, $5$, or~$7$, we characterize the elements of the Grothendieck--Witt ring of rank at most~$8$ that are $\mathbb{A}^1$-Milnor numbers of two-variable polynomials. Over subfields of~$\mathbb{R}$, this implies that there are local $\mathbb{A}^1$-Brouwer degrees of ranks $4$, $6$, and~$8$ that are not the $\mathbb{A}^1$-Milnor number of a two-variable polynomial. We also prove that under a mild hypothesis one can reduce the number of variables of a polynomial while preserving the $\mathbb{A}^1$-Milnor number up to multiplication by a unit.
\end{abstract}

\maketitle

\section{Introduction}

The classical Milnor number, originally defined in~\cite{Milnor}, is an important invariant that measures the complexity of an isolated singularity of a complex algebraic hypersurface. Over~$\C$, the Milnor number of an isolated singularity on a hypersurface defined by $f=0$ is the local Brouwer degree of the gradient $\nabla f$ at that point. The celebrated $\mathbb{A}^1$-homotopy theory of Morel and Voevodsky \cite{Morel, MorelVoevodsky} has allowed the generalization of many topological invariants, including Milnor numbers and Brouwer degrees. These enriched invariants are defined over an arbitrary field~$k$ and take values in the Grothendieck--Witt ring $\GW(k)$ of nondegenerate symmetric bilinear forms over~$k$. 

Let $f$ be a polynomial in $n$~variables with $f(0)=0$, defining a map $f\colon \A{n}_k\to\A1_k$. The $\mathbb{A}^1$-Milnor number of~$f$ is defined in~\cite{KassWickelgren2019} as the local $\mathbb{A}^1$-Brouwer degree of $\nabla f$, i.e.
\[\mu_0^{\A1}(f)= \deg^{\A1}_0 (\nabla f).\]
$\mathbb{A}^1$-Brouwer degrees have recently been studied in several papers, including \cite{KassWickelgren2019,QuickStrandWilson,BrazeltonMcKeanPauli}, and $\mathbb{A}^1$-Milnor numbers have been studied in \cite[Section~6]{KassWickelgren2019}, \cite[Section~6]{PauliWickelgren}, and \cite[Section~3]{Pauli}. They have applications in $\mathbb{A}^1$-enumerative geometry through work described in \cite{Levine,KassWickelgren2021}. A \emph{Macaulay2} package for computing $\mathbb{A}^1$-Brouwer degrees has also been developed (see \cite{BorisovEtAl,Atherton}). 

This paper studies which elements of the Grothendieck--Witt ring $\GW(k)$ arise as $\mathbb{A}^1$-Milnor numbers. The analogous question for local $\A1$-Brouwer degrees was previously studied by Quick, Strand, and Wilson~\cite{QuickStrandWilson}. Their main results characterize the elements of the Grothendieck--Witt ring of rank at most~$7$ that are realized as local $\A1$-Brouwer degrees. In particular, not every element of $\GW(k)$ that is represented by a symmetric bilinear form is a local $\A1$-Brouwer degree. Since every $\A1$-Milnor number is a local $\A1$-Brouwer degree, it is natural to ask whether every local $\A1$-Brouwer degree is an $\A1$-Milnor number. In other words, does every local $\A1$-Brouwer degree arise as the local $\A1$-Brouwer degree \emph{of a gradient}? We answer this question in the affirmative for rank at most~$7$ over fields of characteristic 0 and most fields of positive characteristic by explicitly constructing three-variable polynomials whose $\A1$-Milnor numbers give all of the local $\A1$-Brouwer degrees found in~\cite{QuickStrandWilson}. 

\begin{namedthm*}{Theorem~\ref{thm:allarise2}}
Let $k$ be a field of characteristic not $2$, $3$, or~$5$. Every element of $\GW(k)$ of rank at most~$7$ that is a local $\A1$-Brouwer degree is also the $\A1$-Milnor number of a polynomial map $\A{3}_k\to \A1_k$. 
\end{namedthm*}

The main focus of this paper is on characterizing the elements of the Gro\-then\-dieck--Witt ring of rank at most~$8$ that arise as $\A1$-Milnor numbers of two-variable polynomials. The examples given in Section~\ref{sec:allarise} show that, in the cases of ranks $1$, $2$, $3$, $5$, and~$7$ and fields of characteristic not $2$, $3$, or~$5$, all elements of the Grothendieck--Witt ring that are local $\A1$-Brouwer degrees are also $\A1$-Milnor numbers of two-variable polynomials (Theorem~\ref{thm:allarise}(\ref{part:allariseeasyranks})). Subsequent sections thus restrict attention to ranks $4$, $6$, and~$8$, and the $\A1$-Milnor numbers of two-variable polynomials of these ranks are eventually characterized in Section~\ref{sec:classifyA1Milnor}. This is recorded in our main theorem.

\begin{namedthm*}{Theorem~\ref{thm:characterization}}
Let $k$ be a field of characteristic not $2$, $3$, $5$, or~$7$. The elements of $\GW(k)$ of rank at most~$8$ that are $\A1$-Milnor numbers of polynomial maps $\A2_k\to\A1_k$ are as follows. 
\begin{center}
\begin{tabular}{cc}
\toprule
Rank & $\A1$-Milnor numbers\\
\midrule
1 & $\langle b\rangle$\\
2 & $\hyp$\\
3 & $\hyp+\langle b\rangle$\\
4 & $\hyp+\langle-A, AR \rangle$\\
5 & $2\hyp+\langle b\rangle$\\
6 & $2\hyp+\langle-2,b\rangle$\\
7 & $3\hyp+\langle b\rangle$\\
8 & $3\hyp+\langle -2,b\rangle$\\
\bottomrule
\end{tabular}
\end{center}
Here $\hyp$ is the hyperbolic form, $b\in k^{\times}$, and in the rank~$4$ case the constants $A$ and~$R$ are defined by $A=2a_{2,1}^2 - 6a_{3,0}a_{1,2}$ and
\[R=3(27a_{3,0}^2a_{0,3}^2 - 18a_{3,0}a_{2,1}a_{1,2}a_{0,3} + 4a_{2,1}^3a_{0,3} + 4a_{3,0}a_{1,2}^3 - a_{2,1}^2a_{1,2}^2),\]
where $a_{0,3}\neq0$ and either $a_{2,1}$ or~$a_{1,2}$ is nonzero, with the requirement that $A$ and~$R$ are both nonzero.
\end{namedthm*}

In~\cite{QuickStrandWilson}, it is shown that all elements of $\GW(k)$ of rank at most~$7$ that arise as local $\A1$-Brouwer degrees do so from polynomial maps $\mathbb{A}_k^2\to\mathbb{A}_k^2$. It is natural to ask whether the analogous statement holds for $\A1$-Milnor numbers, i.e.\ whether Theorem~\ref{thm:allarise2} can be strengthened to require a polynomial map $f\colon\A{2}_k\to \A1_k$. However, it is a consequence of Theorem~\ref{thm:characterization} that this is not the case: Over many fields, there exist local $\A1$-Brouwer degrees and $\A1$-Milnor numbers that are not $\A1$-Milnor numbers of two-variable polynomials. 

\begin{namedthm*}{Theorem~\ref{thm:notallarise}}
Let $k$ be a subfield of~$\R$. 
\begin{enumerate}[label=\textup{(\alph*)},ref=\textup{\alph*}]
\item There exist elements of $\GW(k)$ of ranks $4$, $6$, and~$8$ that are local $\A1$-Brouwer degrees of polynomial maps $\A2_k\to\A2_k$ but are not $\A1$-Milnor numbers of polynomial maps $\A2_k\to\mathbb{A}_k^1$.
\item There exist elements of $\GW(k)$ of ranks $4$, $6$, and~$8$ that are $\A1$-Milnor numbers of polynomial maps $\A{3}_k\to\mathbb{A}_k^1$ but are not $\A1$-Milnor numbers of polynomial maps $\A2_k\to\mathbb{A}_k^1$.
\end{enumerate}
\end{namedthm*}

Work of Kass and Wickelgren~\cite{KassWickelgren2019} makes computing local $\A1$-Brouwer degrees tractable by showing that they are equal to \emph{EKL forms}, which are elements of the Grothendieck--Witt ring previously defined by Eisenbud and Levine~\cite{EisenbudLevine} and Khimshiashvili~\cite{Khimshiashvili}. Our approach to finding $\A1$-Milnor numbers follows that of \cite{KassWickelgren2019, Pauli, QuickStrandWilson} by computing $\A1$-Brouwer degrees as EKL forms on a local algebra. This differs from the approach to finding $\A1$-Brouwer degrees taken in~\cite{BrazeltonMcKeanPauli,BorisovEtAl}, which, instead of localization, uses saturation and computes B\'{e}zoutian bilinear forms. 

Our approach to characterizing $\A1$-Milnor numbers of two-variable polynomials uses the Newton diagram of a polynomial, which has previously been used to study classical Milnor numbers by a number of authors; see~\cite{GreuelLossenShustin} and the references therein. A result of Quick, Strand, and Wilson \cite[Lemma~2.9]{QuickStrandWilson} allows us to restrict attention to polynomials that are convenient in the sense that they include terms $x^i$ and~$y^j$ for sufficiently large $i$ and~$j$. A fundamental result of Kushnirenko\footnote{We use the preferred transliteration of Kushnirenko's name as indicated in~\cite{Kushnirenko-correspondence}.} \cite[Th\'eor\`eme~I]{Kushnirenko} shows that the rank of the $\A1$-Milnor number of a polynomial is bounded below by its Newton number, thus allowing us to restrict attention to polynomials with Newton number at most~$8$. Section~\ref{sec:Newtonbackground} gives background about Newton diagrams and Newton numbers. Section~\ref{sec:classifynewton} then catalogs the Newton diagrams that have Newton number at most~$8$. 

In Section~\ref{sec:change}, we study the behavior of $\A1$-Milnor numbers under certain changes of variables. Theorem~\ref{thm:rightequivalence} proves, with mild hypotheses, that $\A1$-Milnor numbers are invariant under right equivalence. Proposition~\ref{prop:increaseNewton} proves that, under certain hypotheses, a two-variable Newton degenerate polynomial has the same $\A1$-Milnor number as one with larger Newton number. As a consequence, Theorem~\ref{thm:increaseNewton} shows that, if a two-variable polynomial is Newton degenerate and has $\A1$-Milnor number of rank at most~$8$, then it has the same $\A1$-Milnor number as a polynomial that is Newton nondegenerate. It follows that all $\A1$-Milnor numbers of rank $r$ at most~$8$ arise from polynomials with Newton number~$r$.

Section~\ref{sec:classifyA1Milnor} directly computes the $\A1$-Milnor number of each Newton nondegenerate polynomial with Newton number $4$, $6$, or~$8$, thus finding the elements of the Grothendieck--Witt ring of rank $4$, $6$, and~$8$ that are $\A1$-Milnor numbers of two-variable polynomials and completing the characterization given in Theorem~\ref{thm:characterization}. 

Section~\ref{sec:specificfields} considers the special case when $k$ is a subfield of~$\R$. Theorem~\ref{thm:possiblesignatures} shows that there is no $\A1$-Milnor number of a two-variable polynomial with rank $4$, $6$, or~$8$ and signature~$2$; as there are local $\A1$-Brouwer degrees with these invariants, this implies Theorem~\ref{thm:notallarise}. The rest of Section~\ref{sec:specificfields} characterizes the elements of $\GW(\Q)$ of rank $6$ and~$8$ that arise as $\A1$-Milnor numbers in terms of their invariants, as given in Corollaries \ref{cor:rank6invariants} and~\ref{cor:rank8invariants}.

In Section~\ref{sec:variablereduction}, we prove a reduction theorem that may be of independent interest to some readers. This follows from strategically applying Theorem~\ref{thm:rightequivalence} to show, under mild hypotheses, that one can reduce the number of variables of a polynomial while only changing the $\A1$-Milnor number by multiplication by a unit. 

\begin{namedthm*}{Theorem~\ref{thm:variablereduction}}
Let $k$ be a field of characteristic not~$2$, let $f\colon\mathbb A_k^n \to \mathbb A_k^1$ be a polynomial map with $f(0)=0$ and $1\leq\mu_0(f)<\infty$, and let $\mathfrak{m}=(x_1, \dots, x_{n})$. If $(f_{x_1}, \dots, f_{x_n})\nsubseteq \mathfrak{m}^2$, then there exists a map $g\colon \mathbb A_k^{n-1} \to \mathbb A_k^1$ with $g(0)=0$ and $1\leq\mu_0(g)<\infty$ such that 
\begin{enumerate}[label=\textup{(\alph*)},ref=\textup{\alph*}]\item $Q_0(\nabla f)\cong Q_0(\nabla g)$,
\item $\mu_0(f)=\mu_0(g)$, and 
\item $\mu^{\mathbb{A}^1}_0(f)$ and $\mu^{\mathbb{A}^1}_0(g)$ differ only by multiplication by a unit.
\end{enumerate}
\end{namedthm*}

This is the analogue for $\A1$-Milnor numbers of a reduction theorem for local $\A1$-Brouwer degrees shown by Quick, Strand, and Wilson \cite[Theorem~2.12]{QuickStrandWilson}. Similar reduction techniques were also used in McKean's calculation of the local $\A1$-Brouwer degree of the intersection of two plane curves \cite[Lemma~5.7]{McKean}.

\section{Preliminaries}\label{sec:preliminaries}

In this section, we recall some facts about Milnor numbers and symmetric bilinear forms and prove some basic properties of $\A1$-Milnor numbers. (See also Ormsby's notes~\cite{Ormsby} for a good introduction.)

\subsection{Isolated singularities and Milnor numbers}
Isolated singularities and Milnor numbers can be studied from both algebraic and topological perspectives. Milnor's original perspective was topological, but our approach is primarily algebraic. Let $k$ be a field, and let $\mathbb{A}_k^n$ denote affine $n$-space over~$k$. We freely move between polynomials $f\in k[x_1,\dots,x_n]$ and polynomial maps $f\colon\A{n}_k\to\A1_k$ with no distinction. We use $f_{x_i}$ to denote the partial derivative $\frac{\partial f}{\partial x_i}$ when the subscripts do not cause confusion. Let $\mathfrak m=(x_1,\dots,x_n)$ be the maximal ideal of $k[x_1,\dots,x_n]$. 

We recall several standard definitions; see, for example, \cite[Chapter~4]{CoxLittleOShea} and \cite[Section~3.1.1]{GreuelLossenShustin} for further details. Note that these concepts can be equivalently defined using localization of polynomial rings and formal power series rings (see Remark~\ref{rem:localizationpowerseries}). We use localization throughout. A polynomial map $f\colon\A{n}_k\to\A1_k$ has a \emph{singularity} at the origin if $f(0)=0$ and $f_{x_i}(0)=0$ for all $i\in[n]$. The \emph{Tjurina number} of~$f$ at the origin is defined as
\[\tau_0(f)\coloneqq\dim_k k[x_1, \dots, x_n]_{\mathfrak m}/(f, f_{x_1},\dots,f_{x_n});\]
when $k[x_1, \dots, x_n]_{\mathfrak m}/(f, f_{x_1},\dots,f_{x_n})$ is an infinite-dimensional vector space over~$k$, we say $\tau_0(f)=\infty$. Since the units of $k[x_1, \dots, x_n]_{\mathfrak m}$ are represented by the polynomials with nonzero constant term, $\tau_0(f)=0$ if and only if one of $f,f_{x_1},\dots,f_{x_n}$ has a nonzero constant term. Therefore $f$ has a singularity at the origin exactly when $\tau_0(f)\geq 1$. We say $f$ has an \emph{isolated singularity} at the origin if $1\leq\tau_0(f)<\infty$. 

The \emph{Milnor number} of~$f$ at the origin is defined as
\[\mu_0(f)\coloneqq\dim_k k[x_1, \dots, x_n]_{\mathfrak m}/(f_{x_1},\dots,f_{x_n}).\]
When $k[x_1, \dots, x_n]_{\mathfrak m}/( f_{x_1},\dots,f_{x_n})$ is an infinite-dimensional vector space over~$k$, we say $\mu_0(f)=\infty$. Since $\mu_0(f)\geq1$ exactly when $f_{x_i}(0)=0$ for all $i\in[n]$, we have that $f$ has a singularity at the origin if and only if $f(0)=0$ and $\mu_0(f)\geq1$. Since $\tau_0(f)\leq\mu_0(f)$, if $f$ has a singularity at the origin and $\mu_0(f)<\infty$, then $f$ has an isolated singularity at the origin. Therefore $f$ having an isolated singularity at the origin with finite Milnor number is equivalent to the conditions $f(0)=0$ and $1\leq\mu_0(f)<\infty$. 

We say $F\colon\mathbb{A}_k^n\to\mathbb{A}_k^n$ has a \emph{zero} at the origin if $\dim_k k[x_1,\dots,x_n]_{\mathfrak{m}}/(F_{1},\dots,F_{n})\geq1$
and that it has an \emph{isolated zero} at the origin if 
\[1\leq\dim_k k[x_1,\dots,x_n]_{\mathfrak{m}}/(F_{1},\dots,F_{n})<\infty.\]
From this perspective, it is clear that an~$f$ with $f(0)=0$ has an isolated singularity with finite Milnor number at the origin exactly when $\nabla f$ has an isolated zero at the origin. Over~$\C$, the Milnor number of a polynomial map $f\colon\A{n}_\C\to\A1_\C$ coincides with the local Brouwer degree of $\nabla f$. 

Over a field of characteristic~$0$, if $f$ has an isolated singularity at the origin, then $\mu_0(f)$ must be finite. Milnor gave a proof of this over~$\C$ using a topological argument \cite[Theorem~7.1]{Milnor} and indicated an algebraic approach to showing it over an arbitrary field of characteristic~$0$ \cite[Problem~3, p.~115]{Milnor}. The latter result also follows from \cite[Corollary~7.1.4 and Theorem~7.1.5]{HunekeSwanson} (see also \cite[Sections 2 and~3]{HefezRodriuesSalomao}). Over fields of positive characteristic, on the other hand, there are polynomial maps with an isolated singularity at the origin whose Milnor number is infinite. For example, if $\charac k=3$, then the polynomial $f = x^2y + xy^2$ has an isolated singularity at the origin and infinite Milnor number (see \cite[Example~1.3]{HefezRodriuesSalomao}). Under an additional nondegeneracy condition, to be discussed in Section~\ref{sec:Newtonbackground}, the Milnor number can be shown to be finite if and only if the Tjurina number is finite (see Remark~\ref{rem:nondegenerate}).

\subsection{Symmetric bilinear forms and the Grothendieck--Witt ring}\label{subsec:GWring}
In this section, we recall some basic facts about symmetric bilinear forms; see~\cite{Lam} for further details. Let $k$ be a field, and let $V$ be a vector space over~$k$. A symmetric bilinear form on~$V$ is a map $\beta\colon V \times V\to k$ that is $k$-linear in both arguments and satisfies $\beta(u,v)=\beta(v,u)$ for all $u,v\in V$. Two symmetric bilinear forms $\beta_1$ on~$V_1$ and $\beta_2$ on~$V_2$ are $\emph{isometric}$ if there is a linear isomorphism $\Psi\colon V_1 \to V_2$ such that $\beta_2(\Psi(u), \Psi(v)) = \beta_1(u, v)$ for all $u, v \in V_1$. The \emph{Gram matrix} of a symmetric bilinear form~$\beta$ on~$V$ with respect to a basis $\{v_1,\dots,v_r\}$ of~$V$ is the $r\times r$ matrix~$G$ whose $(i,j)$ entry is given by $\beta(v_i,v_j)$; a symmetric bilinear form has a symmetric Gram matrix. We let $\langle b_1, \dots, b_r \rangle$ denote the symmetric bilinear form whose Gram matrix is diagonal with diagonal entries $b_1, \dots, b_r$. If $G_1$ and~$G_2$ are the Gram matrices of symmetric bilinear forms $\beta_1$ and~$\beta_2$, respectively, then $\beta_1$ and~$\beta_2$ are isometric if and only if there exists a nonsingular matrix~$P$ such that $G_1 = P^T G_2 P$. 

When $\charac k\neq 2$, symmetric bilinear forms over~$k$ are equivalent to quadratic forms. More precisely, if $G$ is the Gram matrix of a symmetric bilinear form~$\beta$, then the corresponding quadratic form is given by $q({\mathbf x})=\mathbf x^{T} G \mathbf x$, where $\mathbf x = (x_1, \dots, x_r)^T$ is a column vector of variables. Conversely, given a quadratic form $q(\mathbf{x})$, there is a unique symmetric matrix~$G$ such that $q({\mathbf x})=\mathbf{x}^T G \mathbf{x}$, and this~$G$ determines the corresponding symmetric bilinear form. For example, the Gram matrix of the quadratic form $x^2 - y^2$ is $\begin{bsmallmatrix} 1 & 0 \\ 0 & -1 \end{bsmallmatrix}$ and that of $2xy$ is $\begin{bsmallmatrix} 0 & 1 \\ 1 & 0 \end{bsmallmatrix}$. When $\charac k\neq 2$, a standard change of basis shows that these two forms are isometric. Moreover, they are also both isometric to $\langle b,-b\rangle$ and to the form represented by $\begin{bsmallmatrix} 0 & b \\ b & 0 \end{bsmallmatrix}$ for all $b\in k^{\times}$. 

A symmetric bilinear form~$\beta$ is \emph{nondegenerate} if the only $u\in V$ for which $\beta(u,v)=0$ for all $v\in V$ is $u=0$. Equivalently, $\beta$ is nondegenerate if its Gram matrix with respect to any basis is nonsingular. The \emph{Grothendieck--Witt ring} $\GW(k)$ of a field~$k$ is the group completion of the set of isometry classes of nondegenerate symmetric bilinear forms over~$k$. Let $[\beta]$ denote the isometry class in $\GW(k)$ represented by a nondegenerate symmetric bilinear form~$\beta$. Every element of $\GW(k)$ can be represented as a formal difference $[\beta_1]-[\beta_2]$, where $\beta_1$ and~$\beta_2$ are nondegenerate symmetric bilinear forms. In $\GW(k)$, addition is given by direct sum and multiplication is given by tensor product; these operations correspond to standard direct sums and tensor products of symmetric bilinear forms and their Gram matrices. If $\beta_1$ and~$\beta_2$ are nondegenerate symmetric bilinear forms over~$k$ and $c, b_i\in k^{\times}$, we have
\begin{equation}\label{eq:GWoperations}
\begin{aligned}
[\beta_1]+ [\beta_2]&=
[\beta_1\oplus \beta_2]\\
[\beta_1]\cdot [\beta_2]&=
[\beta_1\otimes \beta_2]\\
\langle c \rangle \cdot 
\langle b_1, b_2, \dots, b_n \rangle
&=
\langle c b_1, c b_2, \dots, 
c b_n \rangle.\\
\end{aligned}
\end{equation}
In the following, we will let $\beta$ denote both a nondegenerate symmetric bilinear form and the class it represents in $\GW(k)$.

The \emph{rank} of a symmetric bilinear form~$\beta$ is the rank of its Gram matrix with respect to any basis. Given $\beta=[\beta_1]- [\beta_2]\in \GW(k)$, defining $\rank \beta=\rank \beta_1-\rank \beta_2$ gives a well-defined notion of rank on $\GW(k)$. 

When $\charac k\neq 2$, every symmetric bilinear form over~$k$ is diagonalizable. Therefore, the structure of $\GW(k)$ can be understood as a $k$-algebra generated by the forms $\langle b \rangle$ for all $b \in k^{\times}$. It satisfies the following properties (see \cite[Section~II.4]{Lam}): 
\begin{enumerate}[label=\textup{(\roman*)},ref=\textup{\roman*}]
\item $\langle b \rangle = \langle b c^2\rangle$ for all $b, c\in k^{\times}$,
\item $\langle bc \rangle = \langle b \rangle \cdot\langle c \rangle$ for all $b, c \in k^{\times}$, and 
\item $\langle b \rangle + \langle c \rangle = \langle b + c \rangle + \langle bc (b + c) \rangle$ for all $b, c \in k^{\times}$ with $b+c\neq0$.
\end{enumerate}
When $\charac k\neq 2$, the above examples $\langle b,-b \rangle$ and $\begin{bsmallmatrix} 0 & b \\ b & 0 \end{bsmallmatrix}$ define the same element of $\GW(k)$, called the \emph{hyperbolic form}, which we denote by~$\hyp$. In this paper, we identify $\hyp$ with all of these forms and use them interchangeably.

When $k$ is a subfield of~$\R$, the \emph{signature}~$s_{\beta}$ of a nondegenerate symmetric bilinear form~$\beta$ over~$k$ that is isometric to $\langle b_1, \dots, b_r\rangle$ is the difference $p-q$, where $p$ (resp.~$q$) is the number of positive (resp.\ negative) entries among $b_1, \dots, b_r$. Given $\beta=[\beta_1]-[\beta_2]\in\GW(k)$, defining $s_{\beta}= s_{\beta_1}-s_{\beta_2}$ gives a well-defined notion of signature on $\GW(k)$. 

We now recall some results about symmetric bilinear forms that will be used in computations throughout the paper. Given a symmetric matrix~$P$, let $\beta_P$ denote the corresponding symmetric bilinear form. 

\begin{lemma}\label{lemma:antidiagonal}
Let $k$ be a field of characteristic not~$2$. If $\beta$ is a symmetric bilinear form whose Gram matrix is the upper antitriangular block matrix
\[\begin{bNiceArray}{c|ccc|c}
c_{1,1} & c_{1,2} & \cdots & c_{1,r-1} & b\\
\hline
c_{1,2} & & & & 0\\
\vdots & & P & & \vdots\\
c_{1,r-1} & & & & 0\\
\hline
b & 0 & \cdots & 0 & 0
\end{bNiceArray},\]
where $b\in k^{\times}$ and $P$ is an $(r-2)\times (r-2)$ matrix, then $\beta\cong \hyp + \beta_P$.
\end{lemma}

Note that Lemma~\ref{lemma:antidiagonal} is similar to \cite[Lemma~4.5]{BrazeltonMcKean} and \cite[Proposition~5.2]{Brazelton}, but with different hypotheses, so we give a proof. 

\begin{proof}
For all~$j$ satisfying $2\leq j\leq r-1$, subtract $c_{1,j}/b$ times row~$r$ from row~$j$ and $c_{1,j}/b$ times column~$r$ from column~$j$. Then subtract $c_{1,1}/(2b)$ times row~$r$ from row~$1$ and $c_{1,1}/(2b)$ times column~$r$ from column~$1$. This gives the matrix 
\[\begin{bNiceArray}{c|ccc|c}
0 & 0 & \cdots & 0 & b\\
\hline
0 & & & & 0\\
\vdots & & P & & \vdots\\
0 & & & & 0\\
\hline
b & 0 & \cdots & 0 & 0
\end{bNiceArray}.\]
Reordering rows and columns, this becomes
\[\begin{bNiceArray}{cc|ccc}
0 & b & 0 & \cdots & 0\\
b & 0 & 0 & \cdots & 0\\
\hline
0 & 0 & & \\
\vdots & \vdots & & P & \\
0 & 0 & 
\end{bNiceArray},\]
which is the Gram matrix of $\hyp + \beta_P$.
\end{proof}

Repeatedly applying Lemma~\ref{lemma:antidiagonal} yields the following, which was also proven in \cite[Lemma~6]{KassWickelgren2020}.

\begin{proposition}\label{prop:antidiagonal}
Let $k$ be a field of characteristic not~$2$. If $P$ is the upper anti-triangular matrix 
\[\begin{bmatrix}
* & \cdots & * & b\\
\vdots & \iddots & \iddots & 0\\
* & \iddots & \iddots & \vdots\\
b & 0 & \cdots & 0
\end{bmatrix}\]
with a constant entry $b\in k^{\times}$ along the antidiagonal, then 
\[\beta_P\cong \begin{cases}
\frac{r}{2}\hyp & \text{if }r\text{ is even,} \\
\frac{r-1}{2}\hyp+\langle b\rangle & \text{if }r\text{ is odd.}
\end{cases}\]
\end{proposition}

\begin{proof}
Repeatedly apply Lemma~\ref{lemma:antidiagonal}.
\end{proof}

\subsection{\texorpdfstring{$\A1$}{A\textonesuperior}-Milnor numbers and EKL forms}\label{subsec:A1Milnor}
$\A1$-homotopy theory provides quadratic enrichments of Brouwer degrees and Milnor numbers. More precisely, Morel \cite[Corollary~1.24]{Morel} constructs an element of the Grothendieck--Witt ring $\GW(k)$ associated to a polynomial map $F\colon\mathbb{A}^n_{k}\to\mathbb{A}^n_{k}$ with an isolated zero at the origin. This element is known as the \emph{local $\A1$-Brouwer degree} of~$F$, denoted $\deg_0^{\A1}(F)$. Suppose $f\in k[x_1,\dots,x_n]$ satisfies $f(0)=0$ and $1\leq\mu_0(f)<\infty$. The \emph{$\A1$-Milnor number} $\mu_0^{\A1}(f)$ of~$f$ at the origin is then defined in~\cite{KassWickelgren2019} as the local $\A1$-Brouwer degree of $\nabla f$, i.e.
\[\mu_0^{\A1}(f) = \deg_0^{\A1}(\nabla f).\]
The Milnor number $\mu_0(f)$ can be recovered as the rank of $\mu_0^{\A1}(f)$.

Over~$\R$, Eisenbud and Levine~\cite{EisenbudLevine} computed the original local Brouwer degree differently. Let $C^{\infty}_0(\R^n)$ be the local ring of germs at the origin of smooth real-valued functions on~$\R^n$. Let $F\colon\R^n \to \R^n$ be a smooth function with an isolated zero at the origin, let $F_0\in C^{\infty}_0(\R^n)$ be the germ of~$F$ at the origin, and let $Q_0(F)$ be the local ring $C^{\infty}_0(\R^n)/(F_0)$. In~\cite{EisenbudLevine}, $\deg(F)$ is computed using a certain nondegenerate symmetric bilinear form $\omega_F$ on $Q_0(F)$. Khimshiashvili~\cite{Khimshiashvili} gave an independent construction of this form, which is now called an \emph{EKL form}. By \cite[Theorem~1.2]{EisenbudLevine}, the local Brouwer degree $\deg_0 (F)$ equals the signature of~$\omega_F$.

As observed by Eisenbud \cite{Eisenbud}, the construction of~$\omega_F$ from \cite{EisenbudLevine, Khimshiashvili} remains valid when $F$ is a polynomial map and $\R$ is replaced by an arbitrary field of characteristic not~$2$. Eisenbud asked whether $\omega_F$ can be identified with some topological invariant. This question was answered in the affirmative by Kass and Wickelgren~\cite{KassWickelgren2019}, who showed that $\omega_F$ represents the same element of the Grothendieck--Witt ring as the local $\A1$-Brouwer degree $\deg_0^{\A{1}}(F)$. This result was later extended to polynomial maps with an isolated zero at a non-rational point~\cite{BBMMO}. This provides a tractable way to compute local $\A1$-Brouwer degrees and hence $\A1$-Milnor numbers. 

We next explain how to compute EKL forms in order to find local $\A1$-Brouwer degrees; for more details, see \cite[Section~1]{KassWickelgren2019}. Let $k$ be a field of characteristic not~$2$, and let $F\colon \A{n}_{k} \to \A{n}_{k}$ be a polynomial map with an isolated zero at the origin. Define 
\[Q_0(F)\coloneqq k[x_1, \dots, x_n]_{\mathfrak m}/ (F_1, \dots, F_n),\]
the quotient of the polynomial ring localized at the maximal ideal $\mathfrak m = (x_1, \dots, x_n)$ by the ideal generated by the coordinate functions~$F_i$ of~$F$. Let $J_F$ be the Jacobian matrix of~$F$. In the case when $\dim_k Q_0(F)$ is not divisible by $\charac k$, let $\eta \colon Q_0(F) \to k$ be a linear functional with the property that $\eta(\det J_F)=\dim_k Q_0(F)$. One can then define a nondegenerate symmetric bilinear form~$\omega_F$ on $Q_0(F)$ by $\omega_F( u,v) \coloneqq \eta(uv)$. The form $\omega_F$ determines an element of $\GW(k)$, called the \emph{EKL form} of~$F$, that is independent of the choice of~$\eta$ (see \cite{EisenbudLevine, Khimshiashvili, KassWickelgren2019}). For a discussion of the case when $\charac k=2$, see \cite[Remark~46]{KassWickelgren2019} and \cite[Lemma~48]{KassWickelgren2016}. 

The above definition only works when $\charac k$ does not divide $\dim_k Q_0(F)$, but it can be extended using the \emph{distinguished socle element} $E_0(F)$ of $Q_0(F)$. A procedure for finding $E_0(F)$ is given in~\cite{KassWickelgren2019}. First, find polynomials $F_{i,j}$ such that 
\[ F_i= \sum_{j=1}^n x_j F_{i,j}.\]
Then $E_0(F)=\det(F_{i,j})$ is a nonzero element of $Q_0(F)$ that is independent of the choice of the $F_{i,j}$. Define $\eta \colon Q_0(F) \to k$ such that $\eta(E_0(F))=1$. As above, this~$\eta$ determines a nondegenerate symmetric bilinear form~$\omega_F$ that gives an element of $\GW(k)$ that is the EKL form of~$F$. The two constructions of~$\omega_F$ are consistent as Scheja and Storch \cite[(4.7)~Korollar]{SchejaStorch} show that $\det J_F=\dim_k Q_0(F)\cdot E_0(F)$ in $Q_0(F)$. 

We now give an explicit example that illustrates the use of the EKL form to compute an $\A1$-Milnor number.

\begin{example}\label{ex:running}
Let $k$ be a field of characteristic not $2$ or~$5$, and let $f = (x+y^{2})^{2} + xy^{3}$. The Jacobian ideal of~$f$ is given by 
\[ J_{f} = (f_x, f_y)= (2x + 2y^{2} + y^{3}, 4xy + 3xy^{2} + 4y^{3}).\] 

We find a basis for $Q_{0}(\nabla f)= k[x,y]_{(x,y)}/ \overline{J}_f$, where~$\overline{J}_f$ is the ideal generated by the Jacobian ideal~$J_f$ in the localization. Using $f_x=0$, we substitute for~$x$ in $f_y=0$ to obtain
\[\overline{J}_{f} = (2x+ 2y^{2} + y^{3}, -y^{4}(10+3y)/2) = (2x + 2y^{2} + y^{3}, y^{4})\]
as an ideal in $k[x,y]_{(x,y)}$, as $10+3y$ is a unit since $\charac k$ is not $2$ or~$5$. Then, in $Q_0(\nabla f)$, $x$ can be written in terms of~$y$. As $y^4=0$, 
\[Q_{0}(\nabla f)= k[y^2+y^3/2, y]_{(y)}/(y^4)\cong k[y]/(y^4).\]
Thus $\B=\{1, y, y^2, y^3\}$ is a basis for $Q_0(\nabla f)$ and $\mu_0(f)=\dim Q_0(\nabla f)=4$. Using the basis~$\B$, we now calculate the $\A1$-Milnor number. The determinant of the Hessian matrix~$H_f$ of~$f$ (i.e.\ the Jacobian of the gradient) is 
\[\det H_f = 8x + 12xy + 8y^{2} - 24y^{3} - 9y^{4}=-40 y^3\]
in $Q_{0}(\nabla f)$ as $y^{4} = 0$ and $x = -(2y^{2} + y^{3})/2$. Since $\mu_0(f) = 4$, we define $\eta(y^{3}) = -1/10$ and let $\eta(v)=0$ for all other $v\in\B$. Then $\eta(-40y^{3}) = 4$, as needed for the construction of $\mu^{\A1}_0(f)$. The Gram matrix of $\mu^{\A1}_{0}(f)$ is thus
\[\begin{bmatrix}
 \eta(1 \cdot 1) & \eta(1 \cdot y) & \eta(1 \cdot y^{2}) & \eta(1 \cdot y^{3}) \\
 \eta(y \cdot 1) & \eta(y \cdot y) & \eta(y \cdot y^{2}) & \eta(y \cdot y^{3}) \\
 \eta(y^{2} \cdot 1) & \eta(y^{2} \cdot y) & \eta(y^{2} \cdot y^{2}) & \eta(y^{2} \cdot y^{3}) \\
 \eta(y^{3} \cdot 1) & \eta(y^{3} \cdot y) & \eta(y^{3} \cdot y^{2}) & \eta(y^{3} \cdot y^{3}) \\
\end{bmatrix} =
\begin{bmatrix}
 0 & 0 & 0 & -1/10 \\
 0 & 0 & -1/10 & 0 \\
 0 & -1/10 & 0 & 0 \\
 -1/10 & 0 & 0 & 0 \\
\end{bmatrix}.\]
By Proposition~\ref{prop:antidiagonal}, we have $\mu^{\A1}_{0}(f) = 2\hyp$.
\end{example}

\subsection{How operations affect \texorpdfstring{$\A1$}{A\textonesuperior}-Milnor numbers}
We conclude this section with some results that show how the $\A1$-Milnor number is affected by multiplying the polynomial by a constant or adding terms in other variables.

\begin{proposition}\label{prop:constantmultiple}
Let $k$ be a field of characteristic not~$2$, and let $f\colon\mathbb{A}_k^n\to \mathbb{A}_k^1$ be a polynomial map with $f(0)=0$ and $1\leq\mu_0(f)<\infty$. If $b\in k^{\times}$, then 
\[\mu^{\mathbb{A}^1}_0(bf)=\begin{cases}
\mu^{\mathbb{A}^1}_0(f) & \text{if $n$ is even,}\\
\langle b\rangle\cdot\mu^{\mathbb{A}^1}_0(f) & \text{if $n$ is odd.}
\end{cases}\]
\end{proposition}

\begin{proof}
Let $g=bf$. Then 
\[J_g=(g_{x_1},\dots,g_{x_n})=(bf_{x_1},\dots,bf_{x_n})=(f_{x_1},\dots,f_{x_n})=J_f,\]
so $Q_0(\nabla g)=Q_0(\nabla f)$. The procedure for finding the distinguished socle element given in Section~\ref{subsec:A1Milnor} shows that $\E{g} =b^{n}\E{f}$. Defining $\eta_g=\frac{1}{b^n}\eta_f$, we then have 
\[\eta_g(\E{g})=\frac{1}{b^n}\eta_f(b^n\E{f})=\eta_f(\E{f})=1.\]
The Gram matrix of the EKL form of~$g$ is thus $1/b^n$ times the one for~$f$, so multiplication as given in~\eqref{eq:GWoperations} and the fact that $\langle b^2\rangle=\langle1\rangle$ yields
\[\mu^{\mathbb{A}^1}_0(g)=\Bigl\langle\frac{1}{b^n}\Bigr\rangle\cdot\mu^{\mathbb{A}^1}_0(f)=\begin{cases}
\mu^{\mathbb{A}^1}_0(f) & \text{if $n$ is even,}\\
\langle b\rangle\cdot\mu^{\mathbb{A}^1}_0(f) & \text{if $n$ is odd.}
\end{cases}\qedhere\]
\end{proof}

\begin{remark}
Although $V(f)$ and $V(bf)$ are the same variety, Proposition~\ref{prop:constantmultiple} says that $\mu_0^{\A1}(f)$ and $\mu_0^{\A1}(bf)$ can be different when $n$ is odd, as observed by Pauli and Wickelgren \cite[Section~6.2]{PauliWickelgren}. This is why we speak about the $\A1$-Milnor number of a polynomial map rather than the $\A1$-Milnor number of an isolated singularity on a hypersurface. 
\end{remark}

We next prove a result about what happens when adding polynomials in distinct sets of variables. If $f\colon\A{n_1}_k\to\A{1}_k$ and $g\colon\A{n_2}_k\to\A{1}_k$ are polynomial maps, one can think of $f+g$ as giving a polynomial map $\A{n_1+n_2}_k\to\A{1}_k$ by setting $(f+g)(\mathbf{x}, \mathbf{y})= f(\mathbf{x})+g(\mathbf{y})$. As $J_{f+g}=J_f+J_g$ as ideals in $k[x_1, \dots, x_{n_1},y_1, \dots, y_{n_2}]$, we have that
\[k[x_1, \dots, x_{n_1},y_1, \dots, y_{n_2}]/J_{f+g}\cong k[x_1,\dots,x_{n_1}]/J_f\otimes k[y_1,\dots,y_{n_2}]/J_g.\]
Localizing then gives $Q_0(\nabla(f+g))\cong Q_0(\nabla f)\otimes Q_0(\nabla g)$, so $\mu_0(f+g)=\mu_0(f)\mu_0(g)$. In particular, $\mu^{\mathbb{A}^{1}}_0(f+g)$ is well defined when $\mu^{\mathbb{A}^{1}}_0(f)$ and $\mu^{\mathbb{A}^{1}}_0(g)$ both are. 

Straightforward computations show that under the above isomorphism the distinguished socle element $\E{(f+g)}$ is identified with the tensor product of $\E{f}$ and $\E{g}$. By abuse of notation, we write 
\[\E{(f+g)}= \E{f}\cdot \E{g} = \E{f} \otimes \E{g}.\]

\begin{proposition}\label{prop:product}
Let $k$ be a field of characteristic not~$2$, and let $f\in k[x_1, \dots, x_{n_1}]$ and $g\in k[y_1, \dots, y_{n_2}]$ with $f(0)=g(0)=0$. If $1\leq\mu_0(f),\mu_0(g)<\infty$, then 
\[\mu^{\mathbb{A}^{1}}_0(f+g)=\mu^{\mathbb{A}^{1}}_0(f)\cdot\mu^{\mathbb{A}^{1}}_0(g).\]
\end{proposition}

\begin{proof}
Let $\{u_i\}$ be a vector space basis of $Q_0(\nabla f)$ and $\{v_{\ell}\}$ a vector space basis of $Q_0(\nabla g)$, so that $\{u_i\otimes v_{\ell}\}$ forms a basis of $Q_0(\nabla f)\otimes Q_0(\nabla g)\cong Q_0(\nabla(f+g))$. Let $\eta_f\colon Q_0(\nabla f)\to k$ and $\eta_g\colon Q_0(\nabla g)\to k$ be linear functionals such that $\eta_f(\E{f})=1$ and $\eta_g(\E{g})=1$. Let $\eta\colon Q_0(\nabla f)\times Q_0(\nabla g)\to k$ be the bilinear map defined by $\eta(u_i,v_{\ell})=\eta_f(u_i)\eta_g(v_{\ell})$, and let $\widetilde{\eta}\colon Q_0(\nabla f)\otimes Q_0(\nabla g)\to k$ be the induced map. We then have 
\[\widetilde{\eta}(\E{(f+g)})=\widetilde{\eta}(\E{f}\otimes \E{g})=\eta_f(\E{f})\eta_g(\E{g})=1\cdot 1=1.\]

Let $i$ (resp.~$\ell$) index the rows of the Gram matrix of $\mu^{\mathbb{A}^{1}}_0(f)$ (resp.\ $\mu^{\mathbb{A}^{1}}_0(g)$), and let $j$ (resp.~$m$) index the columns of the Gram matrix of $\mu^{\mathbb{A}^{1}}_0(f)$ (resp.\ $\mu^{\mathbb{A}^{1}}_0(g)$). The pairs $(i,\ell)$ (resp.\ $(j, m)$ then index the rows (resp.\ columns) of the Gram matrix of $\mu^{\mathbb{A}^{1}}_0(f+g)$. In position $((i,\ell),(j,m))$, the Gram matrix of $\mu_0^{\A1}(f+g)$ has entry 
\[\widetilde\eta((u_i\otimes v_{\ell})(u_{j}\otimes v_m))=\widetilde\eta(u_i u_{j}\otimes v_{\ell}v_m)=\eta_f(u_i u_j)\eta_g(v_{\ell} v_m).\] 
As the Gram matrix of $\mu_0^{\A1}(f)$ has entries $\eta_f(u_i u_{j})$ and the Gram matrix of $\mu_0^{\A1}(g)$ has entries $\eta_g(v_{\ell} v_m)$, the entry of the Gram matrix of $\mu_0^{\A1}(f)\cdot\mu_0^{\A1}(g)$ in position $((i,\ell), (j,m))$ is also $\eta_f(u_i u_j)\eta_g(v_{\ell} v_m)$. Therefore $\mu^{\A1}_0(f+g)=\mu_0^{\A1}(f)\cdot\mu_0^{\A1}(g)$.
\end{proof}

We now record a special case of Proposition~\ref{prop:product} that will be of use to us. We will later prove another related result in Theorem~\ref{thm:variablereduction}.

\begin{corollary}\label{cor:addterm}
Let $k$ be a field of characteristic not~$2$. If $f(\mathbf{x},z)=g(\mathbf{x})+az^2$, where $a\in k^{\times}$, then
\[\mu^{\mathbb{A}^{1}}_0(f)=\langle {2a} \rangle \cdot\mu^{\mathbb{A}^{1}}_0(g).\qedhere\]
\end{corollary}

\begin{proof}
We consider $az^2$ as a polynomial in $k[z]$ and compute $\mu_0^{\A1}(az^2)$. Since the derivative is $2az$, the local ring has basis $\{1\}$, and the Hessian determinant is $2a$. Therefore we can define $\eta$ by $\eta(1)=1/(2a)$, giving $\mu_0^{\A1}(az^2)=\langle1/(2a)\rangle$. Proposition~\ref{prop:product} thus gives 
\[\mu^{\mathbb{A}^{1}}_0(f)=\Bigl\langle \frac{1}{2a} \Bigr\rangle \cdot\mu^{\mathbb{A}^{1}}_0(g)=\langle {2a} \rangle \cdot\mu^{\mathbb{A}^{1}}_0(g).\qedhere\] 
\end{proof}

The construction in Corollary~\ref{cor:addterm} can be used to show that elements of $\GW(k)$ that arise from a polynomial with a certain number of variables also arise from polynomials with any larger number of variables. 

\begin{proposition}\label{prop:increasen}
Let $k$ be a field of characteristic not~$2$, and let $n$ be a positive integer. If an element of $\GW(k)$ is the $\A1$-Milnor number of a polynomial in $n$~variables, then, for all $n'\geq n$, it also arises as the $\A1$-Milnor number of a polynomial in $n'$~variables.
\end{proposition}

\begin{proof}
Apply Corollary~\ref{cor:addterm} $n'-n$ times with $a=1/2$ at each step.
\end{proof}

\section{All local \texorpdfstring{$\A1$}{A\textonesuperior}-Brouwer degrees of rank at most \texorpdfstring{$7$}{7} are \texorpdfstring{$\A1$}{A\textonesuperior}-Milnor numbers}\label{sec:allarise}

Quick, Strand, and Wilson~\cite{QuickStrandWilson} determine the elements of the Grothendieck--Witt ring of rank at most~$7$ that are local $\A1$-Brouwer degrees, as shown in Table~\ref{tab:A1Brouwer}. In this section, we show that, over a field of characteristic not $2$, $3$, or~$5$, all of these elements also arise as $\A1$-Milnor numbers of polynomials in some number of variables. Over many fields, however, not all such elements arise as $\A1$-Milnor numbers of single-variable polynomials (see Proposition~\ref{prop:onevariable}).

\begin{table}
\begin{tabular}{cc}
\toprule
Rank & Local $\A1$-Brouwer degrees\\
\midrule
1 & $\langle b\rangle$\\
2 & $\hyp$\\
3 & $\hyp+\langle b\rangle$\\
4 & $\hyp+\langle b,c\rangle$\\
5 & $2\hyp+\langle b\rangle$\\
6 & $2\hyp+\langle b,c\rangle$\\
7 & $3\hyp+\langle b\rangle$\\
\bottomrule
\end{tabular}
\caption{Possible local $\A1$-Brouwer degrees according to~\cite{QuickStrandWilson}. Here $b$ and~$c$ are allowed to take any nonzero values. \label{tab:A1Brouwer}}
\end{table}

We proceed by calculating $\A1$-Milnor numbers of certain families of polynomials. The first family of polynomials is similar to one studied in~\cite{QuickStrandWilson} and shows that many local $\A1$-Brouwer degrees are $\A1$-Milnor numbers of single-variable polynomials.

\begin{example}[cf.\ {\cite[Lemma~2.4]{QuickStrandWilson}}]\label{ex:onevariable}
Let $k$ be a field of characteristic not~$2$, let $b\in k^{\times}$, and let $r$ be a positive integer such that $\charac k$ does not divide $r+1$. Consider the single-variable polynomial $f=\frac{1}{b(r+1)}x^{r+1}$. Then $f_x=\frac{1}{b}x^r$, so $Q_0(\nabla f)\cong k[x]/(f_x)$ is an $r$-dimensional vector space with basis $\{1,x,\dots,x^{r-1}\}$. The Hessian determinant is $\frac{r}{b}x^{r-1}$, and the distinguished socle element is $\E{f}=\frac{1}{b}x^{r-1}$. We can define $\eta\colon Q_0(\nabla f)\to k$ by $\eta(x^{r-1})=b$ and $\eta(x^i)=0$ for all $i\in\{0,1,\dots,r-2\}$. This implies that $\eta(\E{f}) =1$, so we may use this functional to obtain the EKL form. With the given ordering of the basis, the Gram matrix is then an antidiagonal matrix with $b$ in each antidiagonal position. By Proposition~\ref{prop:antidiagonal}, we thus have
\[ \mu^{\A1}_0(f)=\begin{cases}
\frac{r}{2}\hyp & \text{if }r\text{ is even,}\\
\frac{r-1}{2}\hyp+\langle b \rangle & \text{if }r\text{ is odd.}
\end{cases}\]
\end{example}

The following proposition shows that these are the only forms that can arise as $\A1$-Milnor numbers of single-variable polynomials. 

\begin{proposition}\label{prop:onevariable}
Let $k$ be a field of characteristic not~$2$, and let $f\colon\A1_k \to \A1_k$ be a polynomial map with $f(0)=0$ and $1\leq\mu_0(f)<\infty$. If $r\coloneqq\mu_0(f)$, then $\charac k\nmid r+1$ and 
\[ \mu^{\A1}_0(f)=\begin{cases}
\frac{r}{2}\hyp & \text{if $r$ is even,}\\
\frac{r-1}{2}\hyp+\langle {b} \rangle & \text{if $r$ is odd,} 
\end{cases}\]
for some $b\in k^{\times}$. 
\end{proposition}

\begin{proof}
If $f$ is a single-variable polynomial with $f(0)=0$ and $1\leq\mu_0(f)<\infty$, then $f=\sum_{i=2}^{d} a_i x^i \in k[x]$ for some~$d$. Since $\mu^{\A1}_0(f)$ depends only on $f_x$, we can remove all initial terms $a_ix^i$ for which $ia_i=0$ without changing the $\A1$-Milnor number, thus obtaining a new polynomial $g=\sum_{i=j}^{d} a_i x^i \in k[x]$ with $ja_{j}\neq0$ and $\mu^{\A1}_0(g)=\mu^{\A1}_0(f)$. Then $g_x = x^{j-1} \bigl(j a_{j} + \sum_{i={j+1}}^d i a_i x^{i-j}\bigr)$, where the second factor is a unit in the local ring $k[x]_{(x)}$ since $ja_j\neq 0$. Therefore $x^{j-1}=0$ in $Q_0(\nabla g)$, so $Q_0(\nabla g)\cong k[x]/(x^{j-1})$ is a $(j-1)$-dimensional vector space with basis $\{1,\, x, \, \dots, x^{j-2}\}$. Since $\mu_0(g)=\mu_0(f)=r$, we have $j=r+1$. Moreover, $\charac k\nmid r+1$ as $r+1 \neq 0$ in $k$.

Since $x^\ell=0$ in $Q_0(\nabla g)$ for all $\ell \geq r$, the Hessian determinant is $(r+1)ra_{r+1}x^{r-1}$ and the distinguished socle element is $\E{g}= (r+1)a_{r+1}x^{r-1}$ in $Q_0(\nabla g)$. We can define $\eta\colon Q_0(\nabla g)\to k$ by $\eta(x^{r-1})=1/((r+1)a_{r+1})$ and $\eta(x^i)=0$ for all $i\in\{0,1,\dots,r-2\}$. This implies that $\eta(\E{f}) =1$, so we may use this functional to obtain the EKL form. With the given ordering of the basis, the Gram matrix is then an antidiagonal matrix with the entries $b\coloneqq 1/((r+1)a_{r+1})$ along the diagonal. Applying Proposition~\ref{prop:antidiagonal} then yields the result. 
\end{proof}

Over fields of positive characteristic, Proposition~\ref{prop:onevariable} restricts the possible Milnor numbers of single-variable polynomials.

\begin{corollary}
Let $k$ be a field of characteristic not $0$ or~$2$. If $f\colon\A1_k\to \A1_k$ is a polynomial map with $f(0)=0$ and $1\leq\mu_0(f)<\infty$, then $\mu_0(f)\not \equiv -1 \pmod{\charac k}$. In particular, there is no such $f$ with $\mu_0(f)=\charac k -1$.
\end{corollary}

Together, Example~\ref{ex:onevariable} and Proposition~\ref{prop:onevariable} completely characterize the elements of the Grothendieck--Witt ring of any rank that are $\A1$-Milnor numbers of single-variable polynomials.

Quick, Strand, and Wilson \cite[Lemma~2.6]{QuickStrandWilson} show that every element of $\GW(k)$ of the form $m\hyp+\langle b,c\rangle$, where $m$ is a positive integer, is a local $\A1$-Brouwer degree. Therefore, over any field for which there is a nondegenerate form $\langle b,c\rangle$ that is not equivalent to~$\hyp$, there is a local $\A1$-Brouwer degree of every even rank at least~$4$ that does not arise as the $\A1$-Milnor number of a single-variable polynomial.

We next turn our attention to polynomials in more than one variable. The following example computes $\A1$-Milnor numbers of certain two-variable polynomials.

\begin{example}\label{ex:evenranktwovariables}
Let $k$ be a field of characteristic not~$2$, let $b,c\in k^{\times}$, and let $r\geq4$ be an even integer such that $\charac k$ does not divide $r-1$. Consider the polynomial
\[g=x^2y-\frac{b}{c(r-1)}y^{r-1}.\]
The Jacobian ideal is
\[ J_g=(g_x,\,g_y)=\Bigl(2xy,\, x^2-\frac{b}{c}y^{r-2}\Bigr).\]
This shows that $xy=0$ in $Q_0(\nabla g)$, so $xy^k=0$ for all $k\geq 1$. Since $xf_y=0$ in $Q_0(\nabla g)$ and $xy^{r-2}=0$ as $r\geq 4$, one also has $x^3=0$. Similarly, $yf_y=0$ and $x^2y=0$ shows that $y^j=0$ for all $j\geq r-1$. Also, $f_y\in J_f$ gives $x^2=\frac{b}{c}y^{r-2}$. Therefore $Q_0(\nabla g)$ is an $r$-dimensional vector space with basis $\{1, x, y, y^2, \dots, y^{r-2}\}$. We use the ordered basis
\[ \B\coloneqq\{1,\dots,y^{r/2 -2},x,y^{r/2-1},\dots,y^{r-2}\}.\]
The Hessian determinant is given by
\[\det H_g=-\frac{2br}{c}y^{r-2},\]
and the distinguished socle element by
\[\E{g}=-\frac{2b}{c}y^{r-2}.\]
Therefore we can define $\eta\colon Q_0(\nabla g)\to k$ by $\eta(y^{r-2})=-c/(2b)$ and $\eta(v)=0$ for all other $v\in\B$. The Gram matrix is then the block antidiagonal matrix 
\[\begin{bNiceArray}{ccc|c|ccc}
 & & & & & & -c/(2b)\\
 & & & & & \iddots & \\
 & & & & -c/(2b) & & \\
\hline
 & & & P & & \\
\hline
 & & -c/(2b) & & & & \\
 & \iddots & & & & & \\
-c/(2b) & & & & & & 
\end{bNiceArray},\]
where $P=\begin{bsmallmatrix}
-1/2 & 0\\
0 & -c/(2b)
\end{bsmallmatrix}$.
Applying Lemma~\ref{lemma:antidiagonal} $(r-2)/2$ times, we obtain $\mu_0^{\mathbb{A}^1}(g)=\frac{r-2}{2}\hyp+\langle-1/2, -c/(2b)\rangle$.
\end{example}

Using Example~\ref{ex:evenranktwovariables} together with Corollary~\ref{cor:addterm} allows us to find a family of three-variable polynomials that have $\A1$-Milnor numbers of the form given for rank~$4$ and rank~$6$ in Table~\ref{tab:A1Brouwer}.

\begin{example}\label{ex:evenrankthreevariables}
Let $g$ be as in Example~\ref{ex:evenranktwovariables}, and let 
\[f=g-\frac{1}{4b}z^2=x^2y-\frac{b}{c(r-1)}y^{r-1}-\frac{1}{4b}z^2.\]
Corollary~\ref{cor:addterm} gives 
\[\begin{split}
\mu^{\mathbb{A}^{1}}_0(f) & =\langle -2b \rangle \cdot\mu^{\mathbb{A}^{1}}_0(g)\\
 & =\langle -2b \rangle \cdot\biggl(\biggl(\frac{r-2}{2}\biggr) \hyp+\biggl\langle-\frac{1}{2}, -\frac{c}{2b}\biggr\rangle\biggr)\\ & =\biggl(\frac{r-2}{2}\biggr) \hyp+\langle b,c\rangle.
\end{split}\]
\end{example}

Examples \ref{ex:onevariable} and~\ref{ex:evenrankthreevariables} and Proposition~\ref{prop:increasen} allow us to obtain the following result.

\begin{theorem}\label{thm:allarise}
\begin{enumerate}[label=\textup{(\alph*)},ref=\textup{\alph*}]\thmenumhspace
\item Let $k$ be a field of characteristic not $2$ or~$3$, and let $n\geq1$. Each local $\A1$-Brouwer degree of rank $1$, $2$, $3$, $5$, or~$7$ is the $\A1$-Milnor number of a polynomial map $f\colon\A{n}_k\to \A1_k$. \label{part:allariseeasyranks} 
\item Let $k$ be a field of characteristic not $2$ or~$3$, and let $n\geq3$. Each local $\A1$-Brouwer degree of rank~$4$ is the $\A1$-Milnor number of a polynomial map $f\colon\A{n}_k\to \A1_k$. 
\item Let $k$ be a field of characteristic not $2$ or~$5$, and let $n\geq3$. Each local $\A1$-Brouwer degree of rank~$6$ is the $\A1$-Milnor number of a polynomial map $f\colon\A{n}_k\to \A1_k$. 
\end{enumerate}
\end{theorem}

\begin{proof}
Table~\ref{tab:A1Brouwer} gives the elements of $\GW(k)$ of rank at most~$7$ that are local $\A1$-Brouwer degrees. When $\charac k$ is not $2$ or~$3$, Example~\ref{ex:onevariable} shows that each of these forms of rank $1$, $2$, $3$, $5$, and~$7$ is the $\A1$-Milnor number of a single-variable polynomial. When $\charac k$ is not~$3$, Example~\ref{ex:evenrankthreevariables} shows that each of these forms of rank~$4$ is the $\A1$-Milnor number of a three-variable polynomial. When $\charac k$ is not~$5$, Example~\ref{ex:evenrankthreevariables} shows that each of these forms of rank~$6$ is the $\A1$-Milnor number of a three-variable polynomial. The result then follows by applying Proposition~\ref{prop:increasen}. 
\end{proof}

\begin{remark}
It is unclear to us whether the restrictions on characteristic in Theorem~\ref{thm:allarise} can be removed. 
\end{remark}

Theorem~\ref{thm:allarise} implies that, over fields of most characteristics, every local $\A1$-Brouwer degree of rank at most~$7$ also arises as the $\A1$-Milnor number of a three-variable polynomial.

\begin{theorem}\label{thm:allarise2}
Let $k$ be a field of characteristic not $2$, $3$, or~$5$. Every element of $\GW(k)$ of rank at most~$7$ that is a local $\A1$-Brouwer degree is also the $\A1$-Milnor number of a polynomial map $\A{3}_k\to \A1_k$. 
\end{theorem}

With the case of single-variable polynomials fully resolved and the case of polynomials in three or more variables resolved for rank at most~$7$, we now turn to the question of which elements of the Grothendieck--Witt ring are $\A1$-Milnor numbers of two-variable polynomials. Note that Theorem~\ref{thm:allarise}(\ref{part:allariseeasyranks}) shows that all local $\A1$-Brouwer degrees of rank $1$, $2$, $3$, $5$, or~$7$ arise as $\A1$-Milnor numbers of two-variable polynomials. After a detailed examination of Newton diagrams in Sections~\ref{sec:Newtonbackground} and~\ref{sec:classifynewton} and Newton degenerate polynomials with Milnor number at most~$8$ in Section~\ref{sec:change}, we classify the Grothendieck--Witt classes of ranks $4$, $6$, and~$8$ that arise as $\A1$-Milnor numbers of two-variable polynomials 
in Section~\ref{sec:classifyA1Milnor}.

\section{Newton diagrams, Newton nondegeneracy, and Kushnirenko's bound}\label{sec:Newtonbackground}

In this section, we review background about Newton diagrams, Newton nondegeneracy, Newton numbers, and Kushnirenko's bound. We will use this material in Section~\ref{sec:classifynewton} when we find all Newton diagrams with small Newton number. Then, in Section~\ref{sec:classifyA1Milnor}, we will use these Newton diagrams together with Kushnirenko's bound to characterize the elements of the Grothendieck--Witt ring of rank at most~$8$ that arise as $\A1$-Milnor numbers of two-variable polynomials.

We first recall the notion of a Newton diagram in~$\R^2$ as given in \cite{Mondal, Stevens, Wall}. (Note that \cite{Wall} instead refers to this as a Newton polygon.) Let $\Ac$ be a subset of $(\mathbb{Z}_{\geq0})^2$. The \emph{Newton polyhedron} $\Npoly_+(\Ac)$ is the convex hull of the set $\bigcup_{(i,j)\in \Ac} \bigl( (i,j)+(\R_{\geq 0})^2\bigr) \subseteq \R^2$. The \emph{Newton diagram} $\Npoly(\Ac)$ is the union of the compact faces of $\Npoly_+(\Ac)$. If a set~$\Npoly$ arises as $\Npoly(\mathcal{A})$ for some such~$\mathcal{A}$, we say that $\Npoly$ is a Newton diagram. The \emph{Newton polytope} $\Npoly_{-}(\Ac)$ is the union of all segments connecting the origin to points in $\Npoly(\Ac)$. We say that a Newton diagram $\Npoly(\Ac)$ is \emph{convenient} if it contains a point on both axes of~$\R^2$, which happens exactly when $\Ac$ contains a point on both axes. Given a polynomial $f=\sum_{i,j}a_{i,j}x^{i}y^{j}\in k[x,y]$, define 
\[\supp f=\{ (i,j)\in (\mathbb{Z}_{\geq0})^2\mid a_{i,j}\neq 0\}.\] 
We define the Newton polyhedron $\Npoly_+(f)$ to be $\Npoly_+(\supp f)$, the Newton diagram $\Npoly(f)$ to be $\Npoly(\supp f)$, and the Newton polytope $\Npoly_-(f)$ to be $\Npoly_-(\supp f)$. The polynomial~$f$ is called convenient if $\Npoly(f)$ is convenient.

We now state a result that allows us to consider only convenient polynomials when finding all elements of $\GW(k)$ that are $\A1$-Milnor numbers of two-variable polynomials. This follows quickly from \cite{KassWickelgren2019, QuickStrandWilson} and shows that polynomials that agree up to terms of a certain order have equal $\A1$-Milnor numbers.

\begin{lemma}[\cite{KassWickelgren2019, QuickStrandWilson}]\label{lemma:ne}
Let $k$ be a field of characteristic not~$2$, and let $f\colon\mathbb{A}_k^n\to\mathbb{A}_k^1$ be a polynomial map with $f(0)=0$ and $1\leq\mu_0(f)<\infty$. Let $r\coloneqq\mu_0(f)$. If $g\colon\mathbb{A}_k^n\to\mathbb{A}_k^1$ is a polynomial map with $f-g\in(x_1, \dots, x_n)^{r+2}$ in the local ring $k[x_1, \dots, x_n]_{(x_1, \dots, x_n)}$, then 
\begin{enumerate}[label=\textup{(\alph*)},ref=\textup{\alph*}]
\item\label{part:nea} $(f_{x_1},\dots,f_{x_n})=(g_{x_1},\dots,g_{x_n})$ in $k[x_1, \dots, x_n]_{(x_1, \dots, x_n)}$, 
\item\label{part:neb} $Q_0(\nabla f)=Q_0(\nabla g)$, and
\item\label{part:nec} $\mu_0^{\mathbb{A}^{1}}(f)=\mu_0^{\mathbb{A}^{1}}(g)$.
\end{enumerate}
\end{lemma}

\begin{proof}
Since $f-g\in (x_1, \dots, x_n)^{r+2}$, we have $f_{x_i}-g_{x_i}\in (x_1, \dots, x_n)^{r+1}$ for all $i\in[n]$. The maps $\nabla f,\nabla g\colon\mathbb{A}_k^n\to\mathbb{A}_k^n$ therefore satisfy the hypotheses of \cite[Lemma~2.9]{QuickStrandWilson}, and this gives (\ref{part:nea}), (\ref{part:neb}), and~(\ref{part:nec}). 
\end{proof}

Lemma~\ref{lemma:ne} shows that changing terms with total degree at least $r+2$ in~$f$ changes neither the local algebra $Q_0(\nabla f)$ nor the $\A1$-Milnor number $\mu_0^{\mathbb{A}^1}(f)$. Hence, when working with polynomials~$f$ with $\mu_0(f)=r<\infty$, the monomial terms with total degree at least $r+2$ can be chosen arbitrarily. In particular, when working with a polynomial $f\in k[x,y]$, we can remove all terms with degree at least $r+2$ and assume that 
\[f= \sum_{\substack{i,j\\1\leq i+j\leq r+1}}a_{i,j}x^iy^j.\]
Moreover, we can use the same lemma to add $x^{r+2}$ and/or $y^{r+2}$ without changing the $\A1$-Milnor number. Thus it suffices to consider convenient 
polynomials of the form
\[\sum_{\substack{i,j\\1\leq i+j\leq r+1}}a_{i,j}x^iy^j+x^{r+2}+y^{r+2}.\]

Given a subset $\Delta\subseteq\Npoly(f)$, let $f_{\Delta}$ be the polynomial consisting of terms of~$f$ corresponding to points in~$\Delta$. Following Kushnirenko, we say a polynomial~$f$ is \emph{Newton nondegenerate} on a closed face $\Delta\subseteq\Npoly(f)$ if the polynomials
\[x_1\frac{\partial f_{\Delta}}{\partial x_1},\, \dots,\, x_n\frac{\partial f_{\Delta}}{\partial x_n}\]
have no common zero in $\bigl( \overline{k}^{\times} \bigr)^n$, where $\overline{k}$ is the algebraic closure of~$k$. Otherwise, $f$ is \emph{Newton degenerate} on~$\Delta$. A polynomial~$f$ is \emph{Newton nondegenerate} if it is Newton nondegenerate on all closed faces of $\Npoly(f)$ and \emph{Newton degenerate} if it is Newton degenerate on some closed face of $\Npoly(f)$. 

The following lemma gives necessary and sufficient conditions for Newton nondegeneracy on a vertex of a Newton diagram. 

\begin{lemma}\label{lemma:atvertex}
Let $k$ be a field, let $i$ and~$j$ be nonnegative integers, and let $a\in k^\times$. The polynomial $f=ax^iy^j$ is Newton nondegenerate if and only if $(i,j)\neq(0,0)$ in~$k^2$.
\end{lemma}

\begin{proof}
If $(i,j)=(0,0)$ in $k^2$, then $f_x=0$ and $f_y=0$. Therefore any $(x,y)\in(k^\times)^2$ is a solution of $f_x=0$ and $f_y=0$, making $f$ Newton degenerate. If $(i,j)\neq(0,0)$ in~$k^2$, then without loss of generality we can suppose $i\neq0$ in~$k$. Thus $ia\neq0$, so $f_x=iax^{i-1}y^j$ can have a zero only when $x=0$ or $y=0$, making $f$ Newton nondegenerate. 
\end{proof}

\begin{remark}\label{remark:degeneracy}
When $\charac k=0$, Lemma~\ref{lemma:atvertex} shows that one only needs to check Newton nondegeneracy of a two-variable polynomial~$f$ on edges $\Delta\subseteq \Npoly(f)$. When $\charac k>0$, the lemma allows one to reduce the number of vertices of $\Npoly(f)$ at which Newton nondegeneracy needs to be checked.
\end{remark}

This following lemma gives sufficient conditions for Newton nondegeneracy on an edge of a Newton diagram. 

\begin{lemma}\label{lemma:onboundary} 
Let $k$ be a field, $f\in k[x,y]$, and let $e$ be an edge on $\Gamma(f)$. Suppose $f_e=a_{i_0,j_0}x^{i_0}y^{j_0}+ a_{i_1,j_1}x^{i_1}y^{j_1}$.
\begin{enumerate}[label=\textup{(\alph*)},ref=\textup{\alph*}]
\item \label{lemma:onboundary-a} If $i_1=0$ in~$k$ and $i_0\neq 0$ in~$k$, then $f$ is Newton nondegenerate on~$e$. 
\item \label{lemma:nointermediatepoint} If $i_0j_1-i_1j_0\neq0$ in~$k$, then $f$ is Newton nondegenerate on~$e$.
\end{enumerate}
\end{lemma}

\begin{proof}
For~(\ref{lemma:onboundary-a}), we have that $\frac{\partial f_e}{\partial x}=i_0a_{i_0,j_0}x^{i_0-1}y^{j_0}$. Since $i_0a_{i_0,j_0}\neq0$, it follows that $\frac{\partial f_e}{\partial x}=0$ only when $x=0$ or $y=0$.

For~(\ref{lemma:nointermediatepoint}), regarding 
\[\begin{aligned}
x\frac{\partial f_e}{\partial x}
&= i_0a_{i_0,j_0}x^{i_0}y^{j_0}+ i_1a_{i_1,j_1}x^{i_1}y^{j_1}=0\\
y\frac{\partial f_e}{\partial y}
&= j_0a_{i_0,j_0}x^{i_0}y^{j_0}+ j_1a_{i_1,j_1}x^{i_1}y^{j_1}=0
\end{aligned}\]
as a system of linear equations with unknowns $x^{i_0}y^{j_0}$ and $x^{i_1}y^{j_1}$, the determinant of the coefficient matrix is $(i_0j_1-i_1j_0)a_{i_0,j_0}a_{i_1,j_1}$, which is nonzero by hypothesis. Therefore the only solution is when $x^{i_0}y^{j_0}=0$ and $x^{i_1}y^{j_1}=0$, which is only possible when $x=0$ or $y=0$, so $f$ is Newton nondegenerate on the edge between $(i_0,j_0)$ and $(i_1,j_1)$. 
\end{proof}

Associated to a Newton diagram~$\Npoly$ is its Newton number, denoted by $\nu(\Npoly)$. Let $\Npoly$ be a convenient Newton diagram, let $(M,0)$ be the point where $\Npoly$ intersects the horizontal axis, and let $(0,N)$ be the point where $\Npoly$ intersects the vertical axis. The \emph{Newton number} of~$\Npoly$ is
\[\nu(\Npoly) = 2\Area(\Npoly_-)-M-N+1.\]
When $f\in k[x,y]$ is convenient, we define the \emph{Newton number} of~$f$, which we denote by $\nu(f)$, to be $\nu(\Npoly(f))$. When $f$ is not convenient, $\nuf$ is defined as $\sup_{i, j\in\mathbb{Z}_{>0}}\nu(f+x^i+y^j)$.

The next lemma follows immediately from the definition of Newton number.

\begin{lemma}\label{lemma:newtonpolygonareas}
Let $k$ be a field, and let $f,g\in k[x,y]$ be convenient polynomials whose Newton diagrams share the same points on the coordinate axes. If $\Npoly_-(f)\subsetneq \Npoly_-(g)$, then $\nuf< \nug$.
\end{lemma}

An important result of Kushnirenko shows that the Milnor number of a polynomial is bounded below by its Newton number, and that this bound is an equality under additional hypotheses. 

\begin{theorem}[Kushnirenko {\cite[Th\'eor\`eme~I]{Kushnirenko}}]\label{thm:kushnirenko}
Let $k$ be a field, and let $f\colon\A2_k\to\A1_k$ be a polynomial map with $f(0)=0$. Then 
\[\mu_0(f)\geq \nu(f).\]
Moreover, if $f$ is convenient and Newton nondegenerate, 
then 
\[\mu_0(f)= \nu(f).\]
\end{theorem}

\begin{remark}\label{rem:localizationpowerseries}
The result Kushnirenko proved in \cite[Th\'eor\`eme~I]{Kushnirenko} holds more generally for polynomial maps $\A{n}_k\to\A1_k$, but we only state it for maps with $n=2$ as we have only defined Newton numbers for two-variable polynomials. It also holds for polynomials that are not convenient under the additional hypothesis that $\charac k=0$. Kushnirenko states his result in terms of formal power series, but, as $Q_0(\nabla f)$ is an Artinian local ring, it is isomorphic to $k[[x,y]]/(f_x,f_y)$ (see \cite[Chapter 10]{AtiyahMacdonald}). We choose to work with the localized ring throughout this paper.
\end{remark}

\begin{remark}\label{rem:nondegenerate}
When $\charac k=0$ and $f(0)=0$, the Tjurina number of~$f$ is finite exactly when its Milnor number is finite. However, when $k$ has positive characteristic, $f$ may have an isolated singularity at the origin with infinite Milnor number (i.e.\ $1\leq \tau_0(f)<\infty$ and $\mu_0(f)=\infty$). When $f$ is convenient and Newton nondegenerate, Theorem~\ref{thm:kushnirenko} implies that $\mu_0(f)$ is finite. Thus, when $f(0)=0$ and $f$ is convenient and Newton nondegenerate, $\tau_0(f)$ is finite if and only if $\mu_0(f)$ is finite without any restriction on characteristic. 
\end{remark}

\begin{example}\label{ex:running2}
We illustrate Theorem~\ref{thm:kushnirenko} by continuing with the polynomial~$f$ from Example~\ref{ex:running}. As 
\[f = (x+y^{2})^{2} + xy^{3} = x^{2} + 2xy^{2} + xy^{3} + y^{4},\] 
we have $\supp f=\{(2,0)$, $(1,2)$, $(1,3)$, $(0,4)\}$ and that $f$ is convenient. The Newton diagram $\Npoly(f)$ is the boundary of the convex hull of points in $\supp f$ and is the line segment connecting $(2,0)$ and $(0,4)$.
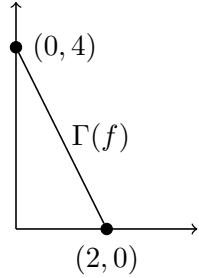
\begin{center}
\begin{tikzpicture}[scale=0.6]
\draw[->,semithick] (0,0)--(4,0);
\draw[->,semithick] (0,0)--(0,5);
\node[circle,draw=black,fill=black, label=right:{$(0,4)$},inner sep=0pt, minimum size=.15cm] (A) at (0,4) {};

\node[circle,draw=black,fill=black, label=below:{$(2,0)$},inner sep=0pt, minimum size=.15cm] (E) at (2,0) {};

\draw[semithick] (A)--node[right]{$\Npoly(f)$}(E);

\end{tikzpicture}
\end{center}
The Newton number is $\nu(f) = 8 - 2 - 4 + 1 = 3$. By Theorem~\ref{thm:kushnirenko},
\[\mu_0(f)\geq \nuf=3,\]
giving a lower bound for $\mu_0(f)$. This bound is strict as $\mu_0(f)$ was found to be~$4$ in Example~\ref{ex:running}. Since $\mu_0(f)> \nuf$, we know that $f$ must be Newton degenerate. One can verify this directly as the edge $e=\Npoly(f)$ corresponds to $f_e=x^{2} + 2xy^{2} + y^{4}$ in~$f$. We have $x\frac{\partial f_e}{\partial x}=2x^2+2xy^2$ and $y\frac{\partial f_e}{\partial y}=4xy^2+4y^4$. Note that $(-1,1)\in(k^{\times})^2$ is a solution of both $x\frac{\partial f_e}{\partial x}=0$ and $y\frac{\partial f_e}{\partial y}=0$, so $f$ is degenerate on~$e$. 
\end{example}

\section{Newton diagrams with Newton number at most \texorpdfstring{$8$}{8}}
\label{sec:classifynewton}

In this section, we find all convenient Newton diagrams in the plane whose Newton number is at most~$8$. By Theorem~\ref{thm:kushnirenko}, all convenient two-variable polynomials with Milnor number at most~$8$ have one of these Newton diagrams. This will be used in Section~\ref{sec:classifyA1Milnor} to characterize the elements of the Grothendieck--Witt ring of rank at most~$8$ that are $\A1$-Milnor numbers of two-variable polynomials.

Let $\Npoly$ be a convenient Newton diagram in~$\mathbb{R}^2$, as defined in Section~\ref{sec:Newtonbackground}, and let $\Npoly_-$ be the region bounded by~$\Npoly$ and the positive axes. The Newton number of~$\Npoly$ is given by 
\[\N(\Npoly)=2\Area(\Npoly_-)-M-N+1,\] 
where $(M,0)$ is the point where $\Npoly$ intersects the horizontal axis and $(0,N)$ is the point where $\Npoly$ intersects the vertical axis.

By construction, $\Gamma$ must consist of a finite collection of line segments of different slopes. The following lemma allows us to restrict attention to cases where $\Npoly$ has at most three segments of different slopes. 

\begin{lemma} \label{lemma:4edges}
If $\Npoly$ is a convenient Newton diagram in~$\mathbb{R}^2$ with at least~$4$ edges, then $\N(\Npoly)\geq 9$.
\end{lemma}

\begin{proof} 
Let $t$ be the number of edges of~$\Npoly$, and suppose $t\geq 4$. Let $(0,N)$, $(\gamma_1,\delta_1)$, $\dots$, $(\gamma_{t-1},\delta_{t-1})$, and $(M,0)$, be the vertices at the ends of the $t$~edges, as shown in Figure~\ref{fig:4edges}. We have $\gamma_1, \delta_{t-1}\geq 1$, and moreover $\gamma_{\ell},\delta_{t-\ell}\geq \ell$ for all $\ell\in[t-1]$. Thus $\gamma_{\ell}+\delta_{\ell}\geq t\geq4$ for all $\ell\in[t-1]$. This means the vertices of~$\Npoly$ are outside or on the boundary of the red triangle with vertices at $(0,0)$, $(3,1)$, and $(1,3)$. Since $(M,0)$ is the vertex of~$\Npoly$ on the horizontal axis, there is no vertex of~$\Npoly$ on the horizontal axis to the left of $(M,0)$. Therefore the gray triangle with vertices at $(0,0)$, $(3,1)$, and $(M,0)$ is contained within the region in the between $\Npoly$ and the axes. Similarly, the gray triangle with vertices at $(0,0)$, $(1,3)$, and $(0,N)$ is also within the region in the between $\Npoly$ and the axes. This shows that the full shaded region in Figure~\ref{fig:4edges} is contained in~$\Gamma_-$. As the area of the shaded region is $4+ M/2+N/2$, we have $\Area(\Npoly_-)\geq 4 +M/2+N/2$. Hence $\N(\Npoly) = 2 \Area(\Npoly_-) -M-N+1\geq 8 + 1 = 9$. 
\end{proof}

\begin{figure}
\centering
\begin{tikzpicture}[scale=0.7]
\draw[->,semithick] (0,0)--(9,0);
\draw[->,semithick] (0,0)--(0,9);
\node[circle,draw=black,fill=black, label=right:{$(0,N)$},inner sep=0pt, minimum size=.15cm] (A) at (0,8) {};
\node[circle,draw=black,fill=black, label=right:{$(\gamma_1,\delta_1)$},inner sep=0pt, minimum size=.15cm] (B) at (1,5) {};
\node[circle,draw=black,fill=black, label=above right:{$(\gamma_2,\delta_2)$},inner sep=0pt, minimum size=.15cm] (C) at (3,2) {};
\node[circle,draw=black,fill=black, label=above right:{$(\gamma_{t-1},\delta_{t-1})$},inner sep=0pt, minimum size=.15cm] (D) at (5,1) {};

\node[circle,draw=black,fill=black, label=below:{$(M,0)$},inner sep=0pt, minimum size=.15cm] (E) at (8,0) {};
\draw[semithick] (A)--(B)--(C);
\draw[dotted] (C)--(D);
\draw[semithick] (D)--(E);

\draw[fill=gray!30, dotted] (0,8)--(1,3)--(0,0)--cycle;
\draw[fill=gray!30, dotted] (8,0)--(3,1)--(0,0)--cycle;

\fill[fill=red!20, dotted] (0,0)--(1,3)--(3,1)--cycle;
\node[circle,draw=black,fill=black, label=below:{$(1,3)$},inner sep=0pt, minimum size=.15cm] (F) at (1,3) {};
\node[circle,draw=black,fill=black, label=left:{$(3,1)$},inner sep=0pt, minimum size=.15cm] (H) at (3,1) {};

\draw[dotted] (A)--(F)--(H)--(E);
\draw[dotted] (0,0)--(F);
\draw[dotted] (0,0)--(H);
\node at (1.1,1.3) {$4$};
\node at (0.38,3.15) {$\frac{N}{2}$};
\node at (3.1,0.38) {$\frac{M}{2}$};
\end{tikzpicture}
\caption{A Newton diagram $\Npoly$ with $t\geq 4$ edges.\label{fig:4edges}}
\end{figure}

As a result of Lemma~\ref{lemma:4edges}, it suffices to consider the cases when $\Npoly$ has zero, one, two, or three edges. If $\Npoly$ is a convenient Newton diagram with zero edges, then the only possibility is $\Npoly=\{(0,0)\}$, which has Newton number~$1$. This Newton diagram will not be of interest to us as it corresponds to polynomials with $f(0)\neq0$ while this paper studies polynomials with $f(0)=0$. 

The next proposition gives the possibilities when $\Npoly$ has one edge. 

\begin{proposition}\label{prop:newton-one-edges}
If $\Npoly$ is a convenient Newton diagram in~$\mathbb{R}^2$ with one edge and $\N(\Npoly)\leq8$, then the edge has endpoints $\{(0,N)$, $(M,0) \}$, or by symmetry $\{(0,M)$, $(N,0)\}$, with $(M, N)$ given by
\begin{table}[ht!]
\begin{tabular}{c|c |c| c |c |c |c |c |c |c }
$\N(\Npoly)$ &0 & 1 & 2 & 3 & 4 & 5 & 6 & 7 & 8 \\
\hline
$(M, N)$ & $(1,N)$ & $(2,2)$ & $(2,3)$
 & $(2,4)$,
 & $(2,5)$
 & $(2,6)$,
 & $(2,7)$
 & $(2,8)$,
 & $(2,9)$\\
 & & & & $(3,3)$ & & $(3,4)$ & & $(3,5)$\\
\end{tabular}
\end{table}
\end{proposition}

\begin{proof}
Let the edge of~$\Npoly$ have terminal vertices $(0, N)$ and $(M, 0)$. Then $\Area(\Npoly_-)= MN/2$ and $\N(\Npoly)=(M-1)(N-1)$. By symmetry, assume that $M\leq N$. For $\N(\Npoly)\leq8$, the possible $(M, N)$ are given by the table. 
\end{proof}

The next proposition gives the possibilities when $\Npoly$ has two edges.

\begin{proposition}\label{prop:newton-two-edges}
If $\Npoly$ is a convenient Newton diagram in~$\mathbb{R}^2$ with two edges and $\N(\Npoly)\leq8$, then the edges have endpoints $\{(0,N)$, $(\gamma,\delta)$, $(M,0)\}$, or by symmetry $\{(0,M)$, $(\delta,\gamma)$, $(N,0)\}$, with $\gamma,\delta,M,N$ as follows: 

\begin{table}[ht]
\begin{tabular}{c|l}
$\N(\Npoly)$ & Possible $\gamma,\delta,M,N$\\
\hline
0 & $\emptyset$\\
\hline
1 & $(\gamma,\delta)=(1,1)$, with $M,N\geq 2$ and $(M,N)\neq(2,2)$\\
\hline
2 & $\emptyset$\\
\hline
3 & $(\gamma,\delta)=(2,1)$, with $M\geq 5$ and $N=2$\\
\hline
4 & $(\gamma,\delta)=(2,1)$, with $M\geq 4$ and $N=3$ \\
\hline
5 & $(\gamma,\delta)=(2,1)$, with $M\geq3$ and $N=4$\\
 & $(\gamma,\delta)=(3,1)$, with $M\geq7$ and $N=2$\\
\hline
6 & $(\gamma,\delta)=(2,1)$, with $M\geq3$ and $N=5$\\
\hline
7 & $(\gamma,\delta)=(2,1)$, with $M\geq3$ and $N=6$\\
 & $(\gamma,\delta)=(3,1)$, with $M\geq5$ and $N=3$
\\
 & $(\gamma,\delta)=(4,1)$, with $M\geq9$ and $N=2$
\\
\hline
8 & $(\gamma,\delta)=(2,1)$, with $M\geq3$ and $N=7$\\
\end{tabular}
\end{table}
\end{proposition}

\begin{figure}
\centering
\begin{tikzpicture}[scale=0.7]
\draw[->,semithick] (0,0)--(6,0);
\draw[->,semithick] (0,0)--(0,6);
\node[circle,draw=black,fill=black, label=right:{$(0,N)$},inner sep=0pt, minimum size=.15cm] (A) at (0,5) {};
\node[circle,draw=black,fill=black, label=above right:{$(\gamma,\delta)$},inner sep=0pt, minimum size=.15cm] (B) at (1.8,1.8) {};

\node[circle,draw=black,fill=black, label=below:{$(M,0)$},inner sep=0pt, minimum size=.15cm] (E) at (5,0) {};

\draw[semithick] (0,0)--(B);
\draw[semithick] (A)--(B)--(E);

\node at (0.7,2) {$\frac{\gamma N}{2}$};
\node at (2,.7) {$\frac{\delta M}{2}$};
\end{tikzpicture}
 \caption{A Newton diagram $\Npoly$ with two edges.\label{fig:2edges}}
\end{figure}

\begin{proof}
Let the terminal vertices of the two edges of~$\Npoly$ be $(0, N)$, $(\gamma,\delta)$, and $(M, 0)$, where $0<\gamma <M$ and $0< \delta<N$. Computing areas as shown in Figure~\ref{fig:2edges}, we have
\[\N(\Npoly)=(\delta-1)M+(\gamma -1)N+1.\]
For each value of $\N(\Npoly)\leq8$, we determine the corresponding Newton diagrams. If $\gamma,\delta\geq2$, then $M,N\geq3$, and moreover either $M\geq5$ or $N\geq5$ in order for $(\gamma,\delta)$ to be a vertex of the Newton diagram. Thus we would have that $\N(\Npoly)=(\delta-1)M+(\gamma -1)N+1 \geq9$. By symmetry, assume $\gamma \geq \delta$. In finding Newton diagrams with $\N(\Npoly)\leq8$, we can therefore assume $\delta=1$, in which case $\N(\Npoly)=(\gamma -1)N+1$. As the right side of this equation is positive, $\N(\Npoly)=0$ is not possible. The possibilities for positive values of $\nu(\Gamma)$ are as follows.
\begin{itemize}
\item $\N(\Npoly)=1$: Then $(\gamma -1)N=0$. This means $\gamma =1$, and there are no restrictions on $M$ and~$N$ beyond $M,N\geq2$ except that $(M,N)=(2,2)$ does not give two edges.
\item $\N(\Npoly)=2$: Then $(\gamma -1)N=1$. As $N\geq2$, there are no solutions. 
\item $\N(\Npoly)=3$: Then $(\gamma -1)N=2$. As $N\geq 2$, we must have $\gamma =2$ and $N=2$. This gives two edges exactly when $M\geq 5$. 
\item $\N(\Npoly)=4$: Then $(\gamma -1)N=3$. As $N\geq 2$, we must have $\gamma =2$ and $N=3$. This gives two edges exactly when $M\geq4$.
\item $\N(\Npoly)=5$: Then $(\gamma -1)N=4$, which forces either $(\gamma,N )=(2,4)$ or $(\gamma,N )=(3,2)$. In the first case, this gives two edges exactly when $M\geq 3$. In the second case, it gives two edges exactly when $M\geq7$.
\item $\N(\Npoly)=6$: Then $(\gamma -1)N=5$, which forces $\gamma=2$ and $N=5$.
This gives two edges exactly when $M\geq3$.
\item $\N(\Npoly)=7$: Then $(\gamma -1)N=6$, which forces $(\gamma,N )=(2,6)$, $(\gamma,N )=(3,3)$, or $(\gamma,N )=(4,2)$. In the first case, this gives two edges exactly when $M\geq3$. In the second case, it gives two edges exactly when $M\geq5$. In the third case, it gives two edges exactly when $M\geq9$.
\item $\N(\Npoly)=8$: Then $(\gamma -1)N=7$, which forces $\gamma=2$ and $N=7$.
This gives two edges exactly when $M\geq3$. \qedhere
\end{itemize}
\end{proof}

The following proposition gives the possibilities when $\Npoly$ has three edges.

\begin{proposition}\label{prop:newton-three-edges}
If $\Npoly$ is a convenient Newton diagram in~$\mathbb{R}^2$ with three edges and $\N(\Npoly)\leq8$, then $\N(\Npoly)$ is $4$, $6$, or~$8$. Moreover, the edges of $\Npoly$ have terminal vertices $\{(0, N)$, $(\gamma_1, \delta_1 )$, $(\gamma_2,\delta_2 )$, $(M, 0)\}$, or by symmetry $\{(0,M)$, $(\delta_2,\gamma_2)$, $(\delta_1,\gamma_1 )$, $(N,0)\}$, with $\gamma_1,\delta_1,\gamma_2,\delta_2,M,N$ as follows: 

\begin{table}[ht]
\begin{tabular}{c|l}
$\N(\Npoly)$ & Possible $\gamma_1,\delta_1,\gamma_2,\delta_2,M,N$\\
\hline
4 & $(\gamma_1,\delta_1 )=(1,2)$ and $(\gamma_2,\delta_2 )=(2,1)$, with $M,N \geq 4$\\
\hline
6 & $(\gamma_1,\delta_1 )=(1,3)$ and $(\gamma_2,\delta_2 )=(2,1)$, with $M\geq3$ and $N \geq 6$\\
\hline
8 & $(\gamma_1,\delta_1 )=(1,4)$ and $(\gamma_2,\delta_2 )=(2,1)$, with $M\geq3$ and $N \geq 8$\\
\end{tabular}
\end{table}
\end{proposition}

\begin{figure}
\centering
\begin{tikzpicture}[scale=1.2]
\draw[->,semithick] (0,0)--(10,0);
\draw[->,semithick] (0,0)--(0,10);
\node[circle,draw=black,fill=black, label=right:{$(0,N)$},inner sep=0pt, minimum size=.15cm] (A) at (0,9.5) {};
\node[circle,draw=black,fill=black, label=right:{$(\gamma_1, \delta_1 )$},inner sep=0pt, minimum size=.15cm] (B) at (2,6) {};
\node[circle,draw=black,fill=black, label=above right:{$(\gamma_2,\delta_2 )$},inner sep=0pt, minimum size=.15cm] (C) at (6,2) {};

\node[circle,draw=black,fill=black, label=below:{$(M,0)$},inner sep=0pt, minimum size=.15cm] (E) at (9.5,0) {};

\draw[semithick] (B)--(0,0);
\draw[semithick] (C)--(0,0);
\draw[semithick] (B)--(2,2);
\draw[semithick] (C)--(2,2);
\draw[semithick] (0,0)--(2,2);
\draw[semithick] (A)--(B)--(C)--(E);

\draw[fill=gray!10, semithick] (0,0)--(2,6)--(2,2)--cycle;
\draw[fill=gray!10, semithick] (2,2)--(0,0)--(6,2)--cycle;
\node at (3.35,3.1) {$\frac{(\gamma_2-\gamma_1 )(\delta_1 -\delta_2 )}{2}$};
\node at (1.4,2.2) {$\frac{\gamma_1 (\delta_1 -\delta_2 )}{2}$};
\node at (2.7,1.5) {$\frac{\delta_2 (\gamma_2-\gamma_1 )}{2}$};
\node at (0.75,6) {$\frac{\gamma_1 N}{2}$};
\node at (6,.75) {$\frac{\delta_2 M}{2}$};
\end{tikzpicture}
 \caption{A Newton diagram $\Npoly$ with three edges.\label{fig:3edges}}
\end{figure}

\begin{proof}
Let the three edges of~$\Npoly$ have terminal vertices $\{(0, N)$, $(\gamma_1, \delta_1 )$, $(\gamma_2,\delta_2 )$, $(M, 0)\}$, where $0<\gamma_1 <\gamma_2<M$ and $0<\delta_2 <\delta_1 <N$. Computing areas as shown in Figure~\ref{fig:3edges}, we have
\begin{equation}\label{eq:Newtonnumber3edges}
\N(\Npoly)=(\gamma_1 -1)N+(\delta_2 -1)M+\gamma_1 (\delta_1 -\delta_2 )+\delta_2 (\gamma_2-\gamma_1 )+(\gamma_2-\gamma_1 )(\delta_1 -\delta_2 )+1.
\end{equation}
If $\gamma_1 \geq2$, this quantity can be at most~$8$ only if $\delta_2 =1$, $\delta_1 =2$, and $\gamma_2=3$, and $N=3$. However, this does not give a Newton diagram with three edges, so we conclude that $\gamma_1 =1$. Similarly, $\delta_2 =1$. 

Under the assumption $\gamma_1 =\delta_2 =1$, \eqref{eq:Newtonnumber3edges} simplifies to $\N(\Npoly)=\gamma_2\delta_1 $. By symmetry, suppose $\gamma_2 \leq \delta_1 $. As $\gamma_2,\delta_1 \geq2$, the only possibilities are $\N(\Npoly)=4$ with $\gamma_2 =2$ and $\delta_1 =2$, $\N(\Npoly)=6$ with $\gamma_2 =2$ and $\delta_1 =3$, and $\N(\Npoly)=8$ with $\gamma_2 =2$ and $\delta_1 =4$. In the first of these cases, we have three edges exactly when $M,N\geq4$. In the second case, we have three edges exactly when $M\geq3$ and $N\geq6$. In the third case, we have three edges exactly when $M\geq3$ and $N\geq8$.
\end{proof}

The above classification of Newton diagrams gives the following classification of degenerate two-variable polynomials with Newton number at most~$8$.

\begin{proposition}\label{prop:degenerateedge}
Let $k$ be a field of characteristic not $2$, $3$, $5$, or~$7$, let $f\in k[x,y]$ with $\nu(f)\leq8$, and let $e$ be an edge of $\Gamma(f)$ with endpoints $(i_0,j_0)$ and $(i_1,j_1)$, where $i_0<i_1$ and $j_0>j_1$.
\begin{enumerate}[label=\textup{(\alph*)},ref=\textup{\alph*}]
\item If $f$ is Newton degenerate on~$e$, then, up to interchanging variables, 
either\label{part:degenerateedge}
\begin{itemize}
\item $\Gamma(f)$ has one edge and $e$ connects $(2,0)$ to $(0,2)$, $(0,4)$, $(0,6)$, or $(0,8)$; 
\item $\Gamma(f)$ has one edge and $e$ connects $(3,0)$ to $(0,3)$;
or 
\item $\Gamma(f)$ has two edges and $e$ connects 
$(2,1)$ to $(0,3)$, $(0,5)$, or $(0,7)$. 
\end{itemize}
\item If $f$ is Newton degenerate on $e$, then $(i_1-i_0, j_0-j_1)$ must be one of $(2,2)$, $(2,4)$, $(2,6)$, $(2,8)$, or $(3,3)$. \label{part:differences}
\end{enumerate}
\end{proposition}

\begin{proof}
By Lemma~\ref{lemma:4edges}, $\Gamma(f)$ has at most~$3$ edges.
If $\Gamma(f)$ has one edge, then it must be one of the Newton diagrams found in Proposition~\ref{prop:newton-one-edges}. If $\Gamma(f)$ has one edge and $M=1$, then $f$ is Newton nondegenerate by Lemma~\ref{lemma:onboundary}(\ref{lemma:onboundary-a}). If $\Gamma(f)$ has one edge and $(M,N)$ is one of $(2,3)$, $(2,5)$, $(2,7)$, $(2,9)$, $(3,4)$, or $(3,5)$, then $f$ is Newton nondegenerate by Lemma~\ref{lemma:onboundary}(\ref{lemma:onboundary-a}). The only remaining possibilities for $(M,N)$ are $(2,2)$, $(2,4)$, $(2,6)$, $(2,8)$, and $(3,3)$.

If $\Gamma(f)$ has two edges, then it must be one of those found in Proposition~\ref{prop:newton-two-edges}. If an edge~$e$ has a vertex on a coordinate axis and does not have an interior lattice point, then Lemma~\ref{lemma:onboundary}(\ref{lemma:onboundary-a}) shows $e$ is Newton nondegenerate given our assumption on $\charac k$. The only case when $e$ may have an interior lattice point is when $e$ is an edge from $(2,1)$ to $(0,3)$, $(0,5)$, or $(0,7)$, as listed in the proposition.

If $\Gamma(f)$ has three edges, then it must be one of those found in Proposition~\ref{prop:newton-three-edges}. If an edge~$e$ has a vertex on a coordinate axis, then Lemma~\ref{lemma:onboundary}(\ref{lemma:onboundary-a}) shows $e$ is Newton nondegenerate given our assumption on $\charac k$. On an edge between $(2,1)$ and $(1,2)$, $(1,3)$, or $(1,4)$, $f$ must be Newton nondegenerate by Lemma~\ref{lemma:onboundary}(\ref{lemma:nointermediatepoint}). This completes the proof of~(\ref{part:degenerateedge}).

From~(\ref{part:degenerateedge}), we know that for each of the edges in on which $f$ is potentially Newton degenerate, the pair $(i_1-i_0, j_0-j_1)$ is one of $(2,2)$, $(2,4)$, $(2,6)$, $(2,8)$, or $(3,3)$, proving~(\ref{part:differences}).
\end{proof}

The Newton diagrams found in this section will be used in Section~\ref{sec:classifyA1Milnor} to find all elements of the Grothendieck--Witt ring of rank at most~$8$ that are $\A1$-Milnor numbers of two-variable polynomials.

\section{Change of variables}\label{sec:change} 

In this section, we prove that $\A1$-Milnor numbers are invariant under right equivalence and show that we can perform changes of variables on certain Newton degenerate polynomials to increase the Newton number while preserving the $\A1$-Milnor number. We use this to prove Theorem~\ref{thm:increaseNewton}, which shows that it suffices to consider Newton nondegenerate polynomials when finding all elements of the Grothendieck--Witt ring of rank at most~$8$ that arise as $\A1$-Milnor numbers of two-variable polynomials.

\subsection{Invariance under right equivalence}
Polynomial maps $f,g\colon\A{n}_k\to\A1_k$ are \emph{right equivalent} if there is a polynomial map $\Phi \colon\mathbb{A}_k^n\to \mathbb{A}_k^n$ with $\Phi(0)=0$ that has an inverse that is also a polynomial map such that $g=f\circ\Phi$. If $f$ and~$g$ are right equivalent, then it is not hard to see that the local algebras $Q_0(\nabla f)$ and $Q_0(\nabla g)$ are isomorphic. For completeness, we give a proof of this in the next two lemmas. 

\begin{lemma}\label{lemma:idealsequal}
Let $k$ be a field, let $f\colon\mathbb{A}_k^n\to\mathbb{A}_k^1$ be a polynomial map, and let $\Phi\colon \mathbb{A}_k^n\to\mathbb{A}_k^n$ be a polynomial map that has an inverse that is also a polynomial map. Let $g=f\circ\Phi$. We have 
\[(f_{x_1}(\Phi(\mathbf{x})),\dots,f_{x_n}(\Phi(\mathbf{x})))=(g_{x_1}(\mathbf{x}),\dots,g_{x_n}(\mathbf{x}))\]
as ideals in $k[x_1,\dots,x_n]$. 
\end{lemma}

\begin{proof}
We prove two inclusions. In the first direction, the chain rule gives $g_{x_i}(\mathbf{x})=\sum_{j=1}^n f_{x_j}(\Phi(\mathbf{x}))\frac{\partial\Phi_j}{\partial x_i}(\mathbf{x})$, so
\[(g_{x_1}(\mathbf{x}),\dots,g_{x_n}(\mathbf{x}))\subseteq(f_{x_1}(\Phi(\mathbf{x})),\dots,f_{x_n}(\Phi(\mathbf{x}))).\]
In the other direction, since $\Phi$ is invertible, we have $f=g\circ \Phi^{-1}$, so the chain rule similarly gives $(f_{x_1}(\mathbf{x}),\dots,f_{x_n}(\mathbf{x}))\subseteq(g_{x_1}(\Phi^{-1}(\mathbf{x})),\dots,g_{x_n}(\Phi^{-1}(\mathbf{x})))$. Substituting $\Phi(\mathbf{x})$ for~$\mathbf{x}$ gives
\[(f_{x_1}(\Phi(\mathbf{x})),\dots,f_{x_n}(\Phi(\mathbf{x})))\subseteq(g_{x_1}(\mathbf{x}),\dots,g_{x_n}(\mathbf{x})).\qedhere\]
\end{proof}

We can now show that $Q_0(\nabla f)$ and $Q_0(\nabla g)$ are isomorphic.

\begin{lemma}\label{lemma:localringisom}
Let $k$ be a field, let $f\colon\mathbb{A}_k^n\to\mathbb{A}_k^1$ be a polynomial map, and let $\Phi\colon \mathbb{A}_k^n\to\mathbb{A}_k^n$ be a polynomial map with $\Phi(0)=0$ that has an inverse that is also a polynomial map. Let $g=f\circ\Phi$. The automorphism $\phi\colon k[x_1,\dots,x_n]\to k[x_1,\dots,x_n]$ determined by $\phi(h)=h\circ\Phi$ induces an isomorphism of $k$-algebras $\overline{\phi}\colon Q_0(\nabla f)\to Q_0(\nabla g)$. 
\end{lemma}

\begin{proof}
The map~$\phi$ sends $(f_{x_1}(\mathbf{x}),\dots,f_{x_n}(\mathbf{x}))$ to $(f_{x_1}(\Phi(\mathbf{x})),\dots,f_{x_n}(\Phi(\mathbf{x})))$, which by Lemma~\ref{lemma:idealsequal} is contained in $(g_{x_1}(\mathbf{x}),\dots,g_{x_n}(\mathbf{x}))$. Therefore $\overline{\phi}$ is well defined. The map~$\Phi^{-1}$ induces a ring isomorphism $\phi^{-1}\colon k[x_1,\dots,x_n]\to k[x_1,\dots,x_n]$ that is the inverse of~$\phi$. Since $\phi^{-1}$ maps $(g_{x_1}(\mathbf{x}),\dots,g_{x_n}(\mathbf{x}))$ to $ (g_{x_1}(\Phi^{-1}(\mathbf{x})),\dots,g_{x_n}(\Phi^{-1}(\mathbf{x})))$, which by Lemma~\ref{lemma:idealsequal} is contained in $(f_{x_1}(\mathbf{x}),\dots,f_{x_n}(\mathbf{x}))$, and $\Phi^{-1}(0)=0$, the map $\overline{\phi^{-1}}\colon Q_0(\nabla g)\to Q_0(\nabla f)$ is well defined. Since $\phi$ and~$\phi^{-1}$ are inverses, it follows that $\overline{\phi}$ and~$\overline{\phi^{-1}}$ are inverses, so $\overline{\phi}$ is an isomorphism.
\end{proof}

Lemma~\ref{lemma:localringisom} gives the well-known fact that Milnor numbers are invariant under right equivalence for invertible polynomial maps. We now show that $\A1$-Milnor numbers are invariant under right equivalence for invertible polynomial maps.

\begin{theorem}\label{thm:rightequivalence}
Let $k$ be a field of characteristic not~$2$, let $f\colon\A{n}_k\to\A1_k$ be a polynomial map with $f(0)=0$ and $1\leq\mu_0(f)<\infty$, and let $\Phi\colon \mathbb{A}_k^n\to\mathbb{A}_k^n$ be a polynomial map with $\Phi(0)=0$ that has an inverse that is also a polynomial map. Define $g=f\circ\Phi$. We have $\mu^{\mathbb{A}^1}_0(f)= \mu^{\mathbb{A}^1}_0(g)$.
\end{theorem}

\begin{proof} 
Let $\phi\colon k[x_1,\dots,x_n]\to k[x_1,\dots,x_n]$ be the automorphism defined by $\phi(h)=h\circ\Phi$, and let $\overline{\phi}\colon Q_0(\nabla f)\to Q_0(\nabla g)$ be the induced isomorphism from Lemma~\ref{lemma:localringisom}. Since $\Phi^{-1}\circ\Phi$ is the identity map, the chain rule gives that $J_{\Phi^{-1}}(\Phi(\mathbf{x}))J_{\Phi}(\mathbf{x})=I$, and hence that $\det(J_{\Phi^{-1}}(\Phi(\mathbf{x})))\det(J_{\Phi}(\mathbf{x}))=1$. Thus $\det J_{\Phi}(\mathbf{x})$ is a unit in $Q_0(\nabla g)$; we define $D=\det J_{\Phi}(\mathbf{x})$. 

We will show the result by following the construction of the $\A1$-Milnor number for both $f$ and~$g$. By the chain rule, 
\[H_{g}(\mathbf{x})_{i,j}=\bigl(J^T_{\Phi}(\mathbf{x})H_{f}(\Phi(\mathbf{x}))J_{\Phi}(\mathbf{x})\bigr)_{i,j}+\sum_{\ell=1}^nf_{x_{\ell}}(\Phi(\mathbf{x}))\biggl(\frac{\partial^2\Phi_{\ell}}{\partial x_i\partial x_j}\biggr)(\mathbf{x}).\]
Since $f_{x_{\ell}}(\Phi(\mathbf{x})) = 0$ in $Q_0(\nabla g)$ for all $\ell\in[n]$ by Lemma~\ref{lemma:idealsequal}, we have 
\[H_{g}(\mathbf{x})=J^T_{\Phi}(\mathbf{x}) H_{f}(\Phi(\mathbf{x}))J_{\Phi}(\mathbf{x})\]
in $Q_0(\nabla g)$. Taking determinants, this means
\[\det H_{g}(\mathbf{x})=\det J_{\Phi}(\mathbf{x})^2\det H_{f}(\Phi(\mathbf{x}))=D^2\overline{\phi}(\det H_{f}(\mathbf{x}))\]
in $Q_0(\nabla g)$. Now suppose we have a $k$-linear functional $\eta_g\colon Q_0(\nabla g)\to k$ with $\eta_g(\det H_g(\mathbf{x}))=\mu_0(g)$. Define a map $\eta_f\colon Q_0(\nabla f)\to k$ by $\eta_f(v)=\eta_g(D^2\overline{\phi}(v))$ for all $v\in Q_0(\nabla f)$. This is a $k$-linear map as both $\eta_g$ and~$\overline\phi$ are $k$-linear and $D$ is a unit in $Q_0(\nabla g)$, and 
\[\eta_f(\det H_{f}(\mathbf{x}))=\eta_{g}(D^2\overline{\phi}(\det H_{f}(\mathbf{x})))=\eta_g(\det H_{g}(\mathbf{x}))=\mu_0(g)=\mu_0(f).\]
Therefore $\eta_f$ can be used to define $\mu^{\mathbb{A}^1}_0(f)$.

We now show that we can choose bases such that the Gram matrix of $\mu^{\mathbb{A}^1}_0(f)$ coming from~$\eta_f$ is the same as the Gram matrix of $\mu^{\mathbb{A}^1}_0(g)$ coming from~$\eta_g$. Let $\mathcal{B}=\{v_{\ell}(\mathbf{x})\}_{\ell=1}^r$ be a basis of $Q_0(\nabla f)$. Since $D$ is a unit in $Q_0(\nabla g)$, it follows by Lemma~\ref{lemma:localringisom} that $\mathcal{B'}\coloneqq\{D\overline\phi(v_{\ell}(\mathbf{x}))\}_{\ell=1}^r$ is a basis of $Q_0(\nabla g)$. Then
\[\eta_f(v_i(\mathbf{x}) v_j(\mathbf{x}))=\eta_g(D^2\overline\phi(v_i(\mathbf{x})v_j(\mathbf{x})))=\eta_g(D\overline\phi(v_i(\mathbf{x}))D\overline\phi(v_j(\mathbf{x})))\]
for all $i,j\in[r]$, meaning each entry of the Gram matrix of $\mu^{\mathbb{A}^1}_0(f)$ with respect to~$\mathcal{B}$ is the same as the corresponding entry of the Gram matrix of $\mu^{\mathbb{A}^1}_0(g)$ with respect to~$\mathcal{B}'$. Hence $\mu^{\mathbb{A}^1}_0(f)=\mu^{\mathbb{A}^1}_0(g)$.
\end{proof}

\subsection{Change of variables and Newton number}
We now turn our attention to how certain changes of variables affect a Newton polyhedron, with an eye toward increasing the Newton number of a Newton degenerate polynomial without changing the $\A1$-Milnor number. 

\begin{lemma}\label{lemma:nwshiftpoint}
Let $k$ be a field, let $a\in k^{\times}$, let $i_0,j_0\in\Z_{\geq 0}$, and let $m\in\Z_{>0}$. Let $f=x^{i_0} y^{j_0}\in k[x,y]$, and let $\Phi\colon\mathbb{A}_k^2\to\mathbb{A}_k^2$ be defined by $\Phi(x,y)= (x+ay^m, y)$. Define $g=f\circ \Phi\in k[x,y]$.
\begin{enumerate}[label=\textup{(\alph*)},ref=\textup{\alph*}]
\item The coefficient of the monomial $x^{i_0} y^{j_0}$ in~$g$ is~$1$. 
\item If a monomial $x^{i}y^{j}$ has nonzero coefficient in~$g$, then $(i, j)= (i_0-\ell,j_0+ m\ell)$ for some integer $\ell$ satisfying $0\leq \ell\leq i_0$. 
\end{enumerate}
\end{lemma}

\begin{proof}
We have
\[g(x, y)=f(x+ay^m, y)=(x+ay^m )^{i_0}y^{j_0}=x^{i_0} y^{j_0} + \sum_{\ell=1}^{i_0}\binom{i_0}{\ell}a^{\ell}x^{i_0-\ell}y^{j_0+m \ell}.\] 
Thus the monomial $x^{i_0}y^{j_0}$ has coefficient~$1$ in~$g$ and any monomial $x^{i}y^{j}$ that has nonzero coefficient in $g$ has $(i,j)=(i_0-\ell,j_0+ m\ell)$ for some integer $\ell$ satisfying $0\leq \ell\leq i_0$. 
\end{proof}

We now show that under certain conditions $\Npoly_+(g)$ is similar to $\Npoly_+(f)$ but with some well-controlled adjustments. An example of the situation we consider is shown in Figure~\ref{fig:nwshift}.

\begin{figure}
\begin{tikzpicture}[scale=0.4]
\usetikzlibrary{automata, arrows.meta, positioning}

\draw[->,semithick] (0,0)--(15,0);
\draw[->,semithick] (0,0)--(0,10);

\node[circle,draw=black,fill=black, label=right:{$(0,N)$},inner sep=0pt, minimum size=.15cm] (A) at (0,9) {};
\node[circle,draw=black,fill=black, label=right:{$(1,\delta)$},inner sep=0pt, minimum size=.15cm] (B) at (2,5) {};
\node[circle,draw=black,fill=black, label=above right:{$v$},inner sep=0pt, minimum size=.15cm] (C) at (4,3) {};
\node[circle,draw=black,fill=black, label=above right:{$v' $},inner sep=0pt, minimum size=.15cm] (D) at (8,1) {};
\node[circle,draw=black,fill=black, label=below:{$(M,0)$},inner sep=0pt, minimum size=.15cm] (E) at (14,0) {};
\node[circle,draw=black,fill=black, label=left:{$(0,\delta+m)$},inner sep=0pt, minimum size=.15cm] (F) at (0,7) {};

\draw[semithick] (A)--(B)--(C)--(D)--(E);
\draw[dashed] (7,0)--(F);
\end{tikzpicture}
\caption{Possible Newton diagram $\Npoly(f)$ in Lemma~\ref{lemma:nwshiftb}.\label{fig:nwshift}}
\end{figure}

\begin{lemma}
\label{lemma:nwshiftb}
Let $k$ be a field, let $a\in k^{\times}$, and let $m,\delta\in \Z_{>0}$. Let $f\in k[x,y]$ be a convenient polynomial, and let $e$ be an edge of $\Npoly(f)$ that contains the point $(1,\delta)$ and has slope at least~$-m$. Let $\Phi\colon\mathbb{A}_k^2\to\mathbb{A}_k^2$ be defined by $\Phi(x,y)= (x+ay^m, y)$, and let $g=f\circ \Phi$. If $(i,j)\in\supp g$, then either
\begin{enumerate}[label=\textup{(\alph*)},ref=\textup{\alph*}]
\item\label{lemma-case-i>0}$(i,j)\in \Npoly_+(f)$, or 
\item\label{lemma-case-i=0}$i=0$ and $j\geq \delta+m$.
\end{enumerate}
\end{lemma}

\begin{proof}
Let $(i_1,j_1)$ be the right endpoint of~$e$. Since $e$ has slope at least~$-m$, every edge of $\Gamma(f)$ to the right of $(i_1,j_1)$ has slope greater than~$-m$. If $(i,j)\in\supp g$, then by Lemma~\ref{lemma:nwshiftpoint} $(i,j)=(i_0,j_0)+ \ell(-1,m)$ for some $(i_0,j_0)\in \supp f$ and some nonnegative integer $\ell$. If $\ell=0$, then $(i,j)\in \Npoly_+(f)$, so (\ref{lemma-case-i>0}) is satisfied. We thus suppose $\ell\geq 1$. By the convexity of the Newton polyhedron $\Npoly_+(f)$, the slopes of the edges in $\Npoly(f)$ are negative and increase as the horizontal coordinate increases. Since the line $(i_0, j_0) + t(-1, m)$ has slope~$-m$ and $e$ has slope at least~$-m$, the point $(i,j)$ is contained in $\Gamma_+(f)$ as long as $i\geq 1$, in which case $\eqref{lemma-case-i>0}$ is satisfied. It remains to consider the case when $i=0$ and $(0,j)\not\in \Gamma_+(f)$. As $i=0$, we have $i_0=\ell\geq 1$ and $j=j_0+mi_0$. If $i_0=1$, then $(1, j_0)\in \Npoly_+(f)$. As $(1, \delta)\in \Npoly(f)$ by the hypotheses, we have $j_0\geq \delta$ and thus $j\geq \delta+m$. If $i_0\geq2$, then by convexity of $\Gamma_+(f)$ the slope of the line segment from $(1,\delta)$ to $(i_0,j_0)$ is at least~$-m$. Therefore we have
\[j_0-\delta\geq\frac{j_0-\delta}{i_0-1}\geq-m,\]
and hence $j=j_0+mi_0\geq \delta-m+mi_0\geq\delta+m$, meaning (\ref{lemma-case-i=0}) is satisfied.
\end{proof}

We are now ready to show that for certain Newton degenerate polynomials there is another polynomial with the same $\A1$-Milnor number but a larger Newton number. For a two-variable polynomial~$f$, if $e$ is an edge of the Newton diagram $\Npoly(f)$, let $f_e$ denote the polynomial consisting of the terms of~$f$ whose support lies on~$e$. 

\begin{proposition}\label{prop:increaseNewton}
Let $k$ be a field of characteristic not~$2$, let $a\in k^{\times}$, and let $m,\delta\in \Z_{>0}$. Let $f\in k[x,y]$ be a convenient polynomial with $f(0)=0$ and $1\leq\mu_0(f)<\infty$. If $e$ is an edge of $\Npoly(f)$ that contains the point $(1,\delta)$ and $f_e$ has $(x-ay^m)^2$ as a factor, then there exists another convenient polynomial $g\in k[x,y]$ with $g(0)=0$ and $1\leq\mu_0(g)<\infty$ such that $\mu^{\mathbb{A}^1}_0(g)=\mu^{\mathbb{A}^1}_0(f)$ and $\nug>\nuf$. 
\end{proposition}

\begin{proof}
To prove the proposition, we will explicitly construct $g$ using a change of variables $\Phi\colon\mathbb{A}_k^2\to\mathbb{A}_k^2$ defined by $\Phi(x,y)= (x+ay^m, y)$. We first define $h=f\circ \Phi$. Since $\Phi$ is invertible as a polynomial map, Theorem~\ref{thm:rightequivalence} gives that $\mu^{\A1}_0(h)= \mu^{\A1}_0(f)$. 

To prove the proposition, we will modify $h$ to find a convenient polynomial~$g$ with the same Milnor number and $\nu(g)>\nu(f)$. Let $(M,0)$ and $(0,N)$ be the points of $\Npoly(f)$ on the axes. By Lemma~\ref{lemma:nwshiftpoint}, we also have $(M,0)\in \Npoly(h)$. Also, the edge~$e$ contains $(1,\delta)$ has slope~$-m$ as $x-ay^m$ is a factor of~$f_e$. Hence since $(0, N)\in \Gamma(f)$, we have $N\geq \delta+m$. Moreover, the equality $N= \delta+m$, can only occur if $(0,N)$ lies on~$e$. By Lemma~\ref{lemma:nwshiftpoint}, the point $(0,j)$ is not in $\Gamma(h)$ for $j<\delta+m$. Moreover, this lemma shows that the points $(0,\delta+m)$ and $(1, \delta)$ are not in $\supp h$ as the assumption that $f_e$ is a multiple of $(x-ay^m)^2$ means that $f_e(x+ay^m,y)$ is a multiple of~$x^2$. Thus $(0, \delta+m)\not\in \Gamma(h)$.

We now define $g\in k[x,y]$ by letting $g=h$ if $h$ is convenient. If $h$ is not convenient, we define $g=h+y^{N_0}$, where $N_0=\max\{\mu_0(g)+2,\delta+m+1\}$. By Lemma~\ref{lemma:ne}, $\mu_0^{\A1}(g)=\mu_0^{\A1}(h)$. As a result, we have a convenient polynomial~$g$ with $\mu_0^{\A1}(g)=\mu_0^{\A1}(f)$ and a vertex $(0,N^\prime)\in \Gamma(g)$ with $N'> \delta+m$.

It remains to show that $\nug>\nuf$. We do this by creating a Newton diagram~$\Gamma'$ with the property that $\nu(\Gamma')=\nuf$ and $\nu(\Gamma')<\nug$. Define $\Gamma'$ by starting with $\Gamma(f)$ and replacing the line segment from $(0,N)$ to $(1,\delta)$ by a line segment from $(0,N')$ to $(1,\delta)$. Note that $\Gamma'$ is a valid Newton diagram because the edge from $(0,N')$ to $(1,\delta)$ has slope less than~$-m$ and the edge with $(1,\delta)$ as it's leftmost vertex has slope~$-m$. Since $\Gamma(f)$ and $\Gamma'$ contain the same points $(i,j)$ when $i\geq1$, the regions $\Gamma_-(f)$ and $\Gamma_-^\prime$ agree to the left of $i=1$. Hence we have that 
\[\Area(\Npoly_{-}(f))-\frac{N-\delta}{2}=\Area(\Npoly'_{-})-\frac{N'-\delta}{2}.\]
The formula for Newton number then gives $\nuf=\nu(\Gamma')$. 

We now show $\nu(\Gamma')<\nug$. We do this by first showing $\Gamma_{-}^\prime\subsetneq \Gamma_{-}(g)$. Then Lemma~\ref{lemma:newtonpolygonareas} gives $\nu(\Gamma')<\nug$ as $\Gamma_{-}$ and $\Gamma_{-}(g)$ are convenient Newton diagrams with the same vertices on the coordinate axes.

We now prove $\Gamma_{-}^\prime\subsetneq \Gamma_{-}(g)$. We have seen above that $(1,\delta)\not\in \supp h$. By Lemma~\ref{lemma:nwshiftpoint}, there is at least one vertex $(i,j)$ on~$e$ in $\supp h$; let $(i_0,j_0)$ be the one with minimal $i$-coordinate. We know $i_0\geq 2$. Then $(i_0,j_0)\in \Gamma(h)$ and all edges in $\Gamma(h)$ to the left (resp.\ right) of $(i_0,j_0)$ have slope less than (resp.\ at least) $-m$. Thus $\Gamma(f)$ and $\Gamma(h)$ coincide on the half-plane to the right of $(i_0, j_0)$ by Lemma~\ref{lemma:nwshiftpoint}. Hence $\Gamma^\prime$ and $\Gamma(g)$ coincide on this half plane. Figure~\ref{fig:increasenewton} illustrates this construction. 

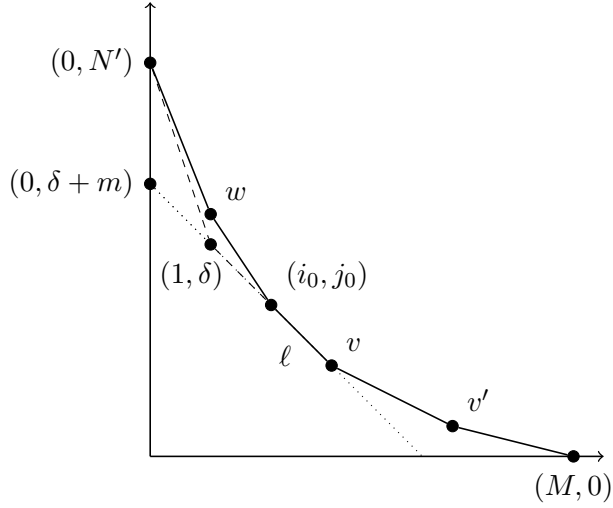
\begin{figure}
\begin{tikzpicture}[scale=0.4]
\usetikzlibrary{automata, arrows.meta, positioning}

\draw[->,semithick] (0,0)--(15,0);
\draw[->,semithick] (0,0)--(0,15);

\node[circle,draw=black,fill=black, label=below:{\hspace*{-.2in}$(1,\delta)$},inner sep=0pt, minimum size=.15cm] (B) at (2,7) {};
\node[circle,draw=black,fill=black, label=above right:{$(i_0, j_0)$},inner sep=0pt, minimum size=.15cm] (C) at (4,5) {};
\node[circle,draw=black,fill=black, label=above right:{$v$},inner sep=0pt, minimum size=.15cm] (C2) at (6,3) {};
\node[circle,draw=black,fill=black, label=above right:{$v' $},inner sep=0pt, minimum size=.15cm] (D) at (10,1) {};
\node[circle,draw=black,fill=black, label=below:{$(M,0)$},inner sep=0pt, minimum size=.15cm] (E) at (14,0) {};
\node[circle,draw=black,fill=black, label=left:{$(0,\delta+m)$},inner sep=0pt, minimum size=.15cm] (F) at (0,9) {};
\node[circle,draw=black,fill=black, label=left:{$(0,N^\prime)$},inner sep=0pt, minimum size=.15cm] (G) at (0,13) {};
\node[circle,draw=black,fill=black, label=above right:{$w$}, inner sep=0pt, minimum size=.15cm] (H) at (2,8) {};

\draw[dotted] (0,9)--(9,0);
\draw[dotted] (C) -- node[below left] {$\ell$} (C2);
\draw[dashed] (G)--(B)--(C);
\draw[semithick] (G)--(H)--(C)--(C2)--(D)--(E);
\draw[dotted] (F)--(B);
);
\end{tikzpicture}
\caption{A possible diagram in the proof of Proposition~\ref{prop:increaseNewton}. The edge $e\in \Gamma^\prime$ is from $(1,\delta)$ to $v$ and has slope~$-m$. The Newton polygons $\Gamma^\prime$ and $\Gamma(g)$ agree to the right of $(i_0, j_0)$. To the left of $(i_0, j_0)$, $\Gamma^\prime$ consists of the dashed edges from $(0, N^\prime)$ to $(1, \delta)$ to $(i_0, j_0)$ and
$\Gamma(g)$ consists of the solid edges from $(0, N^\prime)$ to~$w$ to $(i_0, j_0)$.
\label{fig:increasenewton}}
\end{figure}

Further, to the left of $(i_0, j_0)$, the edges of $\Gamma(g)$ have slope less than~$-m$. This follows since the points $(i,j)$ in $\supp g$ with $i<i_0$ lie strictly above the line~$\ell$ with slope~$-m$ that connects $(0, \delta+m$) and $(i_0, j_0)$. Hence in the half plane $i\leq i_0$, $\Gamma(g)$ consists of a path of line segments between integral vertices, starting at $(0, N^\prime)$ and ending at $(i_0, j_0)$, all with slope less than~$-m$. Thus, in the half plane $i<i_0$, $\Gamma(g)$ lies strictly above~$\ell$. Since $(1,\delta)\in \ell$, this means $(1,\delta)\not\in\Gamma_+(g)$. 

Note that $\Gamma^\prime$ contains the edges connecting the vertices $(0, N^\prime)$ to $(1,\delta)$ to $(i_0, j_0)$. As $(1, \delta)\not\in\Gamma_+(g)$, $\Gamma(g)$ lies strictly above $\Gamma^\prime$ on the interval $0<i<i_0$ and we have $\Gamma_+(g)\subsetneq \Gamma^\prime_+$. Hence $\Gamma'_{-} \subsetneq \Gamma_{-}(g)$, as needed to complete the proof.
\end{proof}

\begin{example}\label{ex:running3}
Let $k$ be a field of characteristic~$0$, and let $f = (x+y^{2})^{2} + xy^{3}$. In Example~\ref{ex:running}, we found $\mu_0^{\A1}(f)=2\hyp$. This is consistent with Example~\ref{ex:running2}, which shows that $f$ is degenerate on the edge from $(0,4)$ to $(2,0)$. The Newton number $\nuf$ is $2(4)-4-2+1=3$, meaning there is a strict inequality $\mu_0(f)>\nuf$.

We now show how to transform $f$ into a polynomial~$g$ that is Newton nondegenerate and has the same $\A1$-Milnor number. Letting $\Phi\colon\A2_k\to\A2_k$ be defined by $\Phi(x,y)=(x-y^2,y)$, we have $g=f\circ\Phi = x^{2} + xy^{3} - y^5$. The Newton diagram $\Npoly(g)$ consists of the edge from $(2,0)$ to $(0,5)$, which is convenient and Newton nondegenerate by Lemma~\ref{lemma:onboundary}(\ref{lemma:onboundary-a}). We have $\nug = 10 - 2 - 5 + 1 = 4=\mu_0( g)$, consistent with Theorem~\ref{thm:kushnirenko}. Furthermore, $\mu_0^{\A1}(f)=\mu_0^{\A1}(g)$ by Theorem~\ref{thm:rightequivalence} since $\Phi$ is invertible with $\Phi^{-1}(x,y)=(x+y^2, y)$.
\end{example}

\subsection{Restriction to Newton nondegenerate case for rank at most \texorpdfstring{$8$}{8}}

In order to prove Theorem~\ref{thm:increaseNewton}, which shows that it suffices to restrict attention to Newton nondegenerate polynomials in finding all Grothendieck--Witt classes of rank at most~$8$ that arise as $\A1$-Milnor numbers of two-variable polynomials, we need some preliminary results about polynomials. We begin with the following lemma.

\begin{lemma}\label{lemma:g_irred}
Let $k$ be a field, let $a\in k^\times$, let $\ell,m\in \Z_{>0}$ with $\gcd(\ell,m)=1$, and let $g=x^\ell-ay^m$.
\begin{enumerate}[label=\textup{(\alph*)},ref=\textup{\alph*}]
\item The polynomial~$g$ is irreducible in $k[x]$. \label{part:irreducible}
\item If $g(x_0, y_0)=0$ for some $(x_0, y_0)\in \bigl(\overline{k}^\times\bigr)^2$, then the zero set of~$g$ in~$\overline{k}^2$ is 
$\{(\lambda^{m} x_0, \lambda^{\ell}y_0)\mid \lambda \in \overline k\}$.\label{part:allzeros}
\end{enumerate}
\end{lemma}

\begin{proof}
We first show~(\ref{part:irreducible}). Suppose $g$ factors as the product of polynomials $h_1$ and~$h_2$. Without loss of generality, assume $h_1$ has positive degree. Since $g$ is quasi-homogeneous with weights $m$ and~$\ell$, the polynomial~$h_1$ must also be quasi-homogeneous with the same weights, Then $h_1$ must contain monomials $x^{\ell_1}$ and $y^{m_1}$ with nonzero coefficients, where $1\leq \ell_1\leq \ell$ and $1\leq m_1\leq m$. As $h_1$ is quasi-homogeneous with weights $m$ and $\ell$, we have $m \ell_1=\ell m_1$. As $\gcd(\ell, m)=1$, this implies $\ell\mid \ell_1$ and $m\mid m_1$, so $\ell=\ell_1$ and $m=m_1$. Therefore the degree of $h_2$ is~$0$, so $g$ is irreducible. 

We next show~(\ref{part:allzeros}). Since $g(x_0,y_0)=0$, we have $g(\lambda^mx_0,\lambda^\ell y_0)=\lambda^{\ell m}g(x_0,y_0)=0$ for all $\lambda\in \overline{k}$. If $g(x_0,y_0)=0$ and $g(x_1,y_1)=0$, then $x_0^\ell=ay_0^m$ and $x_1^\ell=ay_1^m$, so $(x_1/x_0)^{\ell} = (y_1/y_0)^m$. Letting $\lambda =(x_1/x_0)^{1/m }= (y_1/y_0)^{1/\ell}$, we have $(x_1, y_1) = (\lambda^{m} x_0, \lambda^{\ell} y_0)$. 
\end{proof}

We next provide a proof of the following proposition as we found no specific reference in the literature. 

\begin{proposition}\label{prop:resultantfactor}
Let $k$ be a field, and let $f\in k[x,y]$ be a quasi-homogeneous polynomial with positive weights $w_x$ and~$w_y$ for $x$ and~$y$ and weighted degree $d$. Suppose $d\neq0$ in $k$. If $f_x$ and~$f_y$ have a common zero in $\bigl(\overline{k}^{\times}\bigr)^2$, then $f$ has $(x^\ell-ay^m)^2$ as a factor in $\overline k[x,y]$ for some positive integers $\ell$ and~$m$ and some $a\in \overline{k}^\times$.
\end{proposition}

\begin{proof}
Let $(x_0,y_0)\in \bigl(\overline{k}^{\times}\bigr)^2$ be a common zero of $f_x$ and~$f_y$. By Euler's theorem for quasi-homogeneous polynomials \cite[Section~I.5, Exercise~3]{Kunz}, $d \cdot f= w_x x f_{x} + w_y y f_{y}$. Since $d\neq0$ in~$k$, the vanishing of $f_x$ and~$f_y$ at $(x_0, y_0)$ implies that $f$ vanishes at $(x_0,y_0)$ as well. By definition of quasi-homogeneous polynomial, we have
\[f(\lambda^{w_x} x, \lambda^{w_y}y)=\lambda^d f(x, y).\]
Similarly, since $f_x$ (resp.~$f_y$) is a quasi-homogeneous polynomial with the same weights $w_x$ and~$w_y$ and degree $d-w_x$ (resp.\ $d-w_y$), we have
\[\begin{aligned}
f_x(\lambda^{w_x} x, \lambda^{w_y}y)&=\lambda^{d-w_x} f_x(x, y)\\
f_y(\lambda^{w_x} x, \lambda^{w_y}y)&=\lambda^{d-w_y} f_x(x, y).
\end{aligned}\]
Therefore $f$, $f_x$, and~$f_y$ vanish on the parametric curve $(\lambda^{w_x} x_0, \lambda^{w_y}y_0)$ for all $\lambda\in \overline{k}$.

Let $\ell=w_y/\gcd(w_x, w_y)$ and $m=w_x/\gcd(w_x, w_y)$. Note that $w_x$ and~$w_y$ cannot both be~$0$ in~$k$ as otherwise $d$ would be~$0$ in~$k$, so therefore either $\ell$ or~$m$ is nonzero in~$k$. Let $g\coloneqq x^{\ell}-a y^m$, where $a={x_0^\ell}/{y_0^m}$. By construction, $g(x_0,y_0)=0$. By Lemma~\ref{lemma:g_irred}, $g$ is an irreducible polynomial and $f$, $f_x$, and~$f_y$ vanish on the zero set of~$g$ in~$\overline{k}^2$. This implies that $g$ divides $f$, $f_x$, and~$f_y$ as polynomials in $\overline k[x,y]$. Therefore $f=g^q h$ for some $q\geq 1$ and $h\in \overline k[x,y]$ with $\gcd(g,h)=1$. 

It remains to show that $q\geq 2$. Suppose $q=1$. Since $d\neq 0$ in~$k$, one of $\ell$ and~$m$ must be nonzero in~$k$. Suppose $\ell\neq0$ in~$k$. As $f_x=gh_x + g_xh$ and $g\mid f_x$, we have that $g\mid g_x h$. As $\overline k[x,y]$ is a unique factorization domain and $g$ and~$h$ are relatively prime, this means $g\mid g_x$. However, since $\ell\neq 0$, the polynomial $g_x=\ell x^{\ell-1}$ is nonzero with $\deg g_x <\deg g$, so this is impossible. A similar argument with~$f_y$ shows that we cannot have $q=1$ when $m\neq0$. Therefore $q\geq2$.
\end{proof}

In the case of interest to us, the factor $(x - ay^m)^2$ turns out to be in $k[x,y]$.

\begin{corollary}\label{cor:factorization}
Let $k$ be a field of characteristic not $2$ or~$3$, and let $f\in k[x,y]$ be a quasi-homogeneous polynomial with positive weights $w_x$ and~$w_y$ for $x$ and~$y$ and weighted degree~$d$. Let $d_x$ be the degree of $f(x,1)$ considered as a polynomial in~$x$. Assume that $f_x$ and~$f_y$ have a common zero in $\bigl(\overline k^{\times}\bigr)^2$. If $d\neq 0$ in $k$ and $d_x\leq 3$, then $f$ has $(x - ay^m)^2$ as a factor in $k[x,y]$ for some $a\in k^\times$ and some positive integer~$m$.
\end{corollary}

\begin{proof}
By Proposition~\ref{prop:resultantfactor}, $f$ has $(x^{\ell} - ay^m)^2$ as a factor in $\overline k[x,y]$ for some $a\in \overline{k}^\times$ and some $\ell,m\in\mathbb{Z}_{>0}$. Let $g=f(x,1)\in k[x]$, which then has $(x^{\ell} - a)^2$ as a factor in $\overline{k}[x]$. The hypothesis $d_x\leq 3$ shows $\ell =1$, so $g$ has $(x-a)^2$ as a factor in~$\overline{k}[x]$. 

We now show that $a\in k$. If $a\notin k$, then, since $\charac k\neq 2$, we would have $(x-a)^2 \notin k[x]$. As $d_x\leq 3$, this would mean $g$ is an irreducible cubic in $k[x]$ and hence a scalar multiple of the minimal polynomial of~$a$. Since $a$ is a repeated root of~$g$, it is also a root of~$g_x$, so $g_x$ would be divisible by~$g$. Since $\deg g_x<\deg g$, we would then have $g_x=0$. But this is impossible as~$g_x$ is a nonzero quadratic polynomial since $\charac k\neq 3$. Hence $a\in k$, as needed. 
\end{proof}

The final lemma of this section shows that, once we have made a polynomial Newton degenerate on each edge of its Newton diagram, we can also make it Newton nondegenerate on the vertices of the Newton diagram.

\begin{lemma}\label{lemma:pointdegeneracy}
Let $k$ be a field of characteristic not $2$, $3$, $5$, or~$7$. If $f\in k[x,y]$ with $\mu_0(f)\leq8$ that is convenient and Newton nondegenerate on each edge of $\Gamma(f)$, then there is a convenient polynomial $g\in k[x,y]$ that is Newton nondegenerate and for which $\mu_0^{\A1}(g)=\mu_0^{\A1}(f)$. 
\end{lemma}

\begin{proof}
When $\charac k=0$, Lemma~\ref{lemma:atvertex} gives that the vertices of $\Gamma(f)$ are Newton nondegenerate, so we can take $g=f$. In the case of positive characteristic, the vertices of~$f$ might be degenerate and we may need to modify $f$ to obtain~$g$. By hypothesis, assume $\charac k>7$. Since $\mu_0(f)\leq8$, we know from Theorem~\ref{thm:kushnirenko} that $\nuf\leq 8$. Therefore $\Gamma(f)$ must be one of the Newton diagrams found in Section~\ref{sec:classifynewton}. Since $\charac k>7$, Lemma~\ref{lemma:atvertex} shows that $f$ could only be Newton degenerate at vertices of the form $(M,0)$ or $(0,N)$ with $M,N\geq11$. Assume $f$ is Newton degenerate at $(M,0)$. The classification of Newton diagrams in Propositions \ref{prop:newton-one-edges}, \ref{prop:newton-two-edges}, and~\ref{prop:newton-three-edges} shows that $\Gamma(f)$ must have an edge with endpoints $(\gamma,1)$ and $(M,0)$. 

Since $M>\mu_0(f)+2$, Lemma~\ref{lemma:ne} shows that we can replace the monomial term~$x^M$ with~$x^{M'}$ where $M'\geq \mu_0(f)+2$ and $\charac k\nmid M^\prime$ without changing the $\A1$-Milnor number. As $M'\neq0$ in~$k$, the resulting polynomial is Newton nondegenerate at $(M',0)$. Moreover, Lemma~\ref{lemma:onboundary}(\ref{lemma:onboundary-a}) shows that $f$ is Newton nondegenerate on the new edge from $(\gamma,1)$ to $(M',0)$. If $f$ is degenerate at $(0,N)$, we modify~$N$ similarly to obtain a convenient, Newton nondegenerate polynomial~$g$ with $\mu_0^{\A1}(g)=\mu_0^{\A1}(f)$. 
\end{proof}

We are now ready to prove the main result of this section.

\begin{theorem}\label{thm:increaseNewton}
Let $k$ be a field of characteristic not $2$, $3$, $5$, or~$7$. Let $f$ be a convenient polynomial in two variables with $f(0)=0$ and $1\leq\mu_0(f)<\infty$ that is Newton degenerate on an edge~$e$ of $\Gamma(f)$.
\begin{enumerate}[label=\textup{(\alph*)},ref=\textup{\alph*}]
\item If $\nuf\leq8$, then there is a convenient polynomial~$g$ in two variables with $\mu^{\mathbb{A}^1}_0(g)=\mu^{\mathbb{A}^1}_0(f)$ and $\nug>\nuf$. \label{part:increaseNewtona}
\item If $\mu_0(f)\leq8$, then there is a Newton nondegenerate and convenient polynomial~$g$ in two variables with $\mu_0^{\A1}(g)=\mu_0^{\A1}(f)$. \label{part:increaseNewtonb}
\end{enumerate}
\end{theorem}

\begin{proof}
We first prove~(\ref{part:increaseNewtona}). Since $\nuf\leq8$ and $f$ is Newton degenerate on an edge~$e$ of $\Gamma(f)$, after interchanging variables if necessary $e$ must be one of the edges found in Proposition~\ref{prop:degenerateedge}(\ref{part:degenerateedge}). Choose relatively prime weights $w_x$ and~$w_y$ for the quasi-homogeneous polynomial~$f_e$. Examining the list of edges in Proposition~\ref{prop:degenerateedge}(\ref{part:degenerateedge}), we see that the only possible prime divisors of the quasi-homogeneous degree of $f_e$ are $2$, $3$, $5$, and~$7$. Moreover, by Proposition~\ref{prop:degenerateedge}(\ref{part:differences}), the change in the horizontal coordinate along~$e$ is at most~$3$. Corollary~\ref{cor:factorization} then gives that $(x-a y^m)^2$ is a factor of~$f_e$ for some $a\in k^\times$. Together with the fact that each edge in Proposition~\ref{prop:degenerateedge}(\ref{part:degenerateedge}) passes through a point $(1,\delta)$, this allows us to apply Proposition~\ref{prop:increaseNewton} to obtain a convenient polynomial~$g$ in two variables with $\mu^{\mathbb{A}^1}_0(g)=\mu^{\mathbb{A}^1}_0(f)$ and $\nug>\nuf$. 

We next prove~(\ref{part:increaseNewtonb}). Since $\mu_0(f)\leq 8$, by Theorem~\ref{thm:kushnirenko} we have $\nuf\leq 8$. By~(\ref{part:increaseNewtona}), there exists a convenient polynomial~$f_1$ in two variables with $\mu_0^{\A1}(f_1)=\mu_0^{\A1}(f)$ and $\nu(f_1)> \nuf$. If $f_1$ is Newton degenerate on an edge of $\Gamma(f_1)$, then we can repeat this process by applying part~(\ref{part:increaseNewtona}) to find a convenient polynomial~$f_2$ in two variables with $\mu_0^{\A1}(f_2)=\mu_0^{\A1}(f_1)$ and $\nu(f_2)> \nu(f_1)$. Since each step increases the Newton number by at least~$1$ and the Newton numbers are bounded above by $\mu_0(f)$ according to Theorem~\ref{thm:kushnirenko}, this process must terminate, meaning there is a sequence of convenient polynomials $f_1,f_2,\dots,f_q$ in two variables with $f_q$ Newton nondegenerate on each edge of $\Gamma(f_q)$ and $\mu_0^{\A1}(f_q)=\dots=\mu_0^{\A1}(f_1)=\mu_0^{\A1}(f)$. Applying Lemma~\ref{lemma:pointdegeneracy} to~$f_q$ gives a $g$ that is Newton nondegenerate on all edges and vertices of $\Gamma(g)$, showing~(\ref{part:increaseNewtonb}).
\end{proof}

\section{\texorpdfstring{$\A1$}{A\textonesuperior}-Milnor numbers of rank at most \texorpdfstring{$8$}{8} for two-variable polynomials}\label{sec:classifyA1Milnor} 

In this section, we determine the elements of the Grothendieck--Witt ring of rank at most~$8$ that are $\A1$-Milnor numbers of polynomials in two variables. Here we work over a field of characteristic not $2$, $3$, $5$, or~$7$. By Theorem~\ref{thm:allarise}(\ref{part:allariseeasyranks}), all local $\A1$-Brouwer degrees of ranks $1$, $2$, $3$, $5$, and~$7$ arise from a polynomial in two variables, so we only need to consider the cases when $\mu_0(f)$ is $4$, $6$, and~$8$. 

By Lemma~\ref{lemma:ne}, it suffices to consider polynomials that are convenient. By Theorem~\ref{thm:kushnirenko}, it suffices to consider polynomials whose Newton number is at most~$8$, i.e.\ polynomials whose Newton diagram is one of those found in Section~\ref{sec:classifynewton}. By Theorem~\ref{thm:increaseNewton}, it suffices to consider polynomials that are convenient and Newton nondegenerate. We then know from Theorem~\ref{thm:kushnirenko} that $\mu_0(f)=\nuf $, so it suffices to consider the Newton diagrams from Section~\ref{sec:classifynewton} where $\nuf$ is $4$, $6$, or~$8$. We divide these into four cases, to be addressed in the following subsections: 
\begin{enumerate}[label=\textup{(\roman*)},ref=\textup{\roman*}]
\item $\Npoly(f)$ has one edge connecting $(2,0)$ and $(0,r+1)$;
\item $\Npoly(f)$ has one edge connecting $(3,0)$ and $(0,r/2+1)$, and no intermediate points on this edge are in $\supp f$;
\item $\Npoly(f)$ has one edge connecting $(3,0)$ and $(0,r/2+1)$, and one or more intermediate points on this edge are in $\supp f$ (which only occurs for $r=4$); and
\item $\Npoly(f)$ has more than one edge.
\end{enumerate}

We compute the $\A1$-Milnor number of all convenient and Newton nondegenerate polynomials whose Newton diagram is of any of these forms. We first describe our general approach, then treat each of these cases separately.

\subsection{General approach}
Given a polynomial~$f$ in two variables, we compute the $\A1$-Milnor number $\mu_0^{\mathbb{A}^1}(f)$ as follows. We first recall a result from Quick, Strand and Wilson~\cite{QuickStrandWilson}, restated in our context, that we will use to find a finite spanning set for $Q_0(\nabla f)$. 

\begin{lemma}[{\cite[Lemma~2.8]{QuickStrandWilson}}]\label{lemma:monomialszero}
Let $k$ be a field of characteristic not~$2$, and let $r$ be a positive integer. If $f\colon\mathbb{A}_k^n\to\mathbb{A}_k^1$ is a polynomial map with $f(0)=0$ and $\mu_0(f)=r$, then $(x_1,\dots,x_n)^{r}\subseteq (f_{x_1},\dots,f_{x_n})$.
\end{lemma}

If $\mu_0(f)=r$, Lemma~\ref{lemma:monomialszero} implies that $Q_0(\nabla f)$ is spanned by the monomials in the set
\[S\coloneqq\{x^iy^j\mid 0\leq i+j< r\}.\] 
By definition, $f_x$ and~$f_y$ are zero in $Q_0(\nabla f)$, so for all $\alpha\in S$ we have
\begin{equation}\label{eq:system} 
\alpha f_x =0\quad\text{and}\quad \alpha f_y=0.
\end{equation}
We regard these equations as a system of linear equations in $Q_0(\nabla f)$ in the monomials in~$S$ and with coefficients in~$k$. All other relations in $Q_0(\nabla f)$ are $k$-linear combinations of those in this system. Row reduction of the system~\eqref{eq:system} gives a basis of $Q_0(\nabla f)$ as a vector space over~$k$ and expressions of all $\alpha\in S$ in terms of this basis. We then use this information to compute $\mu^{\A1}_0(f)$, as illustrated in the following example. 

\begin{example}
This is a continuation of Example~\ref{ex:running3}. Let $g = x^{2} + xy^{3} - y^5$, which is clearly convenient. The Newton diagram $\Npoly(g)$ consists of the edge from $(2,0)$ to $(0,5)$, so Lemma~\ref{lemma:onboundary}(\ref{lemma:onboundary-a}) gives that $g$ is Newton nondegenerate. Since $\mu_0(g)=4$, we know that $Q_{0}(\nabla g)$ is spanned by monomials with degree at most~$3$, namely the $10$~monomials in
\[ S \coloneqq \{ 1, x, y, x^{2}, xy, y^{2}, x^{3}, x^{2}y, xy^{2}, y^{3} \}. \]
We know that $g_{x} = 2x+y^3$ and $g_y= 3xy^2-5y^4$ are both~$0$ in $Q_{0}(\nabla g)$. Similarly, $\alpha g_{x}=0$ and $\alpha g_{y} = 0$ for all $\alpha \in S$. The equation $xg_x=0$ gives $2x^2+xy^3=0$. Since $xy^3$ has degree~$4$, it is $0$ in $Q_{0}(\nabla g)$, so $x^2=0$. There is a system of $20$~equations in $10$~unknowns (the elements of~$S$). Solving this system by row reduction shows that $\{1, y, y^2, y^3\}$ is a basis of $Q_{0}(\nabla g)$. It also gives $x=-y^3/2$ and $x^2=xy=x^3=x^2y=xy^2=0$ in $Q_{0}(\nabla g)$. 

The Hessian determinant is $\det H_g= 12xy - 40y^{3} - 9y^{4}= -40y^3$ in $Q_{0}(\nabla g)$. We can define $\eta\colon Q_{0}(\nabla g)\to k$ by $\eta(y^{3}) = -1/10$ and $\eta(v)=0$ for the other basis elements~$v$ so that $\eta(\det H_g)=4$, which is the Milnor number as required. The Gram matrix is thus an antidiagonal matrix with $\eta(y^3)= -1/10$ along the antidiagonal. Proposition~\ref{prop:antidiagonal} then gives that $\mu_0^{\A1}(g)=2\hyp$, which agrees with Example~\ref{ex:running}.
\end{example}

\subsection{Case (i)}\label{subsec:casei}
Let $r\in\{4,6,8\}$, and suppose $\Npoly(f)$ consists of a single edge connecting $(2,0)$ and $(0,r+1)$. As $f$ is Newton nondegenerate, Theorem~\ref{thm:kushnirenko} gives 
\[\mu_0(f)=\nuf=2\biggl(\frac{2(r+1)}{2}\biggr)-2-(r+1)+1=r.\]
Using Lemma~\ref{lemma:ne}, we can therefore assume $f$ has the form
\[f=\sum_{\substack{i,j\\2\leq i+j\leq r+1}}a_{i,j}x^iy^j,\]
where 
\begin{itemize}
\item $a_{2,0}\neq0$, 
\item $a_{0,r+1}\neq 0$,
\item $a_{1,j}=0$ for $j< (r+1)/2$, and
\item $a_{0,j}=0$ for $j<r+1$.
\end{itemize}
Row reduction of the equations in~\eqref{eq:system} shows that
\begin{itemize}
\item $\{ 1, y, y^2, \dots, y^{r-1} \}$ is a basis of $Q_0(\nabla f)$; 
\item for all nonnegative integers~$j$, we have that $xy^j$ is a $k$-linear combination of $y^{r/2+j+1},\dots,y^{r-1}$; 
and
\item $\alpha=0$ in $Q_0(\nabla f)$ for all other $\alpha\in S$.
\end{itemize}
A direct calculation gives $\det H_f=2r(r+1)a_{2,0}a_{0,r+1}y^{r-1}$ in $Q_0(\nabla f)$. We can define $\eta\colon Q_0(\nabla f)\to k$ by $\eta(y^{r-1})=b$, where $b=1/(2(r+1)a_{2,0}a_{0,r+1})$, and $\eta(v)=0$ for all other basis elements~$v$. Then $\eta(\det H_f)=r$, the Milnor number. The Gram matrix is thus an antidiagonal matrix with $b$ in each antidiagonal position, so Proposition~\ref{prop:antidiagonal} gives $\mu_0^{\A1}(f)=\frac{r}{2}\hyp$.

\subsection{Case (ii)}\label{subsec:caseii}
Let $r\in\{4,6,8\}$, and suppose $\Npoly(f)$ consists of one edge connecting $(3,0)$ and $(0,r/2+1)$. As $f$ is Newton nondegenerate, Theorem~\ref{thm:kushnirenko} then gives 
\[\mu_0(f)=\nuf=2\biggl(\frac{3(r/2+1)}{2}\biggr)-3-(r/2+1)+1=r.\]
Using Lemma~\ref{lemma:ne}, we can then assume $f$ has the form 
\[f=\sum_{\substack{i,j\\3\leq i+j\leq r+1}}a_{i,j}x^iy^j,\]
where 
\begin{itemize}
\item $a_{3,0}\neq0$, 
\item $a_{0,r/2+1}\neq0$,
\item $a_{2,j}=0$ for $j\leq(r+2)/6$,
\item $a_{1,j}=0$ for $j\leq(r+2)/3$, and
\item $a_{0,j}=0$ for $j<(r+2)/2$.
\end{itemize}
Row reduction of the equations in~\eqref{eq:system} shows that
\begin{itemize}
\item $\{ 1, y, x, \dots, y^{r/2-1}, xy^{r/2-2}, xy^{r/2-1} \}$ is a basis of $Q_0(\nabla f)$, and
\item $x^iy^j=0$ in $Q_0(\nabla f)$ for all $x^iy^j\in S$ with $(r/2+1)i+3j\geq 2r-2$
and $(i,j)\neq(1,r/2-1)$.
\end{itemize}
A direct calculation gives $\det H_f=3r(r/2+1)a_{3,0}a_{0,r/2+1}xy^{r/2-1}$ in $Q_0(\nabla f)$. We can define $\eta\colon Q_0(\nabla f)\to k$ by $\eta(xy^{r/2-1})=b$, where $b=1/(3(r/2+1)a_{3,0}a_{0,r/2-1})$, and $\eta(v)=0$ for all other basis elements~$v$. Then $\eta(\det H_f)=r$, the Milnor number. The Gram matrix is thus an upper anti-triangular matrix
\[\begin{bmatrix}
* & \cdots & * & b\\
\vdots & \iddots & \iddots & 0\\
* & \iddots & \iddots & \vdots\\
b & 0 & \cdots & 0
\end{bmatrix},\]
with $b$ in each antidiagonal position, so Proposition~\ref{prop:antidiagonal} gives $\mu_0^{\A1}(f)=\frac{r}{2}\hyp$.

\subsection{Case (iii)}\label{subsec:caseiii}
Here we only need to consider the case when $r=4$, $\Npoly(f)$ consists of one edge connecting $(3,0)$ and $(0,3)$, and there is an intermediate point on this edge that is in $\supp f$. Since we are assuming that $f$ is Newton nondegenerate, Theorem~\ref{thm:kushnirenko} gives 
\[\mu_0(f)=\nuf=2\biggl(\frac{9}{2}\biggr)-3-3+1=4.\]
Using Lemma~\ref{lemma:ne}, we can therefore assume $f$ has the form 
\begin{equation}\label{eq:caseiiiassumptions}
f=\sum_{\substack{i,j\\3\leq i+j\leq 5}}a_{i,j}x^iy^j,
\end{equation}
where 
\begin{itemize}
\item $a_{3,0}\neq0$,
\item $a_{0,3}\neq0$, and
\item either $a_{2,1}\neq0$ or $a_{1,2}\neq0$.
\end{itemize}
Let $e$ be the edge of $\Gamma(f)$ from $(0,3)$ to $(3,0)$. Then
\[f_e= a_{3,0}x^3+ a_{2,1} x^2 y + a_{1,2} x y^2 + a_{0,3}y^3.\]
As we assume $f$ is Newton nondegenerate, $\frac{\partial f_e}{\partial x}$ and $\frac{\partial f_e}{\partial y}$ do not have a common zero in $\bigl(\overline{k}^{\times}\bigr)^2$. Moreover, they cannot have a nontrivial common zero with $x=0$ or $y=0$ since $a_{3,0}\neq0$ and $a_{0,3}\neq 0$. Therefore the resultant of the homogeneous polynomials $\frac{\partial f_e}{\partial x}$ and $\frac{\partial f_e}{\partial y}$ is nonzero (see \cite[Section~3.1]{CoxLittleOShea}). Let
\[\begin{split}
R&=\Res\biggl(\frac{\partial f_e}{\partial x}, \frac{\partial f_e}{\partial y}; x, y\biggr)\\
&=3(27a_{3,0}^2a_{0,3}^2 - 18a_{3,0}a_{2,1}a_{1,2}a_{0,3} + 4a_{2,1}^3a_{0,3} + 4a_{3,0}a_{1,2}^3 - a_{2,1}^2a_{1,2}^2).
\end{split}\]
Let $A=2a_{2,1}^2 - 6a_{3,0}a_{1,2}$, $B=2a_{1,2}^2 - 6a_{2,1}a_{0,3}$, and $C= 9a_{3,0}a_{0,3} - a_{2,1}a_{1,2}$. Algebraic manipulation gives
\[R=\begin{cases}
\bigl( a_{1,2}^2 A^2 - A^2 B - 2 a_{2,1}^2 AB +9 a_{3,0}^2B^2 \bigr)/\bigl(4 a_{2,1}^2\bigr) & \text{ if }a_{2,1}\neq 0,\\
\bigl( a_{2,1}^2 B^2 - B^2 A - 2 a_{1,2}^2 BA +9 a_{0,3}^2A^2 \bigr)/\bigl(4 a_{2,1}^2\bigr) & \text{ if }a_{1,2}\neq 0.
\end{cases}\]
Since either $a_{2,1}$ or $a_{1,2}$ is nonzero, we have that $A\neq0$ or $B\neq0$. By symmetry, we can assume $A\neq0$. 

Row reduction of the equations in~\eqref{eq:system} (using only the assumptions $A\neq0$, $R\neq 0$, and the conditions following~\eqref{eq:caseiiiassumptions} above) shows
\begin{itemize}
\item $\{ 1, x, y, y^2 \}$ is a basis of $Q_0(\nabla f)$, 
\item $x^2=\frac{B}{A}y^2$,
\item $xy=\frac{C}{A}y^2$, and
\item $\alpha=0$ in $Q_0(\nabla f)$ for all other $\alpha\in S$.
\end{itemize}
A direct calculation gives $\det H_f=\frac{4R}{A}y^2$ in $Q_0(\nabla f)$. We can define $\eta\colon Q_{0}(\nabla f)\to k$ by $\eta(y^2)=A/R$, and $\eta(v)=0$ for all other basis elements~$v$. Then $\eta(\det H_f)=4$, the Milnor number. Using the ordering $\{ 1, y, x, y^2\}$ of the basis, the corresponding Gram matrix is
\[\begin{bmatrix}
0 & 0 & 0 & A/R\\
0 & A/R & C/R & 0\\
0 & C/R & B/R & 0\\
A/R & 0 & 0 & 0
\end{bmatrix}.\]
Using Lemma~\ref{lemma:antidiagonal}, we have $\mu_0^{\A1}(f)=\hyp+ \beta_P$, where
\[P=\begin{bmatrix}
A/R & C/R\\
C/R & B/R
\end{bmatrix}.\]
As $\beta_P$ represents $A/R$ and $\det P=-1/R$, we have 
\[\beta_P\cong\biggl\langle\frac{\det P}{A/R}\biggr\rangle\cong\biggl\langle -\frac{1}{A}, \frac{A}{R} \biggr\rangle \cong \langle -A,AR \rangle.\] 
Therefore 
\[\mu_0^{\A1}(f)=\hyp+ \langle -A,AR \rangle.\]

\subsection{Case (iv)}\label{subsec:caseiv}
Let $r\in\{4,6,8\}$. Suppose $f$ is a Newton nondegenerate polynomial whose Newton diagram $\Npoly(f)$ has an edge~$e$ passing through $(2,1)$ and $(1,r/2)$ and has two or three edges. Since $f$ is Newton nondegenerate, Theorem~\ref{thm:kushnirenko} gives 
\[\mu_0(f)=\nuf=2\biggl(\frac{N}{2}+\frac{r/2-1}{2}+\frac{r/2-1}{2}+\frac{1}{2}+\frac{M}{2}\biggr)-M-N+1=r.\]
Using Lemma~\ref{lemma:ne}, we can then assume $f$ has the form 
\begin{equation}\label{eq:caseivassumptions}
f=\sum_{\substack{i,j\\3\leq i+j\leq r+1}}a_{i,j}x^iy^j,
\end{equation}
where 
\begin{itemize}
\item $a_{3,0}=0$ if $r=4$,
\item $a_{2,1}\neq0 $,
\item $a_{1,r/2}\neq0$ or $a_{0,r-1}\neq0$,
\item $a_{1,j}=0$ for $j<r/2$, and
\item $a_{0,j}=0$ for $j<r-1$.
\end{itemize}
If $a_{0,r-1}\neq0$, let $e$ be the edge from $(2,1)$ to $(0, r-1)$, and, if $a_{0,r-1}=0$, let $e$ be the edge from $(2,1)$ to $(1, r/2)$. In both cases, $f_e= a_{2,1} x^2 y + a_{1, r/2} x y^{r/2}+a_{0, r-1}y^{r-1}$. As we assume $f$ is Newton nondegenerate, $\frac{\partial f_e}{\partial x}$ and $\frac{\partial f_e}{\partial y}$ do not have a common zero in $\bigl(\overline{k}^{\times}\bigr)^2$. In particular, $\frac{\partial f_e}{\partial x} (x,1)$ and $\frac{\partial f_e}{\partial y}(x,1)$ do not have a common zero in~$\overline{k}^{\times}$. Moreover, since $a_{1,r/2}\neq0$ or $a_{0,r-1}\neq0$, we know $(0,1)$ is not a common zero of $\frac{\partial f_e}{\partial x}$ and $\frac{\partial f_e}{\partial y}$, so therefore $\frac{\partial f_e}{\partial x} (x,1)$ and $\frac{\partial f_e}{\partial y}(x,1)$ do not have a common zero in~$\overline{k}$. This means that the resultant (see \cite[Section~3.1]{CoxLittleOShea})
\[\Res\biggl(\frac{\partial f_e}{\partial x}(x,1), \frac{\partial f_e}{\partial y}(x,1); x\biggr)=(r-1)a_{2,1}(4a_{2,1}a_{0,r-1}-a_{1,r/2}^2)\]
is nonzero. In particular, this means that $R\coloneqq (r-1)(4a_{2,1}a_{0,r-1}-a_{1,r/2}^2)/2$ is nonzero. 

Row reduction of the equations in~\eqref{eq:system} (using only the assumption $R\neq 0$ and the conditions following~\eqref{eq:caseivassumptions} above) shows
\begin{itemize}
\item $\{ 1, x, y, y^2,\dots,y^{r-2} \}$ is a $k$-basis of $Q_0(\nabla f)$,
\item $x^2=\frac{1}{2a_{2,1}^2}\bigl(-2(r-1)a_{2,1}a_{0,r-1}+\frac{r}{2}a_{1,r/2}^2\bigr) y^{r-2}$,
\item $xy^j$ is a linear combination of $1, y, y^2,\dots,y^{r-2}$ for all $j$ satisfying $1\leq j\leq r/2-1$,
\item $xy^{r/2-1} =-\frac{a_{1,r/2}}{2a_{2,1}}y^{r-2}$
\item $\alpha=0$ in $Q_0(\nabla f)$ for all other $\alpha\in S$.
\end{itemize}
A direct calculation gives $\det H_f=rR y^{r-2}$ in $Q_0(\nabla f)$. We can define $\eta$ by $\eta(y^{r-2})= 1/R$ and $\eta(v)=0$ on all other basis elements~$v$. Then $\eta(\det H_f)=r$, the Milnor number. Using the ordering $\{ 1, y,\dots,y^{r/2-1}, x,y^{r/2},\dots,y^{r-2}\}$, the corresponding Gram matrix is the $r \times r$ block upper anti-triangular matrix
\[ \begin{bNiceArray}{ccc|c|ccc}
 & & & & & & 1/R\\
 & * & & * & & \iddots & \\
 & & & & 1/R & & \\
\hline
 & * & & P & & & \\
\hline
 & & 1/R & & & & \\
 & \iddots & & & & & \\
1/R & & & & & & \\
\end{bNiceArray},\]
with 
\[P=\begin{bmatrix}
1/R & -a_{1,r/2}/(2a_{2,1}R)\\
-a_{1,r/2}/(2a_{2,1}R) & a_{1,r/2}^2/(4a_{2,1}^2R)-1/(2a_{2,1}^2)
\end{bmatrix}.\]
Using Lemma~\ref{lemma:antidiagonal} $r/2-1$ times, we have $\mu_0^{\A1}(f)=\frac{r-2}{2}\hyp+ \beta_P$, where $\beta_P$ is the form corresponding to~$P$. Since $\beta_P$ represents $1/R$ and $\det P=-1/(2a_{2,1}^2R)$, we have 
\[\beta_P\cong \biggl\langle\frac{\det P}{1/R},\frac{1}{R}\biggr\rangle\cong\biggl\langle -\frac{1}{2a_{2,1}^2},\frac{1}{R}\biggr\rangle\cong\langle -2,R\rangle.\]
Therefore 
\[ \mu_0^{\A1}(f)= \frac{r-2}{2}\hyp + \langle -2,R\rangle.\] 
Symmetric bilinear forms of the form $\langle -2,R\rangle$ are no more restricted than those of the form $\langle -2,b\rangle$. To see this, let $a_{2,1}=1$, $a_{1,r/2}=0$, and $a_{0,r-1}=b/(4(r-1))$, in which case $R=b$. Therefore the elements of $\GW(k)$ found in this case are exactly those of the form $\frac{r-2}{2}\hyp+\langle -2,b\rangle$ for any $b\in k^{\times}$. 

\subsection{Classification}
As a result of the work above, we have the following classification theorem.

\begin{theorem}\label{thm:characterization}
Let $k$ be a field of characteristic not $2$, $3$, $5$, or~$7$. The elements of $\GW(k)$ of rank at most~$8$ that are $\A1$-Milnor numbers of polynomial maps $\A2_k\to\A1_k$ have the following description. 
\begin{center}
\begin{tabular}{cc}
\toprule
Rank & $\A1$-Milnor numbers\\
\midrule
1 & $\langle b\rangle$\\
2 & $\hyp$\\
3 & $\hyp+\langle b\rangle$\\
4 & $\hyp+\langle-A, AR \rangle$\\
5 & $2\hyp+\langle b\rangle$\\
6 & $2\hyp+\langle-2,b\rangle$\\
7 & $3\hyp+\langle b\rangle$\\
8 & $3\hyp+\langle -2,b\rangle$\\
\bottomrule
\end{tabular}
\end{center}
Here $\hyp$ is the hyperbolic form, $b\in k^{\times}$, and in the rank~$4$ case the constants $A$ and~$R$ are defined by $A=2a_{2,1}^2 - 6a_{3,0}a_{1,2}$ and
\[R=3(27a_{3,0}^2a_{0,3}^2 - 18a_{3,0}a_{2,1}a_{1,2}a_{0,3} + 4a_{2,1}^3a_{0,3} + 4a_{3,0}a_{1,2}^3 - a_{2,1}^2a_{1,2}^2),\]
where $a_{0,3}\neq0$ and either $a_{2,1}$ or~$a_{1,2}$ is nonzero, with the requirement that $A$ and~$R$ are both nonzero.
\end{theorem}

\begin{proof}
Theorem \ref{thm:allarise}(\ref{part:allariseeasyranks}) gives the results for ranks $1$, $2$, $3$, $5$, and~$7$, while ranks $4$, $6$, and~$8$ were considered in this section. Since $\langle-2,b\rangle\cong \hyp$ when $b=2$, the form $\frac{r}{2} \hyp$ found for Case~(i) in Section~\ref{subsec:casei} and Case~(ii) in Section~\ref{subsec:caseii} are among those found for Case~(iv) in Section~\ref{subsec:caseiv}. This completes the proof for ranks $6$ and~$8$.

For rank~$4$, the forms found in Section~\ref{subsec:caseiv} are among those found in Section~\ref{subsec:caseiii}. This can be seen by letting $a_{3,0}=0$, $a_{2,1}=1$, $a_{1,2}=0$, and $a_{0,3}=b/24$, in which case $A=2$ and $R=1$. The form $\hyp+\langle-A, AR \rangle$ thus becomes $\hyp+\langle -2,b\rangle$.
\end{proof}

\section{Over specific fields}\label{sec:specificfields}

The results of Section~\ref{sec:classifyA1Milnor} hold over any field of characteristic~$0$. In the simplest case, over~$\C$, the Grothendieck--Witt ring has only one element of each rank, and Example~\ref{ex:onevariable} shows that this is the $\A1$-Milnor number of a single-variable polynomial. Proposition~\ref{prop:increasen} then implies that, for all~$n$, each element of the Grothendieck--Witt ring arises as the $\A1$-Milnor number of a polynomial map $f\colon\A{n}_k \to \A1_k$. This section explores what happens for some other specific fields. 

\subsection{Over subfields of \texorpdfstring{$\R$}{R}}
Over subfields of~$\R$, the signature (see Section~\ref{subsec:GWring}) of a symmetric bilinear form places restrictions on the $\A1$-Milnor numbers.

\begin{theorem}\label{thm:possiblesignatures}
Let $k$ be a subfield of~$\R$, and let $r\in\{4,6,8\}$. If $\beta$ is an element of $\GW(k)$ that is represented by a symmetric bilinear form of rank~$r$ and is the $\A1$-Milnor number of a polynomial map $\A2_k\to\A1_k$, then $s_{\beta}\in\{0, -2\}$. 
\end{theorem}

\begin{proof}
We first prove the case $r=4$. In Theorem~\ref{thm:characterization}, we show that the elements in $\GW(k)$ of rank~$4$ that are $\A1$-Milnor numbers of polynomials in two variables are those represented by a form $\beta=\hyp+\langle -A,AR \rangle $, where $A=2a_{2,1}^2 - 6a_{3,0}a_{1,2}$ and 
\[R=3(27a_{3,0}^2a_{0,3}^2 - 18a_{3,0}a_{2,1}a_{1,2}a_{0,3} + 4a_{2,1}^3a_{0,3} + 4a_{3,0}a_{1,2}^3 - a_{2,1}^2a_{1,2}^2).\]
If $A$ is negative, then the discriminant of~$R$, considered as a quadratic polynomial in~$a_{0,3}$, is 
\[432(a_{2,1}^2 - 3a_{3,0}a_{1,2})^3= 18A^3<0.\]
Since the leading coefficient of $a_{0,3}^2$ in~$R$ is $81 a_{3,0}^2$, this would mean $R$ is positive for all choices of $a_{0,3}$. Hence $-A$ and~$R$ cannot both be positive. Thus $s_{\beta}=2-2=0$ or $1-3=-2$. 

For the case $r=6$, it was shown in Theorem~\ref{thm:characterization} that the elements in $\GW(k)$ of rank~$6$ that are $\A1$-Milnor numbers of polynomials in two variables are those represented by a form $\beta=2\hyp + \langle -2, b \rangle $, where $b\in k^{\times}$. If $b>0$, then $s_{\beta}=3-3=0$. If $b<0$, then $s_{\beta}=2-4=-2$.

For the case $r=8$, it was shown in Theorem~\ref{thm:characterization} that the elements in $\GW(k)$ of rank~$8$ that are $\A1$-Milnor numbers of polynomials in two variables are those represented by a form $\beta=2\hyp + \langle -2, b \rangle $, where $b\in k^{\times}$. If $b>0$, then $s_{\beta}=4-4=0$. If $b<0$, then $s_{\beta}=3-5=-2$.
\end{proof}

We can use Theorem~\ref{thm:possiblesignatures} to show that $\A1$-Milnor numbers of singularities determined by polynomials $\A2_k\to\A1_k$ are more restricted than local $\A1$-Brouwer degrees of polynomial maps $\A2_k\to\A2_k$ and $\A1$-Milnor numbers of polynomials $\A3_k\to\A1_k$. 

\begin{theorem}\label{thm:notallarise}
Let $k$ be a subfield of~$\R$. 
\begin{enumerate}[label=\textup{(\alph*)},ref=\textup{\alph*}]
\item There exist elements of $\GW(k)$ of ranks $4$, $6$, and~$8$ that are local $\A1$-Brouwer degrees of polynomial maps $\A2_k\to\A2_k$ but are not $\A1$-Milnor numbers of polynomial maps $\A2_k\to\mathbb{A}_k^1$.
\item There exist elements of $\GW(k)$ of ranks $4$, $6$, and~$8$ that are $\A1$-Milnor numbers of polynomial maps $\A{3}_k\to\mathbb{A}_k^1$ but are not $\A1$-Milnor numbers of polynomial maps $\A2_k\to\mathbb{A}_k^1$.
\end{enumerate}
\end{theorem}

\begin{proof}
By \cite[Lemma~2.6]{QuickStrandWilson}, every element of $\GW(k)$ of the form $m\hyp+\langle b,c\rangle$ is a local $\A1$-Brouwer degree of a polynomial map $\A2_k\to\A2_k$. By Example~\ref{ex:evenrankthreevariables}, every element of $\GW(k)$ of the form $m \hyp+\langle b,c\rangle$, where $m$ is a positive integer, is an $\A1$-Milnor number of a singularity determined by a polynomial $\A3_k\to \A1_k$. Letting $b,c>0$, this shows that there are elements $\beta\in \GW(k)$ with signature~$2$. However, by Theorem~\ref{thm:possiblesignatures}, forms $\beta\in \GW(k)$ with signature~$2$ cannot arise as $\A1$-Milnor numbers of polynomials~$f$ in two variables with $f(0)=0$ and $1\leq\mu_0(f)<\infty$. 
\end{proof}

In the converse direction, over~$\R$ itself, every form with the requisite signature arises as an $\A1$-Milnor number of a polynomial in two variables.

\begin{proposition}
Let $r\geq4$ be an even integer, and let $n\geq2$ be an integer. If an element of $\GW(\R)$ is represented by a symmetric bilinear form that has signature $0$ or~$-2$, then it is the $\A1$-Milnor number of a polynomial map $f\colon\A{n}_\R\to\A1_\R$ with $f(0)=0$ and $1\leq\mu_0(f)<\infty$.
\end{proposition}

\begin{proof}
If an element of $\GW(\R)$ has signature~$0$, then it is isometric to $\frac{r-2}{2}\hyp+\langle-1,1\rangle$. This is the $\A1$-Milnor number of the polynomial in two variables in Example~\ref{ex:evenranktwovariables} with $b=1$ and $c=-1$, and hence, by Proposition~\ref{prop:increasen}, as the $\A1$-Milnor number of a polynomial in $n$~variables. If an element of $\GW(\R)$ has signature~$-2$, then it is isometric to $\frac{r-2}{2}\hyp+\langle-1,-1\rangle$. This is the $\A1$-Milnor number of the polynomial in two variables in Example~\ref{ex:evenranktwovariables} with $b=c=1$, and hence, by Proposition~\ref{prop:increasen}, as the $\A1$-Milnor number of a polynomial in $n$~variables.
\end{proof}

\subsection{Over \texorpdfstring{$\Q$}{Q}}
We know from Section~\ref{sec:classifyA1Milnor} that the elements of the Grothendieck--Witt ring of rank~$6$ that are $\A1$-Milnor numbers of polynomials~$f$ in two variables with $f(0)=0$ and $1\leq\mu_0(f)<\infty$ are those of the form $2h+ \langle-2,b\rangle$ and the elements of the Grothendieck--Witt ring of rank~$8$ that are $\A1$-Milnor numbers of polynomials~$f$ in two variables with $f(0)=0$ and $1\leq\mu_0(f)<\infty$ are those of the form $3\hyp+ \langle-2,b\rangle $, where in both cases any $b\in k^{\times}$ is possible. We now work over~$\Q$ specifically and characterize these elements of $\GW(\Q)$ in terms of their invariants. 

We refer to Serre~\cite{Serre} for background on symmetric bilinear forms over~$\Q$ and the $p$-adic fields~$\Q_p$. (Serre's treatment is in terms of quadratic forms, which are equivalent to symmetric bilinear forms over fields of characteristic not~$2$.) The isometry class of a nondegenerate symmetric bilinear form~$\beta$ over~$\Q$ of rank~$r$ is determined by its rank~$r_{\beta}$, its signature~$s_{\beta}$, its discriminant $d_{\beta}\in \Q^{\times}/(\Q^{\times})^2$, and the Hasse--Witt invariants $\varepsilon_{p, \beta}=\pm 1$ for each prime~$p$, including $p=\infty$. These satisfy the following relations: 
\begin{enumerate}[label=\textup{(\alph*)},ref=\textup{\alph*}]
\item $-r\leq s_{\beta} \leq r$ and $s_{\beta}\equiv r\pmod{2}$;
\item $d_{\beta} = (-1)^{(r_{\beta}-s_{\beta})/2}$ in $\R^{\times}/(\R^{\times})^2$; 
\item $\varepsilon_{\infty,\beta} = (-1)^{(r_{\beta}-s_{\beta})(r_{\beta}-s_{\beta}-2)/8}$; 
\item if $d_{\beta}= -1$ in $\Q^{\times}_p/(\Q^{\times}_p)^2$, then $\varepsilon_{p, \beta}=1$; 
\item $\varepsilon_{p, \beta}=-1$ for only a finite number of $p$; and
\item $\prod_p \varepsilon_{p, \beta}=1$ (where the product includes $p=\infty$).
\end{enumerate}
Moreover, if these conditions are satisfied for numbers $r$, $s$, $d$, and~$\varepsilon_p$, then there is a nondegenerate symmetric bilinear form~$\beta$ over~$\Q$ with $r_{\beta}=r$, $s_{\beta}=s$, $d_{\beta}=d$, and~$\varepsilon_{p,\beta}=\varepsilon_p$. 

In the following, for a prime~$p$ and a nonzero integer~$n$, let $v_p(n)$ be the \emph{$p$-adic valuation} of~$n$, i.e.\ the largest integer~$j$ such that $p^j$ divides~$n$. We extend the domain of~$v_p$ to nonzero rational numbers by defining $v_p(b/c)= v_p(b)-v_p(c)$. We also let $(b,c)_p$ be the \emph{Hilbert symbol} as in \cite[Chapter~3]{Serre}.

\begin{theorem}\label{thm:invariants}
Let $m$ be a fixed nonnegative integer. The numbers $r$, $s$, $d$, and~$\varepsilon_p$ that are invariants of nondegenerate symmetric bilinear forms over~$\Q$ that are isometric to $m\hyp+\langle -2,b\rangle$ for some $b\in \Q^{\times}$ are precisely those that satisfy the following conditions: 
\begin{enumerate}[label=\textup{(\alph*)},ref=\textup{\alph*}]
\item $r=2m+2$;\label{item:rank}
\item $s\in\{0,-2\}$;\label{item:signature}
\item $d= (-1)^{m+1-s/2}$ in 
$\R^{\times}/(\R^{\times})^2$;\label{item:discriminant}
\item if $p$ is an odd prime, then \label{item:Hilbertp}
\[\vareps_p=\biggl(\frac{(-1)^{m+1}2}{p}\biggr)^{v_p(d)};\]
\item $\eps_{\infty}=(-1)^{(2m+2-s)(2m-s)/8}$; and \label{item:Hilbertinfinity}
\item $\prod_p \vareps_p=1$ \textup{(}where the product includes $p=\infty$\textup{)}. \label{item:Hilbertproduct} 
\end{enumerate}
\end{theorem}

\begin{proof}
We will show that the invariants of every form $m\hyp+\langle-2,b\rangle$ satisfy conditions (\ref{item:rank}) through~(\ref{item:Hilbertproduct}) and that any numbers satisfying these conditions correspond to a form $m\hyp+\langle-2,b\rangle$.

Define $\beta=m\hyp+\langle-2,b\rangle$, and let $r_{\beta}$, $s_{\beta}$, $d_{\beta}$, and~$\vareps_{p,\beta}$ be the corresponding invariants. It is immediate that $r_{\beta}=2m+2$, so~(\ref{item:rank}) holds. If $b>0$, then $s_{\beta}=0$, and if $b<0$, then $s_{\beta}=-2$. Therefore~(\ref{item:signature}) holds. Using~(\ref{item:rank}), the general relation $d_{\beta}=(-1)^{(r_\beta-s_\beta)/2}$ in $\R^{\times}/(\R^{\times})^2$ becomes $d_{\beta}=(-1)^{m+1-s_\beta/2}$, so~(\ref{item:discriminant}) also holds. Now let $p$ be an odd prime. Using the definition of $\vareps_{p, \beta}$ given in \cite[Chapter~IV, Section~2.1]{Serre}, and noting that $(1,b)_p=1$ for all~$b$, and $(-1,-1)_p=1$, we have 
\[\begin{split}
\vareps_{p, \beta}&=(-1, -2)_p^m\, (-1, b)_p^m\, (-2, b)_p.
\end{split}\]
When $m$ is even, using the fact that Hilbert symbols are always~$\pm1$ simplifies this to $\vareps_{p, \beta}=(-2,b)_p$. Using the formula for Hilbert symbols in \cite[Chapter~III, Theorem~1]{Serre} and the fact that $d_{\beta}=\pm2b$ means $v_p(b)=v_p(d_\beta)$, we then have
\[\vareps_{p, \beta}=(-2,b)_p=\biggl(\frac{-2}{p}\biggr)^{v_p(b)}=\biggl(\frac{-2}{p}\biggr)^{v_p(d_{\beta})}.\]

When $m$ is odd, using multiplicative properties of Hilbert symbols and Legendre symbols, we have
\[\begin{split}
\vareps_{p, \beta}&=(-1, -2)_p\, (-1, b)_p\, (-2, b)_p\\
&= (-1,2b)_p\,(-2,b)_p\\
&=\biggl( \frac{-1}{p}\biggr)^{v_p(2b)}\biggl( \frac{-2}{p}\biggr)^{v_p(b)} \\
&=\biggl( \frac{-1}{p}\biggr)^{v_p(d_{\beta})}\biggl( \frac{-2}{p}\biggr)^{v_p(d_{\beta})} \\
&=\biggl( \frac{2}{p}\biggr)^{v_p(d_{\beta})}.
\end{split}\]
The preceding two cases together give~(\ref{item:Hilbertp}). Using~(\ref{item:rank}), the general relation $\varepsilon_{\infty,\beta} = (-1)^{(r_{\beta}-s_{\beta})(r_{\beta}-s_{\beta}-2)/8}$ in $\R^{\times}/(\R^{\times})^2$ becomes $\eps_{\infty}= (-1)^{(2m+2-s_{\beta})(2m-s_{\beta})/8}$, showing~(\ref{item:Hilbertinfinity}). Finally, (\ref{item:Hilbertproduct}) holds for all forms, completing the proof that conditions (\ref{item:rank}) through~(\ref{item:Hilbertproduct}) hold for~$\beta$. 

Conversely, assume conditions (\ref{item:rank}) through~(\ref{item:Hilbertproduct}) are satisfied for some values $r$, $s$, $d$, and~$\vareps_p$, and define $\beta=\langle -2,(-1)^{m+1}2d\rangle$. We will show that $r_{\beta}=r$, $s_{\beta}=s$, $d_{\beta}= d$ in $\Q^{\times}/(\Q^{\times})^2$, and $\vareps_{p, \beta}=\vareps_p$. Using~(\ref{item:rank}), we have $r_{\beta}=2m+2=r$. If $d>0$, then (\ref{item:discriminant}) implies that $m-s/2$ is odd. If $d<0$, then (\ref{item:discriminant}) implies that $m-s/2$ is even. If $m$ is odd and $d>0$, then $(-1)^{m+1}2d>0$, so $s_\beta=0=s$. If $m$ is even and $d>0$, then $(-1)^{m+1}2d<0$, so $s_\beta=-2=s$. If $m$ is odd and $d<0$, then $(-1)^{m+1}2d<0$, so $s_\beta=-2=s$. If $m$ is even and $d<0$, then $(-1)^{m+1}2d>0$, so $s_\beta=0=s$. Thus we have $s_\beta=s$. Since $d_{\beta}=(-1)^{m+1}(-1)^{m+1}4d=4d$, we have $d_{\beta}= d$ in $\Q^{\times}/(\Q^{\times})^2$. For an odd prime~$p$, again using \cite[Chapter~III, Theorem~1]{Serre} we have 
\[\vareps_{p,\beta}=\begin{cases}
\Bigl(\frac{2}{p}\Bigr)^{v_p(d)}&\text{if $m$ is odd,}\\
\Bigl(\frac{-2}{p}\Bigr)^{v_p(d)}&\text{if $m$ is even.}
\end{cases}\]
Together with condition~(\ref{item:Hilbertp}), this gives $\vareps_{p, \beta}=\vareps_p$ for all odd primes~$p$. Using~(\ref{item:rank}), we have $\vareps_{\infty,\beta}=(-1)^{(2m+2-s)(2m-s)/8}= \vareps_{\infty}$, showing~(\ref{item:Hilbertinfinity}). Finally,
\[\vareps_{2,\beta}=\prod_{p\neq 2} \vareps_{p,\beta}=\prod_{p\neq 2} \vareps_p=\vareps_2,\]
where the last equality follows from part~(\ref{item:Hilbertproduct}). 
\end{proof}

\begin{corollary}\label{cor:rank6invariants}
The numbers $s$, $d$, and~$\vareps_p$ that are invariants of symmetric bilinear forms that represent elements of $\GW(\Q)$ of rank~$6$ that are $\A1$-Milnor numbers of polynomial maps $\A2_\Q\to\A1_\Q$ with $f(0)=0$ and $1\leq\mu_0(f)<\infty$ are precisely those that satisfy the following conditions: 
\begin{enumerate}[label=\textup{(\alph*)},ref=\textup{\alph*}]\item $s\in\{0,-2\}$; \label{item:signature6}
\item $d= (-1)^{1-s/2}$ in 
$\R^{\times}/(\R^{\times})^2$; \label{item:discriminant6}
\item if $p$ is an odd prime, then
\[\eps_p=\begin{cases}
1&\text{if $p\equiv1,3\pmod{8}$,}\\
(-1)^{v_p(d)}&\text{if $p\equiv5,7\pmod{8}$};
\end{cases}\]\label{item:Hilbertp6}
\item $\vareps_{\infty} =(-1)^{1-s/2}$; and \label{item:Hilbertinfinity6}
\item $\prod_p \vareps_p=1$ \textup{(}where the product includes $p=\infty$\textup{)}. \label{item:Hilbertproduct6}
\end{enumerate}
\end{corollary}

\begin{proof}
We showed in Theorem~\ref{thm:characterization} that the elements of $\GW(\Q)$ of rank~$6$ that arise as $\A1$-Milnor numbers of two-variable polynomials are precisely those of the form $2\hyp+\langle-2,b\rangle$. Therefore we just need to apply Theorem~\ref{thm:invariants} with $m=2$. This gives $d=(-1)^{3-s/2}$ in $\R^{\times}/(\R^{\times})^2$, which is equivalent to~(\ref{item:discriminant6}). It also gives 
\[\eps_p=\biggl(\frac{-2}{p}\biggr)^{v_p(d)}.\]
Since
\[\biggl(\frac{-2}{p}\biggr)=\begin{cases}
1&\text{if } p\equiv 1, 3 \pmod{8},\\ 
-1 &\text{if } p\equiv 5, 7\pmod{8},
\end{cases}\]
this is equivalent to~(\ref{item:Hilbertp6}). Finally, it yields $\eps_\infty=(-1)^{(6-s)(8-s)/8}$, which is equivalent to~(\ref{item:Hilbertinfinity6}) in the cases $s=0$ and $s=-2$. 
\end{proof}

\begin{corollary}\label{cor:rank8invariants}
The numbers $d$, $s$, and~$\vareps_p$ that are invariants of symmetric bilinear forms that represent elements of $\GW(\Q)$ of rank~$8$ that are $\A1$-Milnor numbers of polynomial maps $\A2_\Q\to\A1_\Q$ with $f(0)=0$ and $1\leq\mu_0(f)<\infty$ are precisely those that satisfy the following conditions: 
\begin{enumerate}[label=\textup{(\alph*)},ref=\textup{\alph*}]
\item $s\in\{0,-2\}$; \label{item:signature8}
\item $d= (-1)^{-s/2}$ in $\R^{\times}/(\R^{\times})^2$; \label{item:discriminant8}
\item if $p$ is an odd prime, then
\[\eps_p=\begin{cases}
1&\text{if $p\equiv1,7\pmod{8}$,}\\
(-1)^{v_p(d)}&\text{if $p\equiv3,5\pmod{8}$};
\end{cases}\]\label{item:Hilbertp8}
\item $\vareps_{\infty} =1$; and \label{item:Hilbertinfinity8}
\item $\prod_p \vareps_p=1$ \textup{(}where the product includes $p=\infty$\textup{)}. \label{item:Hilbertproduct8}
\end{enumerate}
\end{corollary}

\begin{proof}
We showed in Theorem~\ref{thm:characterization} that the elements of $\GW(\Q)$ of rank~$8$ that arise as $\A1$-Milnor numbers of two-variable polynomials are precisely those of the form $3\hyp+\langle-2,b\rangle$. Therefore we just need to apply Theorem~\ref{thm:invariants} with $m=3$. This gives $d=(-1)^{4-s/2}$ in $\R^{\times}/(\R^{\times})^2$, which is equivalent to~(\ref{item:discriminant8}). It also gives 
\[\eps_p=\biggl(\frac{2}{p}\biggr)^{v_p(d)}.\]
Since
\[\biggl(\frac{2}{p}\biggr)=\begin{cases}
1&\text{if } p\equiv 1, 7 \pmod{8},\\ 
-1 &\text{if } p\equiv 3, 5\pmod{8},
\end{cases}\]
this is equivalent to~(\ref{item:Hilbertp8}). Finally, it yields $\eps_\infty=(-1)^{(8-s)(6-s)/8}$, which equals~$1$ when $s=0$ or $s=-2$, thus giving~(\ref{item:Hilbertinfinity8}). 
\end{proof}

\section{A dimension reduction technique}\label{sec:variablereduction}

In Corollary~\ref{cor:addterm}, we saw how to add a new variable to a polynomial while only changing the $\A1$-Milnor number by multiplication by a unit. In this section, we show that, under certain hypotheses, one can find a new polynomial with one fewer variable while only changing the $\A1$-Milnor number by multiplication by a unit. A result of this type for local $\A1$-Brouwer degrees was previously shown by Quick, Strand and Wilson~\cite{QuickStrandWilson}, and we first recall their result.

\begin{theorem}[{\cite[Theorem~2.12]{QuickStrandWilson}}]\label{thm:QSWvariablereduction}
Let $k$ be a field of characteristic not~$2$. Suppose $F\colon\mathbb{A}_k^n\to\mathbb{A}_k^n$ is a map with an isolated zero at the origin and that the rank of $Q_0(F)$ is at least~$1$. If $(F_1,\dots,F_n)\nsubseteq(x_1,\dots,x_n)$, then there is a map $G\colon\mathbb{A}_k^{n-1}\to\mathbb{A}_k^{n-1}$ with isolated zero at the origin and $Q_0(G)\cong Q_0(F)$ and where the EKL forms of $F$ and~$G$ differ only by multiplication by a unit.
\end{theorem}

As Quick, Strand, and Wilson point out in \cite[Remark~2.13]{QuickStrandWilson}, this type of reduction was introduced by McKean \cite[Lemma~5.7]{McKean} when calculating the local $\A1$-Brouwer degree of the intersection of two plane curves. In both \cite{McKean} and~\cite{QuickStrandWilson}, the reduction finds a function with fewer variables, an isomorphic local ring, and where the form in question differs only by multiplication by a unit. 

For $\A1$-Milnor numbers, one finds the local $\A1$-Brouwer degree of $F=\nabla f$, but applying Theorem~\ref{thm:QSWvariablereduction} to~$F$ yields a~$G$ that is not necessarily the gradient of some~$g$. Despite this complication, we are able to prove a reduction theorem for $\A1$-Milnor numbers in Theorem~\ref{thm:variablereduction} as an analogue to Theorem~\ref{thm:QSWvariablereduction}.

\begin{definition}
Let $R$ be a local ring with maximal ideal~$\m$. If $f\in \m$, the \emph{order} of~$f$, denoted $\order(f)$, is the largest integer~$\ell$ such that $f\in \m^{\ell}$. By convention, the zero polynomial has order~$\infty$.
\end{definition}

On the way to proving Theorem~\ref{thm:variablereduction}, we first show that, if a polynomial has a $y^2$ term, then we can modify the polynomial to increase the order of the coefficient of~$y$ while preserving the $\A1$-Milnor number. The following example illustrates this approach. 

\begin{example}\label{ex:running4}
Let $k$ be a field of characteristic not $2$ or $5$. Let $f=x^4 + 2 x^2y + y^2 + x^3y+x^3y^{2}$. Considering $f$ as a polynomial in~$y$, we have $f = x^4 + (2 x^2 + x^3) y + y^2 + x^3y^{2}$, where the coefficient of $y$ has order~$2$. Completing the square, we have
\[\begin{split}
f & = \biggl(y + \frac{2x^2+x^3}{2}\biggr)^2 - \biggl(\frac{2x^2+x^3}{2}\biggr)^2 + x^4 + x^3y^{2}.
\end{split}\]
We then perform a change of variables to obtain a polynomial $g\in k[x,Y]$. Precisely, let
\[\begin{split}
g(x,Y) &= f\biggl(x, Y- \frac{2x^2+x^3}{2}\biggr)\\
&= Y^2 - \biggl(\frac{2x^2+x^3}{2}\biggr)^2 + x^4 + x^3 \biggl( Y- \frac{2x^2+x^3}{2} \biggr)^{2}\\
&= x^4-\biggl(\frac{2x^2+x^3}{2}\biggr)^2 + x^3 \biggl(\frac{2x^2+x^3}{2}\biggr)^{2} -x^3(2x^2+x^3) Y + Y^2 + x^3Y^2. 
\end{split}\]
In this case, $g$ is Newton nondegenerate as $\Gamma(g)$ consists of the edge from $(0,2)$ to $(5,0)$ and $\charac k\notin\{ 2, 5\}$. Then, by Theorem~\ref{thm:kushnirenko}, $\mu_0(g)=\nu(g)=4$. Moreover, $\mu_0^{\A1}(f)=\mu_0^{\A1}(g)$ by Theorem~\ref{thm:rightequivalence}. The change of variables increased the order of the coefficient of~$y$ (now~$Y$) to~$5$. Since the order of $x^3(2x^2+x^3) Y$ is $6$, which is at least $\mu_0(g)+2$, it can be removed from~$g$ to obtain a new polynomial $h$ with the same 
$\A1$-Milnor number as~$g$ by Lemma~\ref{lemma:ne}. One then has
\begin{equation} 
\label{eqn:hY}
h(x,Y) = x^4-\biggl(\frac{2x^2+x^3}{2}\biggr)^2 + x^3 \biggl(\frac{2x^2+x^3}{2}\biggr)^2 + Y^2 + x^3Y^2. 
\end{equation}
Since $h$ lacks a $Y$~term and $h_Y=Y(2+2x^3)$, the variable~$Y$ will be zero in the local ring $Q_0(\nabla{h})$. This allows us to eliminate the variable~$Y$ from~$h$ as demonstrated in Example~\ref{ex:running5}. 
\end{example}

The following lemma shows that the approach in the above example works more generally to transform $f$, considered as a polynomial in~$x_n$, into another polynomial for which the coefficient of~$x_n$ has larger order. Repeatedly applying this lemma will then make the order of the coefficient of~$x_n$ sufficiently large that it can be ignored according to Lemma~\ref{lemma:ne}.

\begin{lemma}\label{lemma:increaseorder}
Let $k$ be a field of characteristic not~$2$, let $n,q\geq2$ be integers, and let $a\in k^{\times}$. Let $f\in k[x_1,\dots, x_{n}]$ be a polynomial of the form 
\[ f = \alpha_0 + \alpha_1 x_n + ax_n^2+\sum_{i=2}^q \alpha_i \cdot x_n^i,\]
where $\alpha_i \in k[x_1,\dots, x_{n-1}]$ are polynomials for $i\in\{0,1,\dots,q\}$ with $\alpha_1\neq 0$ and $\alpha_0(0)=\alpha_1(0)=\alpha_2(0)=0$. If $1\leq\mu_0(f)<\infty$, then there exists a polynomial $g\in k[x_1,\dots, x_{n}]$ of the form
\[ g= \beta_0 + \beta_1 x_n + ax_n^2+\sum_{i=2}^q \beta_i \cdot x_n^i, \]
where $\beta_i \in k[x_1,\dots, x_{n-1}]$ for $i\in\{0,1,\dots,q\}$ with $\beta_0(0)=\beta_1(0)=\beta_2(0)=0$ and $\order(\beta_1) > \order(\alpha_1)$ such that $Q_0(\nabla f) \cong Q_0(\nabla g)$ and $\mu^{\mathbb{A}^1}_0 ( f ) = \mu^{\mathbb{A}^1}_0 (g)$. 
\end{lemma}

\begin{proof}
By completing the square, letting $x_n=X_n-\alpha_1/(2a)$, and expanding, we have
\[\begin{split}
f&=\alpha_0 + \alpha_1 x_n + ax_n^2+\sum_{i=2}^q \alpha_i \cdot x_n^i\\ & =a\Bigl(x_n+\frac{\alpha_1}{2a}\Bigr)^2-\frac{\alpha_1^2}{4a}+\alpha_0+\sum_{i=2}^q \alpha_i \cdot x_n^i\\
& =aX_n^2+\alpha_0-\frac{\alpha_1^2}{4a}+\sum_{i=2}^q \alpha_i \cdot \Bigl(X_n-\frac{\alpha_1}{2a}\Bigr)^i\\
& =aX_n^2+\alpha_0-\frac{\alpha_1^2}{4a}+\sum_{i=2}^q\alpha_i\cdot\Bigl(-\frac{\alpha_1}{2a}\Bigr)^i+\sum_{i=2}^qi\alpha_i\cdot\Bigl(-\frac{\alpha_1}{2a}\Bigr)^{i-1}X_n\\
& \qquad+\sum_{j=2}^q\sum_{i=j}^q\binom{i}{j}\alpha_i\cdot\Bigl(-\frac{\alpha_1}{2a}\Bigr)^{i-j}X_n^j.
\end{split}\]
This is of the form
\[\beta_0+\beta_1x+aX_n^2+\sum_{j=2}^q \beta_j \cdot X_n^j\]
with
\[\beta_j=\begin{cases}
\displaystyle \alpha_0-\frac{\alpha_1^2}{4a}+\sum_{i=2}^q\alpha_i\cdot\Bigl(-\frac{\alpha_1}{2a}\Bigr)^i & \text{for $j=0$,}\\
\displaystyle \sum_{i=2}^qi\alpha_i\cdot\Bigl(-\frac{\alpha_1}{2a}\Bigr)^{i-1} & \text{for $j=1$,}\\
\displaystyle \sum_{i=j}^q\binom{i}{j}\alpha_i\cdot\Bigl(-\frac{\alpha_1}{2a}\Bigr)^{i-j} & \text{for $j\geq2$.}
\end{cases}\]
Note that $\alpha_0(0)=\alpha_1(0)=0$ implies $\beta_0(0)=0$, that $\alpha_1(0)=0$ implies $\beta_1(0)=0$, and that $\alpha_1(0)=\alpha_2(0)=0$ implies $\beta_2(0)=0$. 

We now show that $\order(\beta_1) >\order(\alpha_1)$ by inspecting each term in~$\beta_1$. The hypotheses $\alpha_1(0)=0$ and $\alpha_1\neq 0$ give that $1\leq\order(\alpha_1)<\infty$. For $i=2$, the hypothesis $\alpha_2(0)=0$ gives $\order(\alpha_2)>0$, so 
\[\order(\alpha_2\cdot \alpha_1) = \order(\alpha_2) + \order(\alpha_1)>\order(\alpha_1).\]
For $i\geq3$, we have $\order(\alpha_i\cdot \alpha_1^{i-1})\geq \order(\alpha_1^{i-1}) > \order(\alpha_1)$. Since $\order (\alpha_i\cdot \alpha_1^{i-1})>\order(\alpha_1)$ for all $i\geq2$ and $\beta_1=\sum_{i=2}^qi(-1/(2a))^{i-1}\alpha_i\cdot\alpha_1^{i-1}$, we have $\order(\beta_1) >\order(\alpha_1)$. 

Defining
\[ g= \beta_0 + \beta_1 x_n + ax_n^2+\sum_{i=2}^q \beta_i \cdot x_n^i, \]
Lemma~\ref{lemma:localringisom} and Theorem~\ref{thm:rightequivalence} give that $Q_0(\nabla f) \cong Q_0(\nabla g)$ and $\mu^{\mathbb{A}^1}_0 ( f ) = \mu^{\mathbb{A}^1}_0 (g)$.
\end{proof}

In general by repeatedly applying Lemma~\ref{lemma:increaseorder} we can make the order of~$\alpha_1$ sufficiently large that the term $\alpha_1\cdot x_n$ can be deleted, as the next lemma makes precise. 

\begin{lemma}\label{lemma:removelinearterm}
Let $k$ be a field of characteristic not~$2$, and let $f\colon\mathbb A^n_k \to \mathbb A_k^1$ be a polynomial map with $f(0)=0$ and $1\leq\mu_0(f)<\infty$. If $(f_{x_1}, \dots, f_{x_n})\not\subseteq (x_1, \dots, x_{n})^2$, then there exists a map $g\colon \mathbb A^{n}_k \to \mathbb A_k^1$ with $g(0)=0$ and $1\leq\mu_0(g)<\infty$ of the form 
\[g=\alpha_0+ax_n^2+\sum_{i=2}^q \alpha_i\cdot x_n^i,\]
for $a\in k^{\times}$ and $\alpha_i\in k[x_1, \dots, x_{n-1}]$, where $\order(\alpha_2)\geq1$, $Q_0(\nabla f)\cong Q_0(\nabla g)$, and $\mu^{\mathbb{A}^1}_0(f)=\mu^{\mathbb{A}^1}_0(g)$.
\end{lemma}

\begin{proof}
Since $\mu_0(f)>0$, the partial derivatives $f_{x_1}, \dots, f_{x_n}$ have no constant terms. Since $(f_{x_1}, \dots, f_{x_n})\not\subseteq (x_1, \dots, x_{n})^2$, one of $f_{x_1}, \dots, f_{x_n}$ has a nonzero term of degree~$1$, so $f$ has a nonzero term of degree~$2$. Such a term must be of the form $ax_i^2$ or $ax_ix_j$. If $f$ does not contain a nonzero term $ax_i^2$ for some~$i$, then it must have a nonzero term $ax_ix_j$ for some $i\neq j$. In this case, by using a change of variables that replaces $x_j$ by $x_i+x_j$, we obtain a polynomial~$h_1$ that has a nonzero term $ax_i^2$, $Q_0(\nabla f)\cong Q_0(\nabla h_1)$, and $\mu^{\mathbb{A}^1}_0(f)=\mu^{\mathbb{A}^1}_0(h_1)$ by Theorem~\ref{thm:rightequivalence}. Replacing $f$ by~$h_1$, we can assume $f$ contains a term $ax_i^2$ for some~$i$ and $a\in k^{\times}$. Reindexing the variables, we can assume $f$ contains the term $a x^2_n$ for some $a\in k^{\times}$. Thus, we may assume 
\[f=\alpha_0+\alpha_1\cdot x_n+a x_n^2+\sum_{i=2}^q \alpha_i\cdot x_n^i\]
for some polynomials $\alpha_i\in k[x_1,\dots, x_{n-1}]$ with $\alpha_1(0)=\alpha_2(0)=0$ and $a\in k^{\times}$. If $n=1$, then $\alpha_1=0$, and we can let $g=f$ to obtain the result. The same holds if $n>1$ and $\alpha_1=0$. If $n>1$ and $\alpha_1\neq 0$, then the polynomial~$f$ satisfies the hypotheses of Lemma~\ref{lemma:increaseorder}. By repeatedly applying Lemma~\ref{lemma:increaseorder}, we can find a polynomial $h_2\in k[x_1, \dots, x_n]$ such that
\[h_2=\beta_0+\beta_1\cdot x_n+a x_n^2+\sum_{i=2}^q \beta_i\cdot x_n^i,\]
with $\beta_i\in k[x_1, \dots, x_{n-1}]$, $\order(\beta_1)\geq \mu_0(f) + 1$, $\beta_1(0)=\beta_2(0)=0$, $a\in k^{\times}$, $Q_0(\nabla f)\cong Q_0(\nabla h_2)$, and $\mu_0^{\A1}(f) = \mu_0^{\A1}(h_2)$. Then, as $\order( \beta_1\cdot x_n)\geq \mu_0(h_2)+2$, Lemma~\ref{lemma:ne} allows us to delete the term $\beta_1\cdot x_n$ from~$h_2$ without changing its local algebra or $\A1$-Milnor number. Since $\order(\beta_2)\geq 1$ as $\beta_2(0)=0$, the polynomial $g=\beta_0+ax_n^2+\sum_{i=2}^q \beta_i\cdot x_n^i$ fulfills the conditions of the theorem. 
\end{proof}

We are now ready to prove the main result of this section.

\begin{theorem}\label{thm:variablereduction}
Let $k$ be a field of characteristic not~$2$. Let $f\colon\mathbb A_k^n \to \mathbb A_k^1$ be a polynomial map with $f(0)=0$ and $1\leq\mu_0(f)<\infty$, and let $\mathfrak{m}=(x_1, \dots, x_{n})$. If $(f_{x_1}, \dots, f_{x_n})\nsubseteq \mathfrak{m}^2$, then there exists a map $g\colon \mathbb A_k^{n-1} \to \mathbb A_k^1$ with $g(0)=0$ and $1\leq\mu_0(g)<\infty$ such that 
\begin{enumerate}[label=\textup{(\alph*)},ref=\textup{\alph*}]\item $Q_0(\nabla f)\cong Q_0(\nabla g)$,\label{item:variablereductionlocalalgebra}
\item $\mu_0(f)=\mu_0(g)$, and \label{item:variablereductionMilnor}
\item $\mu^{\mathbb{A}^1}_0(f)$ and $\mu^{\mathbb{A}^1}_0(g)$ differ only by multiplication by a unit.\label{item:variablereductionA1Milnor} 
\end{enumerate}
\end{theorem}

\begin{proof}
By Lemma~\ref{lemma:removelinearterm}, we can assume $f$ has the form 
\[\beta_0+ax_n^2+\sum_{i=2}^q \beta_i\cdot x_n^i,\]
with $\beta_i\in k[x_1, \dots, x_{n-1}]$, $\order(\beta_2)\geq1$, and $a\in k^{\times}$. Let $g=\beta_0$. Differentiating, we have 
\[\begin{aligned}
f_{x_j} & =g_{x_j}+\sum_{i=2}^q \frac{\partial \beta_{i}}{\partial x_j} x_n^i\text{ for } j\in\{1,\dots,n-1\},\\
f_{x_n} & =2a x_n+\sum_{i=2}^q i\beta_i\cdot x_n^{i-1}=x_n\biggl(2a+\sum_{i=2}^{q} i\beta_i\cdot x_n^{i-2}\biggr).
\end{aligned}\]
As ideals in the local ring ring $k[x_1, \dots, x_n]_{\mathfrak{m}}$, we have
\[\begin{split}
(f_{x_1},\dots,f_{x_n}) & =\biggl(f_{x_1},\dots,f_{x_{n-1}},x_n\biggl(2a+\sum_{i=2}^{q} i\beta_i\cdot x_n^{i-2}\biggr)\biggr)\\
& =(f_{x_1},\dots,f_{x_{n-1}},x_n)\\
& = (g_{x_1},\dots,g_{x_{n-1}},x_n),\\
\end{split}\]
where the second equality holds because $2a+\sum_{i=2}^{q} i\beta_i\cdot x_n^{i-2}$ is a unit in $k[x_1, \dots, x_n]_{\mathfrak{m}}$ and the last equality holds because $f_{x_i}$ and $g_{x_i}$ are equal up to a multiple of~$x_n$ for all~$i$. We then have 
\[\begin{split}
Q_0(\nabla f) & =k[x_1,\dots,x_n]_{\mathfrak{m}}/(f_{x_1},\dots,f_{x_n})\\
& = k[x_1,\dots,x_n]_{\mathfrak{m}}/(g_{x_1},\dots,g_{x_{n-1}},x_n)\\
& \cong k[x_1,\dots,x_{n-1}]_{\mathfrak{m}'}/(g_{x_1},\dots,g_{x_{n-1}})\\
& = Q_0(\nabla g),
\end{split}\]
where $\mathfrak{m}'$ is the ideal $(x_1,\dots,x_{n-1})$ in $k[x_1,\dots,x_{n-1}]$, and we identify elements that are in both $Q_0(\nabla f)$ and $Q_0(\nabla g)$ as equal. This proves that $g$ satisfies~(\ref{item:variablereductionlocalalgebra}), which implies that it also satisfies~(\ref{item:variablereductionMilnor}).

For $j\in[n-1]$, if we have $g_{x_j}=\sum_{\ell=1}^{n-1} g_{j,\ell}\cdot x_\ell$, then we can write $f_{x_j}=\sum_{\ell=1}^{n} f_{j,\ell}\cdot x_\ell$ with $f_{j,\ell}=g_{j,\ell}$ for all $\ell\in[n-1]$ and $f_{j,n}=\sum_{i=2}^r \frac{\partial \beta_{i}}{\partial x_j}x_n^{i-1}$. We can also write $f_{x_n}=\sum_{\ell=1}^n f_{n,\ell}\cdot x_\ell$ with $f_{n,\ell}=0$ for all $\ell\in[n-1]$ and $f_{n,n}=2a+\sum_{i=2}^{q} i\beta_i\cdot x_n^{i-2}$. Since $\beta_2(0)=0$, we have that $f_{n,n}=2a$ in $Q_0(\nabla f)$. This means that 
\[\begin{split}
E_0(\nabla f)&=\begin{vmatrix}
g_{1,1}&\cdots &g_{1,n-1} & \sum_{i=2}^r \frac{\partial \beta_{i}}{\partial x_1}x_n^{i-1}\\
\vdots&\ddots&\vdots&\vdots\\
g_{n-1,1}&\cdots&g_{n-1,n-1}&\sum_{i=2}^r \frac{\partial \beta_{i}}{\partial x_{n-1}}x_n^{i-1}\\
0&\cdots&0&2a
\end{vmatrix}\\
&=2a\begin{vmatrix}
g_{1,1}&\cdots &g_{1,n-1}\\
\vdots&\ddots&\vdots\\
g_{n-1,1}&\cdots&g_{n-1,n-1}
\end{vmatrix}\\
&=2aE_0(\nabla g).
\end{split}\]

Let $\eta_g$ satisfy $\eta_g(E_0(\nabla g)) = 1$, and define $\eta_f(v) = \frac{1}{2a} \eta_g(v)$ for all $v \in Q_0(\nabla f)$. Then 
\[\eta_f( E_0(\nabla f) ) = \frac{1}{2a} \eta_g( 2a E_0(\nabla g)) = \eta_g( E_0(\nabla g)) = 1.\]
We can now compute the $\A1$-Milnor numbers of $f$ and~$g$ by using $\eta_f$ and~$\eta_g$, respectively, on the same set of basis elements. Each entry in the Gram matrix of $\mu_0^{\A1}(f)$ is exactly $1/(2a)$ times the corresponding entry of the Gram matrix of $\mu_0^{\A1}(g)$, so the $\A1$-Milnor numbers of $f$ and~$g$ are thus related by 
\[\mu^{\mathbb{A}^{1}}_0(f)=\Bigl\langle \frac{1}{2a} \Bigr\rangle \cdot\mu^{\mathbb{A}^{1}}_0(g).\]
Since $\langle 1/(2a)\rangle$ is a unit, this gives~(\ref{item:variablereductionA1Milnor}).
\end{proof}

\begin{example}\label{ex:running5}
We illustrate Theorem~\ref{thm:variablereduction} with the polynomial $f=(y+x^2)^2+x^3y+x^3y^2$. In Example~\ref{ex:running4}, $f$ was shown to have the same $\A1$-Milnor number as the polynomial $h(x,Y)$ in \eqref{eqn:hY}. The proof of Theorem~\ref{thm:variablereduction} produces the single-variable polynomial 
\[g=x^4-\frac{(2x^2+x^3)^2}{4}+\frac{(2x^2+x^3)^2}{4}x^3\]
with $Q_0(\nabla g) \cong Q_0(\nabla h)\cong Q_0(\nabla f)$. In this example, $g$ is obtained from $h=g + Y^2 + x^3 Y^2$, so $a=1$ in the proof of Theorem~\ref{thm:variablereduction}. The proof then gives $\mu^{\mathbb{A}^1}_0(f)= \langle1/2\rangle\cdot\mu^{\mathbb{A}^1}_0(g)$. Although we excluded $\charac k=5$ in Example~\ref{ex:running4}, a similar calculation when $\charac k=5$ gives the same~$g$.
\end{example}

Theorem~\ref{thm:variablereduction} is related to Corollary~\ref{cor:addterm}. Whereas Corollary~\ref{cor:addterm} finds the $\A1$-Milnor number when a new variable is added, Theorem~\ref{thm:variablereduction} shows that in some cases a variable can be eliminated while only changing the $\A1$-Milnor number by multiplication by a unit. In particular, the proof of Theorem~\ref{thm:variablereduction} gives $\mu_0^{\A1}(g+z^2) = \langle2\rangle\cdot\mu_0^{\A1}(g)$, which agrees with Corollary~\ref{cor:addterm}. However, $\mu_0^{\A1}(g+z^3) = \mu_0^{\A1}(g) \mu_0^{\A1}(z^3)$ by Proposition~\ref{prop:product} and thus is not a unit multiple of $\mu_0^{\A1}(g)$. Hence, a hypothesis such as $(f_{x_1}, \dots, f_{x_n})\nsubseteq \mathfrak{m}^2$ is needed in Theorem~\ref{thm:variablereduction}.

\section*{Acknowledgments}

We would like to thank Kyle Ormsby for introducing us to this topic in the 2021 Park City Mathematics Institute (PCMI) Undergraduate Faculty Program and Marc Levine, Stephen McKean, Jan Stevens, and Kirsten Wickelgren for helpful comments and suggestions. Much of this paper results from work supported by four programs. We would like to express our thanks to the 2023 Collaborate@ICERM program (NSF grant No.\ DMS-1929284) at the Institute for Computational and Experimental Research in Mathematics in Providence, RI, the 2024 PCMI Undergraduate Faculty Program, and the 2025 Summer Research in Mathematics program (NSF Grant No.\ DMS-1928930) and the 2026 Lasting Alliance Through Team Immersion and Collaborative Exploration (LATTICE) program (supported by the American Mathematical Society and the Sloan Foundation Grant G-2025-25245), both at the Simons Laufer Mathematical Sciences Institute in Berkeley, California. Tom Hagedorn received support from a TCNJ research award. Joel Louwsma was partially supported by two Niagara University Summer Research Awards. 

\bibliographystyle{amsplain}
\bibliography{A1MilnorNumbersSmallRank}

\end{document}